\documentclass[10pt]{article}

\usepackage[english]{babel}
\usepackage{amsmath,amssymb,amsthm,bbm}
\usepackage[abbrev]{amsrefs}
\usepackage{enumerate}
\usepackage[pdftex]{graphicx,color}
\usepackage{hyperref}
\usepackage{listings}
\usepackage{verbatim}

\numberwithin{equation}{section}
\usepackage[title, toc]{appendix}
\theoremstyle{plain}
\newtheorem{theorem}{Theorem}[section]
\newtheorem{proposition}[theorem]{Proposition}
\newtheorem{lemma}[theorem]{Lemma}
\newtheorem{corollary}[theorem]{Corollary}

\theoremstyle{remark}
\newtheorem{definition}[theorem]{Definition}

\newtheorem{remark}[theorem]{Remark}

\DeclareMathOperator{\Var}{Var}

\newcommand{\bR}{\mathbb{R}}
\newcommand{\bN}{\mathbb{N}}

\newcommand{\bE}{\mathbb{E}}
\newcommand{\bP}{\mathbb{P}}

\newcommand{\cF}{\mathcal{F}}

\newcommand{\HS}{\mathrm{HS}}

\newcommand{\SG}{\mathrm{SG}}
\newcommand{\tY}{\tilde{Y}}
\newcommand{\tR}{\tilde{R}}

\newcommand{\tB}{\tilde{B}}

\newcommand{\hR}{\hat{R}}

\newcommand{\hB}{\hat{B}}
\newcommand{\hC}{\hat{C}}

\newcommand{\bfY}{\mathbf{Y}}
\newcommand{\bfX}{\mathbf{X}}
\newcommand{\bfa}{\mathbf{a}}
\newcommand{\bfb}{\mathbf{b}}

\newcommand{\brR}{\bar{R}}
\newcommand{\ve}{\varepsilon}
\newcommand{\g}{\gamma}
\newcommand{\la}{\lambda}
\newcommand{\bs}{\bar{\sigma}_n}
\newcommand{\sn}{{\sigma}_n}

\renewcommand{\a}{\alpha}
\renewcommand{\b}{\beta}
\renewcommand{\d}{\delta}
\newcommand{\w}{\omega}
\newcommand{\e}{\varepsilon}

\newcommand{\PP}{\mathbb{P}}
\newcommand{\EE}{\mathbb{E}}
\newcommand{\RR}{\mathbb{R}}
\newcommand{\CC}{\mathbb{C}}
\newcommand{\cD}{\mathcal{D}}

\newcommand{\bfafront}{\mathbf{a}^{\text{(l)}}}
\newcommand{\bfaback}{\mathbf{a}^{\text{(f)}}}

\newcommand{\beq}{ \begin{equation} }
\newcommand{\eeq}{ \end{equation} }
\newcommand{\beqq}{ \begin{equation*} }
\newcommand{\eeqq}{ \end{equation*} }

\newcommand{\D}{\Delta_n}

\begin{document}

	\title{An edge CLT for the log determinant of Laguerre beta ensembles}
	\author{Elizabeth W. Collins-Woodfin\footnote{Department of Mathematics \& Statistics, McGill University,
Montreal, QC, H3A 0G4, Canada \newline email: \texttt{elizabeth.collins-woodfin@mail.mcgill.ca}} \and Han Gia Le\footnote{Department of Mathematics, University of Michigan,
Ann Arbor, MI, 48109, USA \newline email: \texttt{hanle@umich.edu}}}
	\date{\today}

	\maketitle

\begin{abstract}
We obtain a CLT for $\log|\det(M_n-s_n)|$ where $M_n$ is a scaled Laguerre beta ensemble and $s_n=d_++\sigma_n n^{-2/3}$ with $d_+$ denoting the upper edge of the limiting spectrum of $M_n$ and $\sigma_n$ a slowly growing function ($\log\log^2 n\ll\sigma_n\ll\log^2 n$).  In the special cases of LUE and LOE, we prove that the CLT also holds for $\sigma_n$ of constant order.  A similar result was proved for Wigner matrices by Johnstone, Klochkov, Onatski, and Pavlyshyn.  Obtaining this type of CLT of Laguerre matrices is of interest for statistical testing of critically spiked sample covariance matrices as well as free energy of bipartite spherical spin glasses at critical temperature.
\end{abstract}

	\section{Introduction}
	\subsection{Background}
	As one of the most fundamental quantities in the study of matrices, determinants have been well studied in random matrix theory and there is a natural interest in how these determinants behave asymptotically as the size of the matrix grows.  More specifically, a number of studies of the past decade have studied the log determinant, $\log|\det(M_n)|$, for various random matrix ensembles, $M_n$, and have established CLT results for this quantity as $n\to\infty$.  See papers by Nguyen and Vu for results non-Hermitian i.i.d. matrices \cite{Nguyen_2014} and Tao and Vu for results on Wigner matrices \cite{TaoVu2012}.
	
	It is also of interest to study a log determinant away from the origin (i.e. $\log|\det(M_n-s)|$ for $s\ne0$).  We note that this quantity can also be written as $\sum_{i=1}^n\log|\la_i-s|$ where $\{\la_i\}_{i=1}^n$ are the eigenvalues of $M_n$.  For $s$ outside the spectrum of $M_n$, this is a special case of the well-studied linear spectral statistics, i.e. $\sum_{i=1}^n f(\lambda_i)$ where $f$ is a smooth function on the support of the spectrum of $M_n$.  Johansson proved a CLT for linear spectral statistics of Gaussian beta ensembles (with some generalization to other random matrices) \cite{Johansson_linstat} and Bai and Silverstein proved a similar result for Laguerre beta ensembles \cite{BaiSilverstein}.
	
	Recently, Johnstone, Klochkov, Onatski, and Pavlyshyn \cite{JKOP1} considered a case in which $M_n$ is a scaled Wigner ensemble (or Gaussian beta ensemble) and $s$ is close to edge of the spectrum of $M_n$ and approaches the edge as $n\to\infty$.  This is not covered by the studies of linear statistics, since $\sum_{i=1}^n\log|\lambda_i-s|$ is singular for $s$ at the edge of the spectrum.  This work was motivated by high dimensional statistical testing and spin glasses.  Johnstone et al derived a CLT for this case (see also a related result by Lambert and Paquette \cite{lambertpaquette}). The goal of this paper is to derive an analogous result to \cite{JKOP1} in the case where the matrix is from a Laguerre beta ensemble.
	\vspace{0.1in}
	
	\noindent\textbf{Laguerre beta ensembles:}
	By Laguerre beta ensemble (L$\b$E), we mean an $n\times n$ random matrix $M_{n,m}$ with joint eigenvalue density
\beq
p(\lambda_1,\lambda_2,...,\lambda_n)=C_{n,m,\beta}\prod_{i<j}|\lambda_i-\lambda_j|^\beta \prod_{i=1}^n \lambda_i^{\frac\beta2 (m-n+1)-1}e^{-\lambda_i/2},
\eeq
where $m\geq n$ and $\beta>0$ and $C_{n,m,\b}$ is the corresponding normalization constant. The cases of $\b=1$ and $\b=2$ correspond to the Laguerre Orthogonal Ensemble (LOE) and the Laguerre Unitary Ensemble (LUE) respectively, which can be constructed by setting $M_{n}:=AA^*$ where $A$ is taken to be an $n\times m$ matrix with i.i.d. entries that are real Gaussian (LOE) or complex Gaussian (LUE) with mean 0 and variance 1.  We fix a parameter $\lambda$ and take $n,m\to\infty$ such that their ratio converges to $\lambda$.  More specifically, we require 
\beq
\frac{n}{m}=\lambda+O(n^{-1}),\quad 0<\lambda\leq1.
\eeq
	Let $\mu_1\geq\mu_2\geq\cdots\mu_n\geq0$ denote the eigenvalues of the scaled L$\b$E matrix $\frac1m M_{n,m}$.  It was shown by Mar{\v{c}}enko and Pastur (for $\b=1$) \cite{Marcenko_1967} and by Dumitriu and Edelman (for general $\b>0$) \cite{DumitriuEdelman} that, as $n,m\to\infty$ with $n/m\to\la\leq1$,
	\beq\label{eq:def_MP}
	\frac1n \sum_{i=1}^n\delta_{\mu_i}\to \frac{\sqrt{(d_+-x)(x-d_-)}}{2\pi\la x}\mathbf{1}_{[d_-,d_+]},
	\eeq
	where the convergence is weakly in distribution and $d_\pm=(1\pm\la^{1/2})^2$.  
	
	Of particular importance for our purposes is the behavior of the largest eigenvalue.  As $n\to\infty$, this eigenvalue approaches the constant $d_+$ and displays Tracy-Widom type fluctuations of order $n^{2/3}$ about $d_+$ (see \cite{Ram_rez_2011} for the  general $\b$ case):
	\beq
	C_{\lambda,\beta} (\mu_1-d_+)n^{2/3} \to  TW_\beta,
	\eeq
	where the arrow denotes convergence in distribution, $d_+$ is as defined above, $C_{\lambda,\beta}$ is a constant, and $TW_\b$ is the $\b$ version of the Tracy-Widom distribution.
	
		\vspace{0.1in}
	
	\noindent\textbf{Motivation and recent related research:}
	In this paper we derive a CLT for the log determinant of L$\b$E matrices near the edge of the spectrum.  More precisely, we study $\log|\det(M_{n,m}/m-\gamma)|$ where $\gamma:=d_++\sigma_n n^{-2/3}$ for $\sigma_n$ satisfying $-\tau < \sigma_n\ll (\log n)^2$ for some fixed $\tau>0$.  The motivation for this research question is two-fold, with applications in both statistics and spin glasses.
	
	In high dimensional statistics, there is much interest in hypothesis testing for spiked models, i.e. matrices of the form $M_n+h\mathbf{x}\mathbf{x}^*$ where $M_n$ is a random matrix, $h$ is a scalar, and $\mathbf{x}$ is a vector giving the direction of the spike (see, e.g. \cite{JohnstoneOnatskiSpiked}).  Laguerre beta ensembles are of particular interest in this context because of their connection to sample covariance matrices.  The log determinant near the edge of the spectrum is useful in detecting the presence of a spike when $h$ is small. Johnstone et al derive a CLT similar to ours for Gaussian beta ensembles (G$\b$E), which they also extend to Wigner ensembles with certain moment restrictions \cite{JKOP1}.  They used G$\b$E as a proxy for L$\b$E because they behave similarly but are less messy to analyze.  Our paper confirms that, indeed, the CLT of the log determinant near the spectral edge of a L$\b$E matrix closely resembles that of a Wigner matrix, up to differences in the values of certain constants in the CLT formulas.  Furthermore, in calculating these constants, we are able to make explicit the dependence of the CLT formula for L$\b$E on the parameter $\lambda$.
	

Gaussian beta ensembles were also studied in this context by Lambert and Paquette \cite{lambertpaquette}, but via a different method.  They prove that a rescaled version of the characteristic polynomial converges to a random function that can be characterized as a solution to the Stochastic Airy equation.  From this convergence result, they obtain the CLT for the log determinant near the edge as a corollary.
	
In addition to the statistical motivation, this paper relates to questions of interest in spin glasses. Johnstone et al \cite{JKOP2} and Landon \cite{Landon_crit} observe that the quantity $\log|\det(M_n-s)|$ (with $M_n$ being a scaled GOE matrix) appears in the calculations of the free energy of the spherical Sherrington-Kirkpatrick (SSK) spin glass model.  Baik and Lee \cite{BaikLeeSSK} showed in 2016 that the asymptotic fluctuations of the SSK free energy are Gaussian at high temperature but Tracy-Widom at low temperature.  However, the nature of the free energy fluctuations near the critical temperature remained an open question, requiring a more detailed analysis of $\log|\det(M_n-s)|$ in the case where $s$ is near the spectral edge.  The papers \cite{JKOP2},\cite{Landon_crit} analyze this critical case.  Building on the findings of \cite{JKOP1} and \cite{lambertpaquette}, they provide a free energy formula for SSK near the critical temperature that interpolates between the high temperature and low temperature cases.
	
	Just as the edge CLT for the log determinant of GOE was needed to analyze the free energy of SSK at critical temperature, our result for Laguerre ensembles provides a necessary piece of information for the analysis of bipartite spherical spin glasses.  As with the SSK model, the free energy of bipartite spherical spin glasses exhibits Gaussian fluctuations at high temperature and Tracy-Widom fluctuations at low temperature \cite{BaikLeeBipartite}.  Our paper provides a key tool to analyze the critical temperature setting, which we will address in a subsequent paper.

	
	\subsection{Main result}
	Our contribution consists of two related Central Limit Theorems.  Theorem \ref{thm:mainresult} holds for general Laguerre beta ensembles and provides a CLT for the log determinant evaluated at a distance of $\sigma_n n^{-2/3}$ above the spectral edge where $\sigma_n$ is a slowly growing function (e.g. $\log n$).  Theorem \ref{thm:mainresult2} extends this CLT all the way to the spectral edge in the cases of LUE and LOE.
	
	\begin{theorem}[CLT slightly away from the edge]\label{thm:mainresult}
	Let $M_{n,m}$ be a L$\b$E matrix where $n\leq m$ and $n/m=\la+O(n^{-1})$ as $n,m\to\infty$ for some $0<\la\leq1$.  Define $\a=2/\b$. Let $\cD_n=\det(M_{n,m}/m-\gamma)$ where $\gamma=d_++\sigma_n n^{-2/3}$ with $(\log\log n)^2\ll \sigma_n\ll (\log n)^2$ and $d_+$ denotes the upper edge of the limiting spectral distribution of $\frac1m M_{n,m}$.  Then,
	\beq
	\frac{\log|\cD_n|-C_\la n - \frac{1}{\la^{1/2}(1+\la^{1/2})}\sigma_n n^{1/3}
+\frac{2}{3\la^{3/4}(1+\la^{1/2})^2}\sigma_n^{3/2} +\frac16\left(\a-1\right) \log n}{\sqrt{\frac\a3 \log n}}\to \mathcal{N}(0,1),
	\eeq
	where
	\beq
	C_\la:=(1-\la^{-1})\log(1+\la^{1/2})+\log(\la^{1/2})+\la^{-1/2}.
	\eeq
	\end{theorem}
	
	\begin{theorem}[CLT at the edge]\label{thm:mainresult2}
	In the case where $M_{n,m}$ is from LUE or LOE ($\a=1$ or $\a=2$, respectively), the CLT in Theorem \ref{thm:mainresult} can be extended to hold for any $\sigma_n$ satisfying $-\tau<\sigma_n\ll(\log n)^2$ for some fixed $\tau>0$.
	\end{theorem}
	
	The majority of this paper is devoted to the proof of Theorem \ref{thm:mainresult}, after which the extension to Theorem \ref{thm:mainresult2} is accomplished in Section \ref{sec:extension_to_edge}. Our proof of Theorem \ref{thm:mainresult} is largely inspired by the proof of Theorem 2 in \cite{JKOP1}. As shown by Dumitriu and Edelman \cite{DumitriuEdelman}, the eigenvalue distribution of a G$\b$E matrix is the same as that of a symmetric tridiagonal matrix. The key component of the proof of paper \cite{JKOP1} is an analysis of a recurrence relations on the minors of the tridiagonal matrix. The recurrence relation is nonlinear with random coefficients. Johnstone et al were able to replace the nonlinear recurrence with a linear one with good error control and derived a CLT from the linear recurrence.
	
	
	
	For L$\b$E, the tridiagonal matrix representation is formed as a product of a bi-diagonal matrix and its transpose \cite{DumitriuEdelman}. Similar to the proof of \cite{JKOP1}, we use this representation to arrive at a nonlinear recurrence relation, which we approximate by a linear one. However, unlike in the Gaussian case, our tridiagonal matrix has dependence between adjacent entries and the diagonal entries are not identically distributed. The more intricate structure of the matrix and the additional parameter $\lambda$ make the analysis of the recurrence significantly more technical.  We outline the details of our proof of Theorem \ref{thm:mainresult} in Subsection \ref{subsec:setup} after the set-up.
	
	
	As in \cite{JKOP1}, the extension of Theorem \ref{thm:mainresult} to Theorem \ref{thm:mainresult2} is first done in the case $\b=2$, relying on determinantal structures \cite{GotzeTikhomirov}, then it is obtained for $\b=1$ using the inter-relationship between eigenvalues of unitary and orthogonal ensembles \cite{ForresterRains}. However, there is some subtlety in our case due to the singularity of the Mar{\v{c}}enko--Pastur measure in the case $\lambda=1$.

	
	\subsection{Organization of this paper and remarks on notations}
	The rest of the paper is organized as follows. Section \ref{sec:setup} introduces key quantities, discusses sub-gamma random variables and concentration inequalities associated with them. In Section \ref{sec:A.5}, we provide an asymptotic expression for the log determinant in terms of log of a rescaled determinant and a deterministic shift. In Section \ref{sec:A.4}, we analyze a linear approximation of this log of the rescaled determinant. A CLT for the linear approximation is derived in Section \ref{sec:A.1}. Error incurred from the linear approximation is shown to be negligible in Section \ref{sec:unif}.  Taken together, Sections \ref{sec:setup}-\ref{sec:unif} complete the proof of Theorem \ref{thm:mainresult}.  The extension of Theorem \ref{thm:mainresult} to Theorem \ref{thm:mainresult2} is proved in Section \ref{sec:extension_to_edge}.
	The \hyperref[sec:technical_lemmas]{Appendix} contains proofs of some technical asymptotic estimates. 
	
	\subsection*{Acknowledgement}
	We would like to thank Jinho Baik for his advice and insights throughout our work on this project. The work of the second author was supported in parts by NSF grant DMS-1954790.

	\section{Set-up and preliminary lemmas}\label{sec:setup}
	\subsection{Remarks on notation}
	We use several asymptotic notations throughout this paper and define our conventions here.  Given a sequence $\{a_n\}$ and a positive sequence $\{b_n\}$, we write:
\begin{itemize}
\item $a_n=O(b_n)$ if there exists some constant $C$ such that $|a_n|\leq Cb_n$ for all $n$,
\item $a_n=\Omega(b_n)$ if there exists some constant $C$ such that $|a_n|\geq Cb_n$ for all $n$,
\item $a_n=\Theta(b_n)$ if there exist constants $C_1,C_2$ such that $C_1b_n\leq |a_n|\leq C_2b_n$ for all $n$\\ (or, equivalently, $a_n=O(b_n)$ and $a_n=\Omega(b_n)$),
\item $a_n\ll b_n$ if $\lim_{n\to\infty} a_n/b_n=0$,
\item $a_n\gg b_n$ if $\lim_{n\to\infty} b_n/a_n=0$.
\end{itemize}
	\begin{remark}
		Throughout the paper, we use $C$, $C_1$, $C_2$, or $c$, $c_1$, $c_2$ in order to denote constants that are independent of $N$. Even if the constant is different from one place to another, we may use the same notation $C$, $C_1$, $C_2$, or $c$, $c_1$, $c_2$ as long as it does not depend on $N$ for the convenience of the presentation.
	\end{remark}
	
	\begin{remark}
		Throughout the paper, we omit including $\lfloor \;\rfloor$ and/or $\lceil\; \rceil$ for floor and ceiling functions whenever a quantity that is seemingly not integer-valued is used as  an integer. Instead, we implicitly apply floor function in all such cases. For example, $\sum_{i=n^{1/3}}^{n^{2/3}}$ represents a sum over $i \in \{\lfloor n^{1/3} \rfloor,\lfloor n^{1/3} \rfloor+1, \dots,\lfloor n^{2/3} \rfloor-1, \lfloor n^{2/3} \rfloor \}$. 
	\end{remark}
	
	\begin{remark}
	At various points throughout the paper, we replace $n/m$ with $\la$ without writing the $O(n^{-1})$ term to avoid cumbersome notation. This does not affect the computations as in all cases, the $O(n^{-1})$ term is small and gets absorbed into other error terms in the final approximation.
	\end{remark}

	\subsection{Set-up}\label{subsec:setup}
	As shown in \cite{DumitriuEdelman}, the eigenvalue distribution of a L$\b$E matrix $M_{n,m}$ is the same as that of the $n\times n$ matrix $T_n=BB^T$ where $B$ is a bi-diagonal matrix of dimension $n\times n$.  More specifically, 
	\beq
B=
\begin{bmatrix}
a_1&&&&\\
b_1&a_2&&&\\
&b_2&a_3&&\\
&&\ddots&\ddots&\\
&&&b_{n-1}&a_n
\end{bmatrix}
\quad\text{so}\quad
BB^T=
\begin{bmatrix}
a_1^2&a_1b_1&&&\\
a_1b_1&a_2^2+b_1^2&a_2b_2&&\\
&a_2b_2&a_3^2+b_2^2&&\\
&&&\ddots&a_{n-1}b_{n-1}\\
&&&a_{n-1}b_{n-1}&a_n^2+b_{n-1}^2
\end{bmatrix}
\eeq
where the quantities $\{a_i\},\{b_i\}$ are all independent random variables with distributions satisfying
\beq
a_i^2\sim\frac\a2 \chi^2\left(\frac2\a (m-n+i)\right),\qquad
b_i^2\sim\frac\a2 \chi^2\left(\frac2\a i\right).
\eeq
We observe that, while the entries of $B$ are pairwise independent, $T$ has dependence between adjacent entries.  This is different from what occurs in the tridiagonalization of GOE/GUE matrices and it makes certain aspects of our computations more intricate than what is required in the Gaussian case.
	
	We will find it useful to deal with a centered and rescaled version of the variables $\{a_i\}$ and $\{b_i\}$, so we introduce the notation
	\beq
	d_i=\frac{a_i^2-(m-n+i)}{\sqrt{m-n+i}},\qquad
	c_i=\frac{b_i^2-i}{\sqrt{i}},
	\eeq
	where $\{d_i\}$ and $\{c_i\}$ all have mean 0 and variance $\a$.
	Our goal is to study the quantity
	\beq
	D_n:=\det(T-\gamma m)
	\eeq
	for $\gamma$ as defined in the introduction. Let $D_i$ be the determinant of the upper left $i\times i$ minor of the matrix $T-\gamma m$. Then the determinants satisfy the recursion
	\beq
	D_i=(a_i^2+b_{i-1}^2-\gamma m)D_{i-1}-a_{i-1}^2b_{i-1}^2D_{i-2}
	\eeq
	and, using our centered rescaled variables,
	\beq\label{eq:D_recursion}\begin{split}
D_i=&(d_i\sqrt{m-n+i}+m-n+i+c_{i-1}\sqrt{i-1}+i-1-\gamma m)D_{i-1}\\
&-(d_{i-1}\sqrt{m-n+i-1}+m-n+i-1)(c_{i-1}\sqrt{i-1}+i-1)D_{i-2}.
\end{split}\eeq
We remark that the deterministic analog of this recursion is given by
\beq
D_i'=(m-n+2i-1-\gamma m)D_{i-1}'
-(m-n+i-1)(i-1)D_{i-2}',
\eeq
which has characteristic roots
\beq\label{eq:rho_def}
\rho_i^\pm=-\frac12\left(\gamma m-(m-n+2i-1)
\pm\sqrt{(\gamma m-(m-n+2i-1))^2-4(m-n+i-1)(i-1)}
\right).
\eeq
Observe that the roots $\rho_i^+$ and $\rho_i^-$ are both negative for all $i$.    Their positive versions, $|\rho_i^+|$ and $|\rho_i^-|$, will be important throughout our analysis. To control the growth of $D_i$, we introduce a normalized version of the recursion, following the approach used by Johnstone et al in the Gaussian case \cite{JKOP1}.  In particular, we define
\beq\label{defn:Ei}
E_i:=\frac{D_i}{\prod_{j=1}^i|\rho_j^+|}
\eeq
and obtain the recursion
\beq\label{eq:recursionEi}\begin{split}
E_i=&\frac{d_i\sqrt{m-n+i}+m-n+i+c_{i-1}\sqrt{i-1}+i-1-\gamma m}{|\rho^+_i|} E_{i-1}\\
&-\frac{(d_{i-1}\sqrt{m-n+i-1}+m-n+i-1)(c_{i-1}\sqrt{i-1}+i-1)}{|\rho^+_i| |\rho^+_{i-1}|}E_{i-2}.
\end{split}\eeq
We simplify this expression as
\beq
E_i=\left(\a_i+\b_i+\tau_i+\d_i-\frac{\gamma m}{|\rho_i^+|}\right)E_{i-1}-(\a_{i-1}+\tau_{i-1})(\b_i+\d_i)E_{i-2},
\eeq
where
\beq\label{defn:abtau}
\a_i=\frac{d_i\sqrt{m-n+i}}{|\rho_i^+|},
\qquad\b_i=\frac{c_{i-1}\sqrt{i-1}}{|\rho_i^+|},
\qquad\tau_i=\frac{m-n+i}{|\rho_i^+|},
\qquad\d_i=\frac{i-1}{|\rho_i^+|}.
\eeq
We note that $\tau_i$ and $\d_i$ are deterministic while $\a_i$ and $\b_i$ are centered random variables with variance $\a\tau_{i}/|\rho_i^+|$ and $\a\d_i/|\rho_i^+|$ respectively.

In the subsequent sections of this paper, we obtain a CLT for $E_n$ and deduces a CLT for our original determinant. Our general approach, modeled after the methods in \cite{JKOP1}, is to approximate the recursion for $E_i$ by a linear recursion. The authors \cite{JKOP1} observe that, in their setting, the ratio $E_i/E_{i-1}$ is close to $-1$ for all $i$ when $n$ is large. This observation holds in our setting as well (we note that $E_i$ is non-zero since it is the rescaled characteristic polynomial of a minor of $M_{n,m}/m$, evaluated at a point that is outside of the spectrum).  Therefore, we define the quantity
\beq\label{defn:Ri}
R_i:=1+\frac{E_i}{E_{i-1}},
\eeq
and show is close to zero.  Dividing the recursion \eqref{eq:recursionEi} by $E_{i-1}$ and rearranging terms, we obtain
\beq
R_i=\left(\a_i+\b_i+\tau_i+\d_i+1-\tfrac{\gamma m}{|\rho_i^+|}\right)+(\a_{i-1}+\tau_{i-1})(\b_i+\d_i)\frac{1}{1-R_{i-1}}.
\eeq
To obtain our linear approximation of the recursion, we make the following observations:
\begin{itemize}
\item $\frac{1}{1-R_{i-1}}=1+\frac{R_{i-1}}{1-R_{i-1}}=1+R_{i-1}+\frac{R_{i-1}^2}{1-R_{i-1}}$,
\item For any $i$, we have $\a_i,\b_i,R_i\to0$ as $m,n\to\infty$.  This is easy to see for $\a_i,\b_i$ and not immediately obvious for $R_i$, but we prove it later in the paper.
\end{itemize}
Using these observations, we rewrite the recursion for $R_i$ as 
\beq\label{eqn:Rirecursion}
R_i=\xi_i+\omega_iR_{i-1}+\e_i,
\eeq
where
\begin{align}
\xi_i&=\a_i+\b_i(1+\tau_{i-1})+\a_{i-1}\d_i, \label{defn:xii}\\
\omega_i&=\tau_{i-1}\d_i,\\
\e_i&=-(\g_i-\w_i) +\a_{i-1}\beta_i  +(\a_{i-1}\b_i+\a_{i-1}\d_i+\tau_{i-1}\b_i)\tfrac{R_{i-1}}{1-R_{i-1}}
+\tau_{i-1}\d_i\tfrac{R_{i-1}^2}{1-R_{i-1}}.\label{defn:wi_ei}
\end{align}
and $\g_i = \frac{|\rho_i^-|}{|\rho_i^+|}$ for $3 \leq i \leq n$.

We note that $\{\xi_i\}$ are mean-zero random variables while $\{\omega_i\}$ are deterministic and we will prove that $\{\e_i\}$ are small.  Thus, we can define a recursion on a new sequence of variables $L_i$, which we will show are a good approximation of $R_i$.  We define $L_i$ to satisfy
\beq
L_i:=\xi_i+\omega_i L_{i-1} \text{for}i\geq4,\qquad L_3:=\xi_3.
\eeq
From this recursive definition,
\beq
L_j=\sum_{i=3}^{j-1}\xi_i\omega_{i+1}\omega_{i+2}\cdots\omega_j + \xi_j, \quad  \text{for } j\geq4.
\eeq
It is important (in showing CLT) to express $L_j$ as a sum of independent random variables, yet we have dependence between consecutive terms in the sequence $\{\xi_i\}$. To address this issue, we expand $\xi_i$ using \eqref{defn:xii} to have
	\begin{equation}\label{eqn:Lj_full}
		\begin{split}
			L_j 
			&=  \sum_{i=3}^{j-1} \omega_{i+1}\dots \omega_j X_i + X_j + \alpha_j -\omega_3 \dots \omega_j \alpha_2,
		\end{split}
	\end{equation}
	where 
	\begin{equation}\label{defn:Xi}
		X_i = (1 + \tau_{i-1})(\delta_i \alpha_{i-1} + \beta_i), \quad 3 \leq i \leq n.
	\end{equation}
Note that, unlike $\xi_i$, the variables $X_i$ are pairwise independent.  In later calculations, it is more convenient to work with $Y_i$ rather than with $L_i$, where $Y_i$ is given by
\beq\label{defn:Yi}
Y_i = \sum_{j = 3}^{i-1} \w_{j+1}\dots\w_i X_j+X_i, \quad 3 \leq i \leq n.
\eeq

With this set-up, our proof of Theorem \ref{thm:mainresult} consists of the following key steps:
\begin{enumerate}
	\item First, we write the log determinant of $T_n -\gamma m$ in terms of log of the rescaled quantity $|E_n|$, asymptotically as $n$ goes to infinity. 
	\item We then show that in the regime $(\log\log n)^2\ll\sigma_n \ll (\log n)^2$, with probability $1-O(n^{-1})$, both $\max_i|L_i|$ and $\max_i|R_i|$ are $o(n^{-1/3})$. Thus Taylor's approximation for logarithm is applied to obtain 
	\[\log|E_n| = \sum_{i=3}^n\log|1-R_i|+\log|E_2| =\sum_{i=3}^n(-R_i-R_i^2/2)+ o(1),\]
	with probability $1-O(n^{-1})$.
	\item With probability $1-O(n^{-1})$, we have $\sum_{i=3}^n(-R_i-R_i^2/2)$ is $-\sum_{i=3}^n L_i$ plus a deterministic shift, up to an error of order $\sqrt{\log n}$.
	\item Lastly, we show $-\sum_{i=3}^n L_i$ has variance of exact order $\log n$, and satisfies Lyapunov's CLT. 
\end{enumerate}
While this general outline has close resemblance to that of the Gaussian case \cite{JKOP1}, each step involves more technical treatment due to the complicated structure of the recurrence relations. Before proceeding with these steps, we examine properties of the quantities introduced in this section.

\subsection{Properties of sub-gamma random variables}
It is central in our analysis that error due to linear approximation and similar reductions are negligible. In most instances, these error terms appear as sum of independent random variables that behave similarly to sub-gaussian random variables, known as \textit{sub-gamma} families.
\begin{definition}\label{def:SG}
	For $v,u>0$, a real-valued centered random variable $X$ is said to belong to a sub-gamma family $\SG(v,u)$ if for all $t \in \bR$ such that $|t|< \frac{1}{u}$, 
	\beq 
	\EE e^{tX} \leq \exp\left(\frac{t^2 v}{2 (1 - tu)} \right). 
	\eeq
\end{definition}
The following properties of sub-gamma random variables are useful for our analysis.
\begin{itemize}
	\item If $X \sim \chi^2(d) - d$, then $X \in \SG(2d,2)$
	\item Given a real number $c$ and $X \in \SG(v_X,u_X)$, $cX \in \SG(c^2v_X,|c|u_X)$
	\item If $X \in \SG(v_X,u_X)$ and $Y \in \SG(v_Y, u_Y)$ are independent, then $X+Y \in \SG(v_X+v_Y, u_X \vee u_Y)$
\end{itemize}
We verify that for $i=3,\dots,n$, the random variables $\a_i$ and $\b_i$ as defined in \eqref{defn:abtau}, and their linear combination $X_i$ belong to sub-gamma families. 
\begin{lemma}\label{lem:SG}
	For $i = 3,\dots, n$,
	\beqq
	\a_i\in \SG\left(\frac{\a \tau_i}{|\rho_i^+|},\frac{\a}{|\rho_i^+|}\right), \qquad
	\b_i\in\SG\left(\frac{\a\delta_i}{|\rho_i^+|},\frac{\a}{|\rho_i^+|}\right), \qquad
	X_i 
	\in \SG(v_i,u_i),
	\eeqq
	where
	\beq 
	v_i = \frac{\a \delta_i}{|\rho_i^+|} (\w_i+1)(1+\tau_{i-1})^2, \quad u_i = \frac{\a(1+\tau_{i-1})}{|\rho_i^+|}.
	\eeq
\end{lemma}
In the subsequent sections, both characterizations of sub-gamma random variables in terms of tail probabilities, and in terms of $p$-norms for $p\geq 1$ are used. In particular, we regularly apply the following result.
\begin{lemma}\label{lem:SG_boucheron}(see Theorem 2.3 of \cite{boucheron2013concentration})
		
If $X$ belongs to $\SG(v,u)$, then for every $t>0$,
\beq
\PP(|X| > \sqrt{2vt}+ut)\leq 2e^{-t}. 
\eeq
In addition, for every integer $p\geq 2$,
\beq 
\|X\|_{p}^{p} =\EE [X^{p}] \leq (p/2)! (8v)^{p/2} + p! (4u)^{p}. 
\eeq
\end{lemma}
\subsection{Preliminary lemmas concerning the values of \texorpdfstring{$\rho_i^+$}{rho+}, \texorpdfstring{$\rho_i^-$}{rho-}, and \texorpdfstring{$\w_i$}{wi}}
We begin by observing that  $|\rho_i^+|$ is a decreasing function of $i$ and $|\rho_i^-|$ is an increasing function of $i$.  Other key properties are captured in the following lemma.

\begin{lemma}\label{lem:rho_properties} The quantities $|\rho_i^+|$ and $|\rho_i^-|$ satisfy the following asymptotic bounds, uniformly in $i$:
\begin{enumerate}[(i)]
\item\label{lem:rho_properties_rho+bound} $|\rho_i^+|=\Theta(n)$,
\item\label{lem:rho_properties_rhodiffbound} $|\rho_i^+|-|\rho_i^-|=\Omega(n^{2/3}\sigma_n^{1/2})$,
\item\label{lem:rho_properties_anotherdiffbound} 
$|\rho_i^-|-|\rho_{i-1}^-|=O(n^{1/3}\sigma_n^{-1/2})$\quad and \quad $|\rho_{i-1}^+|-|\rho_{i}^+|=O(n^{1/3}\sigma_n^{-1/2})$,
\item\label{lem:rho_properties_rhoratiocompare} $\frac{|\rho_{i}^-|}{|\rho_{i}^+|}-\frac{|\rho_{i-1}^-|}{|\rho_{i-1}^+|}=O(n^{-2/3}\sigma_n^{-1/2})$.
\end{enumerate}
\end{lemma}
\begin{proof}
To show \eqref{lem:rho_properties_rho+bound}, for the lower bound, we have
\beq\begin{split}
|\rho_i^+|\geq|\rho_n^+|
&>\tfrac12\left(\gamma m-(m+n-1)\right)
=\tfrac12\left(2\sqrt{mn}+\lambda^{-1}\sigma_n n^{1/3}+1\right)=\Omega(n).
\end{split}\eeq
For the upper bound, we have 
\beq\begin{split}
|\rho_i^+|\leq|\rho_1^+|
&=\gamma m-(m+n-1)
=2\sqrt{mn}+2n+\la^{-1}\sigma_n n^{1/3}-1=O(n).
\end{split}\eeq
For \eqref{lem:rho_properties_rhodiffbound}, we have
\beq\begin{split}
|\rho_i^+|-|\rho_i^-|>|\rho_n^+|-|\rho_n^-|
&=\sqrt{2\lambda^{-3/2}\sigma_n n^{4/3}+O(n)}
=\Omega(n^{2/3}\sigma_n^{1/2}).
\end{split}\eeq
For \eqref{lem:rho_properties_anotherdiffbound}, it suffices to show that $|\rho_i^-|-|\rho_{i-1}^-|+|\rho_{i-1}^+|-|\rho_i^+|=O(n^{1/3}\sigma^{-1/2})$. This quantity can be rewritten as $\left(|\rho_{i-1}^+|-|\rho_{i-1}^-|\right)-\left(|\rho_i^+|-|\rho_i^-|\right)$, which is the difference of two square root expressions.  Thus, 
\beq\begin{split}
\left(|\rho_{i-1}^+|-|\rho_{i-1}^-|\right)-\left(|\rho_i^+|-|\rho_i^-|\right)
&=\frac{\left(|\rho_{i-1}^+|-|\rho_{i-1}^-|\right)^2-\left(|\rho_i^+|-|\rho_i^-|\right)^2}{|\rho_{i-1}^+|-|\rho_{i-1}^-|+|\rho_i^+|-|\rho_i^-|}\\
&=O\left(\frac{\left(|\rho_{i-1}^+|-|\rho_{i-1}^-|\right)^2-\left(|\rho_i^+|-|\rho_i^-|\right)^2}{n^{2/3}\sigma_n^{1/2}}\right).
\end{split}\eeq
Since the numerator inside the big-O term simplifies to $4\gamma m-4=O(n)$,
part \eqref{lem:rho_properties_anotherdiffbound} of the lemma follows. Lastly, since
\beq\begin{split}
\frac{|\rho_{i}^-|}{|\rho_{i}^+|}-\frac{|\rho_{i-1}^-|}{|\rho_{i-1}^+|}
&=\frac{1}{|\rho_i^+|}(|\rho_i^-|-|\rho_{i-1}^-|)
+\frac{|\rho_{i-1}^-|}{|\rho_i^+| |\rho_{i-1}^+|}(|\rho_{i-1}^+|-|\rho_i^+|)\\
&<\frac{|\rho_i^-|-|\rho_{i-1}^-|+|\rho_{i-1}^+|-|\rho_i^+|}{|\rho_i^+|},
\end{split}\eeq
applying parts \eqref{lem:rho_properties_rho+bound} and \eqref{lem:rho_properties_anotherdiffbound} of the lemma to this inequality, we obtain \eqref{lem:rho_properties_rhoratiocompare}.
\end{proof}

Since $\w_i = |\rho_i^-|/|\rho_{i-1}^+|$ for $i=3, \dots, n$, we know $\w_i$ takes values in $(0,1)$ and is increasing in $i$. Furthermore, the $i$-dependent asymptotic descriptions of $\w_i$ as $n \to \infty$ can also be obtained from the equation, as in the following lemma. 

\begin{lemma}\label{lem:wi}
	For $i\leq n$ satisfying $i\to\infty$ as $n \to \infty$, the value of $\omega_i$ satisfies the following asymptotic expressions. 
	\begin{enumerate}[(i)]
		\item\label{lem:wi(1)} If $n-i \ll n^{1/3} \sigma_n$, $\omega_{i} = 1 - 2\la^{-1/4} n^{-1/3} \sigma_n^{1/2} \left(1+O(n^{-1/3}\sigma_n^{1/2})\right)$.
		\item\label{lem:wi(2)} If $n-i =\Theta( n^{1/3} \sigma_n)$, $\omega_{i} = 1 - 2\left(\la^{-1/2} + \left(\la^{1/2} +1\right)^2\cdot \frac{n-i}{n^{1/3} \sigma_n}\right)^{1/2} n^{-1/3} \sigma_n^{1/2} \left(1+O(n^{-1/3}\sigma_n^{1/2})\right)$.
		\item\label{lem:wi(3)} If $n^{1/3}\sigma_n \ll n-i \ll n$, $\omega_{i} = 1 -2(1+\la^{1/2})\left(\frac{n-i}{n}\right)^{1/2} \left(1+O((\frac{n-i}{n})^{1/2})\right)$.
		\item\label{lem:wi(4)} If $n-i =\Theta( n)$, $\omega_{i}= \frac{\la^{-1/2}+\frac{n-i}{n} - (\la^{-1/2}+1)\left(\frac{n-i}{n}\right)^{1/2}}{\la^{-1/2}+\frac{n-i}{n} + (\la^{-1/2}+1)\left(\frac{n-i}{n}\right)^{1/2}}(1+o(1))$.
	\end{enumerate}
\end{lemma}
\begin{proof}
We begin by observing that
\beq
\frac{|\rho_{i-1}^-|}{|\rho_{i-1}^+|}<\omega_i<\frac{|\rho_{i}^-|}{|\rho_{i}^+|}
=\frac{m_i^-}{m_i^+},
\eeq
where
\beq\label{defn:mi_pm}
m_i^\pm
:=1\pm\sqrt{1-\frac{4(i-1)(m-n+i-1)}{(\gamma m-(m-n+2i-1))^2}}.
\eeq
By Lemma \ref{lem:rho_properties}, we obtain
\beq\label{eqn:wi-mi_unif}
\omega_i=\frac{m_i^-}{m_i^+}+O(n^{-2/3}\sigma_n^{-1/2}),
\eeq
and it suffices for the proof the lemma to consider $\frac{m_i^-}{m_i^+}$ in place of $\omega_i$. Setting $x=n-i+1$, we get
\beq\label{eqn:mi_pm_full}
\begin{split}
m_i^\pm
&=1\pm\sqrt{1-\frac{4(n-x)(m-x)}{(2\sqrt{mn}+mn^{-2/3}\sigma_n+2x-1)^2}}\\
&=1\pm\sqrt{\frac{4\la^{-3/2}\sigma_n n^{4/3}+4(1+\la^{-1/2})^2nx-4x^2+k_{n,x}}{4\la^{-1}n^2+4\la^{-3/2}\sigma_n n^{4/3}+8\la^{-1/2}nx+k_{n,x}}},
\end{split}\eeq
where $k_{n,x}=-4\la^{-1/2}n+(\la^{-1}\sigma_n n^{1/3}+2x-1)^2$.
We use the notation $m_i^\pm=1\pm f_n(x)$ where $f_n(x)$ is the square root term and observe that, when $x\ll n$, $f_n(x)=o(1)$.  In this case, $m_i^-/m_i^+=1-2f_n(x)+O(f_n(x)^2)$. Evaluating the leading order term of $f_n(x)$ gives us (i)-(iii) of the lemma.  To obtain (iv), we evaluate the expression $|\rho_i^-|/|\rho_i^+|$ directly, suppressing all $o(1)$ terms.
\end{proof}
\begin{corollary}\label{cor:wi}
There exist constants $0<C_1<C_2$ such that, for sufficiently large $n$ and uniformly in $i$, we have
\begin{enumerate}[(i)]
\item\label{cor:wi(1)} for $i\leq n-n^{1/3}\sigma_n$,\quad $C_1\left(\frac{n-i}{n}\right)^{1/2}<1-\omega_i<C_2\left(\frac{n-i}{n}\right)^{1/2}$,
\item\label{cor:wi(2)} for $i\geq n-n^{1/3}\sigma_n$,\quad 
$C_1 n^{-1/3}\sigma_n^{1/2} <1-\omega_i<C_2 n^{-1/3}\sigma_n^{1/2}$.
\end{enumerate}
\end{corollary}
Since $\g_i = \frac{|\rho_i^-|}{|\rho_i^+|}$, Lemma \ref{lem:rho_properties} and \eqref{eqn:wi-mi_unif} implies that
\begin{equation}\label{eqn:wi-gi_unif}
\g_i-\w_i = O(n^{-\frac23}\sigma_n^{-\frac12}) \quad  \text{uniformly in } i.
\end{equation}
In some instances, this uniform bound is not sufficient and an upper bound that depends on $i$ as in the following lemma is required (e.g. see Lemma \ref{lem:sumA0i}).
\begin{lemma}\label{lem:gi-wi}
	There exists constant $C>0$ such that for sufficiently large $n$,
	\beqq
	\g_i-\w_i<\frac{C}{n(1-\w_i)}, \quad  \text{for every } 3\leq i \leq n.
	\eeqq
\end{lemma}
\begin{proof}
We have the relation  
\beq\label{eqn:diff_giwi}
\g_i-\w_i=\frac{\w_i}{|\rho^+_i|}(|\rho^+_{i-1}|-|\rho^+_i|).
\eeq
Uniformly in $i\leq n$, $|\rho^+_i|=\Theta(n)$ and $\w_i\in(0,1)$, so it suffices to show $|\rho^+_{i-1}|-|\rho^+_i| = O(\frac{1}{1-\w_i})$. Define for $3\leq i \leq n$, 
\begin{equation}\label{defn:Ui}
U_i=\left(\g m - (m-n+2i-1)\right)^2-4(i-1)(m-n+i-1).
\end{equation}
Then $U_{i-1}-U_i=4(\g m -1)$, and by \eqref{eq:rho_def}, 
\beq\label{eqn:rho_asymp}
|\rho_i^+|=\frac12\left(\g m -(m-n+2i-1)+\sqrt{U_i}\right).
\eeq
We then note that $\frac{\sqrt{U_i}}{\g m - (m-n+2i-1)} =m_i^+-1$ by \eqref{defn:mi_pm} to arrive at 
\beq\label{eqn:diff_rho}
|\rho^+_{i-1}|-|\rho^+_i|
= 1+\frac{2(\g m -1)}{\sqrt{U_{i-1}}+\sqrt{U_{i}}}
=1+\frac{\frac{2(\g m -1)}{\g m - (m-n+2i-1)}}{(m_i^+-1)\left(1+\sqrt{1+\frac{4(\g m-1)}{U_i}}\right)}.
\eeq
Using the asymptotics $\g=(1+\sqrt{\la})^2+\sigma_n n^{-2/3}$ as $n\to\infty$, 
\beq\label{eqn:Ui_asymp}
U_i=4(\la^{-\frac12}+1)^2n^2\left[\frac{n-i}{n}+4\la^{-1}\left(\la^{-\frac12}+\frac{n-i}n\right)\sigma_nn^{-\frac23}+o(\sigma_nn^{-\frac23})\right].
\eeq
Thus, the ratio on the right hand side of \eqref{eqn:diff_rho} satisfies that its numerator is $O(1)$ while the expression under the square root in the denominator is $1+O(n^{-1})$. Both the big-O bounds are uniformly in $i$. Hence, the right hand side of \eqref{eqn:diff_rho} is of order $1+\frac{1}{m_i^+-1}$, where $\frac{1}{m_i^+-1}\geq 1$, by the definition of $m_i^+$. Therefore, 
\[
|\rho^+_{i-1}|-|\rho^+_i|
= O\left(\frac{1}{m_i^+-1}\right).
\]
Since $m_i^++m_i^-=2$,
\beqq
\frac1{m_i^+-1}=\frac2{m_i^+(1-\g_i)}=\frac2{m_i^+(1-\w_i)\left(1-\frac{\g_i-\w_i}{1-\w_i}\right)}
=\frac{2/m_i^+}{1-\w_i}\left(1+O\left(n^{-\frac13}\sigma_n^{-1}\right)\right),
\eeqq
following from \eqref{eqn:wi-gi_unif} and Corollary \ref{cor:wi}. We conclude $|\rho_{i-1}^+|-|\rho_i^+|=O\left(\frac1{1-\w_i}\right)$.
\end{proof}

One other quantity that comes up frequently throughout our calculations is the variance $\bE X_i^2$.  In the following lemma, we give upper and lower bounds for this quantity.
\begin{lemma}
The variance of $X_i^2$ satisfies the following properties for $3\leq i\leq n$:
\begin{enumerate}[(i)]
\item $\EE X_i^2=\Theta (\frac{\a\delta_i}{n})$ for all $i$,
\item $\EE X_i^2=O(n^{-1})$ uniformly in $i$,
\item $\EE X_i^2=\Omega(n^{-2})$ uniformly in $i$.
\end{enumerate}
\end{lemma}

\begin{proof}
From \eqref{defn:Xi}, we have
\beq\label{eqn:varXi}\begin{split}
\EE X_i^2=(1+\tau_{i-1})^2\EE(\delta_i\a_{i-1}+\b_i)^2
=\a\delta_i(1+\tau_{i-1})^2\left(\frac{\w_i}{|\rho_{i-1}^+|}+\frac{1}{|\rho_i^+|}\right).
\end{split}\eeq
By Lemma \ref{lem:rho_properties}, $|\rho_i^+|^{-1}=\Theta(n^{-1})$.  Furthermore, it follows directly from definitions that $\tau_i,\w_i$ are positive and bounded above by a constant, uniformly in $i$.  This yields part (i) of the lemma.  Parts (ii) and (iii) follow from the fact that $\delta_i=\frac{i-1}{|\rho_i^+|}=\Theta\left(\frac{i-1}{n}\right)$.
\end{proof}

\section{Expressing \texorpdfstring{$\log |\mathcal{D}_n|$}{log|Dn|} in terms of \texorpdfstring{$\log|E_n|$}{log|En|}} \label{sec:A.5}
Our goal in this section is to obtain a closed form asymptotic expansion for the quantity $\log|\mathcal{D}_n|-\log|E_n|$, accurate down to order $O(1)$. 
We will use this to obtain a CLT for $\log|\cD_n|$ in terms of a CLT for $\log|E_n|$.  

\begin{lemma}
Assume $\g=(1+\sqrt{\la})^2+\sigma_n n^{-2/3}$ for $(\log\log n)^2\ll\sigma_n\ll (\log n)^2$. The quantity $\log|\mathcal{D}_n|-\log|E_n|$ has the asymptotic expansion
\beq\label{eqn:diff_logDnEn}
\log|\mathcal{D}_n|-\log|E_n|
=C_\la n+\frac{1}{\la^{1/2}(1+\la^{1/2})}\sigma_n n^{1/3}
-\frac{2}{3\la^{3/4}(1+\la^{1/2})^2}\sigma_n^{3/2}+O(1),
\eeq
where
\beq
C_\la:=(1-\la^{-1})\log(1+\la^{1/2})+\log(\la^{1/2})+\la^{-1/2}.
\eeq
\end{lemma}
\begin{proof}
It follows from \eqref{defn:Ei} that
\beq
E_n=\frac{m^n\mathcal{D}_n}{\prod_{i=1}^n|\rho_i^+|}.
\eeq
Expanding $|\rho_i^+|$ using \eqref{eq:rho_def}, we obtain
\beq\begin{split}
	\cD_n
	&=E_n\prod_{i=1}^n\left(\frac12(\gamma-(1-\la))-\frac{i-\frac12}{m}+\sqrt{\left(\frac12(\gamma-(1-\la))-\frac{i-\frac12}{m}\right)^2-\left(1-\la+\frac{i-1}{m}\right)\left(\frac{i-1}{m}\right)}\right).
\end{split}\eeq
Thus,
\beq\begin{split}
	\log|\cD_n|-\log|E_n|=&\sum_{i=1}^n
	\log\left(\tfrac12(\gamma-(1-\la))-\tfrac{i-\frac12}{m}
	+\sqrt{\left(\tfrac12(\gamma-(1-\la))-\tfrac{i-\frac12}{m}\right)^2-\tfrac{i-1}{m}\left(1-\la+\tfrac{i-1}{m}\right)}\right).
\end{split}\eeq
Observe that the argument of the log is bounded away from zero, since is it equal to $|\rho_i^+|/m$ where $|\rho_i^+|=\Theta(n)$ by Lemma \ref{lem:rho_properties}(\ref{lem:rho_properties_rho+bound}).  For large $n$, we approximate the above sum by the integral
\beq\label{eqn:int_log}
\frac{n}{\la}\int_0^\la \log\left(c-x+\sqrt{(c-x)^2-(1-\la+x)x}\right)dx
\eeq
for $c=\frac12(\gamma-(1-\la))$, incurring an error of order $O(1)$ in the process. Note that 
\beqq
c-x+\sqrt{(c-x)^2-(1-\la+x)x}=\left(\sqrt{\frac{c^2}{\g}-x}+r_+\right)\left(\sqrt{\frac{c^2}{\g}-x}+r_-\right),
\eeqq
where $r_\pm=\frac12\left(\g\pm(1-\la)\right)\g^{-1/2}$. For every $s\in\bR$,
\beqq
\int\log(\sqrt y +s)dy=(y-s^2)\log(\sqrt y +s)-\frac12y+s\sqrt y+C.
\eeqq
Thus, using $s=r_\pm$ together with the change of variable $y=\frac{c^2}{\g}-x$, we find that   \eqref{eqn:int_log} is equal to 
\beq\label{eqn:A}
\frac n\la A=\frac n\la(A_1+A_2+A_3+A_4+A_5),
\eeq
where
\beq\begin{split}
	A_1&=(a+\la-r_+^2)\log(\sqrt{a+\la}+r_+),\qquad
	A_2=-(a-r_+^2)\log(\sqrt{a}+r_+),\\
	A_3&=(a+\la-r_-^2)\log(\sqrt{a+\la}+r_-),\qquad
	A_4=-(a-r_-^2)\log(\sqrt{a}+r_-),\\
	A_5&=-\la+(r_++r_-)(\sqrt{a+\la}-\sqrt{a}),
\end{split}\eeq
and $a=\frac{c^2}{\g}-\la$. Therefore, 
\beq\label{eqn:diff_log}
\log|\cD_n|-\log|E_n|=\frac n\la A+O(1).
\eeq
We now evaluate each of $A_i$ asymptotically, using $\g=(1+\sqrt{\la})^2+\sigma_n n^{-2/3}$ as given. Setting $\Delta_n:=\frac{\sigma_n n^{-2/3}}{(1+\la^{1/2})^2}$, we have 
\begin{align*}
	r_+&=1+\tfrac12\la^{1/2}\D+O(\D^2), \qquad r_-=\la^{1/2}+\tfrac12\D+O(\D^2),\\
	a&=\la^{1/2}\D+\tfrac14 (1-\la^{1/2})^2\D^2+O(\D^3).
\end{align*}
Therefore,  $A_3=O(\Delta_n^2)$, and 
\begin{align*}
	A_1&=(\la-1)\left(\log(1+\la^{1/2})+\tfrac12\D\right)+O(\D^2),\\
	A_2&=\la^{1/4}\D^{1/2}+\left(\tfrac18 \la^{-1/4}-\tfrac14 \la^{1/4}-\tfrac{1}{24} \la^{3/4}\right)\D^{3/2}+O(\D^2),\\
	A_4&=\la\log(\la^{1/2})+\la^{3/4}\D^{1/2}
	+(\tfrac{11}{24}\la^{1/4}-\tfrac34\la^{3/4}+\tfrac18\la^{5/4})\D^{3/2}+O(\D^2),\\
	A_5&=\la^{1/2}-\la^{1/4}(1+\la^{1/2})\D^{1/2}+\tfrac12(1+ \la^{1/2})^2\D
	-\tfrac{1}{8\la^{1/4}}(1+\la^{1/2})^3\D^{3/2}+O(\D^2).
\end{align*}
Substituting the values of $A_i$ into \eqref{eqn:A}, then by \eqref{eqn:diff_log}, we obtain the statement \eqref{eqn:diff_logDnEn} as in the lemma. 
\end{proof}

We now move to the step of approximating $\log|E_n|$.

\section{Linear approximation for \texorpdfstring{$\log |E_n|$}{log|En|}}\label{sec:A.4}
Recall Definition \ref{defn:Ri} of $R_i$. Assuming that $R_i$ for $3 \leq i \leq n$ are $o(n^{-1/3})$ uniformly in $i$, then Taylor expansion of the logarithm implies
\beq
\begin{split}
	\log|E_n|
	&=\sum_{i=3}^n\log|1-R_i|+\log|E_2|
	=\sum_{i =3}^{n}(-R_i-R_i^2/2+o(n^{-1}))+\log|E_2|.
\end{split}
\eeq
The following lemma shows that the uniform bound of $R_i$ indeed holds. 
\begin{lemma}\label{lem:unif_Ri}
Assume $(\log\log n)^2\ll\sigma_n\ll(\log n)^2$. With probability $1 - O(\log^{-5}n)$,
\beqq
\max_{2 \leq i \leq n} |R_i| = o(n^{-1/3}).
\eeqq
\end{lemma}
We include its proof in Section \ref{sec:unif}. Assuming the lemma, we rewrite \eqref{defn:wi_ei} as
\beq
\e_i = -(\g_i-\w_i)+\a_{i-1}\b_i+(\a_{i-1}\b_i+\a_{i-1}\delta_i+\tau_{i-1}\b_i)\frac{R_{i-1}}{1-R_{i-1}}+\w_i\frac{R_{i-1}^3}{1-R_{i-1}}+\w_iR_{i-1}^2,
\eeq
and set for $3\leq i\leq n$,
\beq
R^{(1)}_i = \frac{R_{i-1}}{1 - R_{i-1}}, \quad R^{(2)}_i = \w_{i} \frac{R_{i-1}^3}{1 -R_{i-1}}, \quad R^{(3)}_i = \w_{i} R_{i-1}^2.
\eeq
Then from the recursion \eqref{eqn:Rirecursion}, we obtain the decomposition
\beq\label{eqn:Ri_AB}
R_i=L_i+\w_i\dots\w_3R_2-A_{0i}+B_{0i}+B_{1i}+B_{2i}+B_{3i},
\eeq	
where
\beq\label{def:A0i}
A_{0i}=\g_i - \w_i + \w_i(\g_{i-1} - \w_{i-1}) +\dots+\w_i\dots\w_4(\g_3 - \w_3),
\eeq and
\begin{align*}
	B_{0i} &= \left(\a_{i-1}+ ( \tau_{i-1} + \a_{i-1} ) R^{(1)}_i\right) \beta_i  + \w_i \left(\a_{i-2} +( \tau_{i-2} + \a_{i-2} )R^{(1)}_{i-1}\right) \beta_{i-1} \\
	&\quad + \dots + \w_i\dots\w_4 \left(\a_2  + (\tau_2  + \a_2 ) R^{(1)}_3\right) \beta_3,\\
	B_{1i} &= \a_{i-1}\delta_i R^{(1)}_i + \w_i\a_{i-2}\delta_{i-1}R^{(1)}_{i-1} + \dots + \w_i\dots \w_4\a_2 \delta_3 R^{(1)}_3,\\
	B_{2i} &= R^{(2)}_i + \w_iR^{(2)}_{i-1} + \dots + \w_i\dots\w_4 R^{(2)}_3,\\
	B_{3i} &= R^{(3)}_i + \w_iR^{(3)}_{i-1} + \dots + \w_i\dots\w_4 R^{(3)}_3.
\end{align*}
Substituting this into expression for $\log|E_n|$, we have
\beq\label{eqn:logEn}
\log|E_n|=-\sum_{i=3}^nL_i+\sum_{i =3}^nA_{0i}-\sum_{i=3}^n B_{3i} - \sum_{i=3}^n \left(\w_i\dots\w_3R_2+B_{0i}+B_{1i}+B_{2i}\right) - \frac12\sum_{i=3}^nR_i^2+\log|E_2|+o(1).
\eeq
The following three lemmas state that the last three quantities in \eqref{eqn:logEn} are $O(1)$ with probability $1 - o(1)$. Their proofs are included in the \hyperref[sec:technical_lemmas]{Appendix}. 
\begin{lemma}\label{lem:Ri2_bound}
	$\sum_{i=3}^n R_i^2 = O(1)$ with probability $1 - o(1)$.
\end{lemma}
\begin{lemma}\label{lem:AB_bound}
	$\sum_{i=3}^n \w_{i}\dots\w_3R_2+B_{0i}+B_{1i}+B_{2i} = O(1)$ with probability $1 - o(1)$. 
\end{lemma}
\begin{lemma}\label{lem:E2_bound}
	$\log|E_2| = O(1)$ with probability $1 - o(1)$.
\end{lemma}
The above lemmas imply that the main contribution to $\log|E_n|$ comes from the first three sums $\sum_{i=3}^nL_i$, $\sum_{i =3}^nA_{0i}$, and $\sum_{i=3}^n B_{3i}$. We turn to the tasks of computing the second and third sums in this section, and study the first sum $\sum_{i=3}^nL_i$ in Section \ref{sec:A.1}.
\begin{definition}\label{defn:gi}
	Given integer $n$, define sequence $\{g_i\}_{i=3}^{n+1}$ by the recurrence
	\beqq
	g_{n+1}=1, \quad g_i=1+\w_ig_{i+1}.
	\eeqq 
	That is, $g_i = 1+\w_i+\w_i\w_{i+1}+\dots +\w_i\dots\w_n$ for $3\leq i \leq n$.
\end{definition}

\begin{lemma}\label{lem:sumA0i}
\beq
\sum_{i = 3}^nA_{0i} =\frac16\log n +O(\log\log n).
\eeq
\end{lemma}
\begin{proof}
	We prove the lemma by computing an upper and a lower bound for the sum $\sum_{i = 3}^nA_{0i}$. Observe that 
	\beq\label{eqn:sumA0i}
	\sum_{i=3}^nA_{0i}=\sum_{i=3}^ng_{i+1}(\gamma_i-\w_i)>	\sum_{i=n-n\nu_n^{-1}}^{n-n^{1/3}\sigma_n\nu_n}g_{i+1}(\g_i-\w_i),
	\eeq
	for any slowing increasing sequence $\nu_n$. For the purpose of this proof, it suffices to take $\nu_n=\log\log n$. 
	
	The indices $i$ in the sum on the right hand side of \eqref{eqn:sumA0i} satisfies $i< n-n^{1/3}$, so by Lemma \ref{lem:gi_bounds}\eqref{lem:gi_bounds_upper},  $g_{i+1}>\frac{1-\log^{-2}n}{1-\w_{i+1}}$ for sufficiently large $n$. We obtain
	\beq\label{eqn:sum_A0i}
	\sum_{i=3}^nA_{0i}>(1-\log^{-2}n)	\sum_{i=n-n\nu_n^{-1}}^{n-n^{1/3}\sigma_n\nu_n}\frac{\g_i-\w_i}{1-\w_{i+1}}.
	\eeq
	Recall that $\g_i-\w_i=\w_i\frac{|\rho_{i-1}^+|-|\rho_i^+|}{|\rho_i^+|}$. Since $n^{1/3}\sigma_n\ll n-i\ll n$, Lemma \ref{lem:wi}\eqref{lem:wi(3)} implies $\w_i=1+O\left(\sqrt{\frac{n-i}{n}}\right)$ and 
	\beq\label{eqn:1-wi_asymp}
	(1-\w_{i+1})^{-1}=\frac1{2(\la^{\frac12}+1)}\left(\frac{n-i}n\right)^{-\frac12}
	\left(1+O\left(\sqrt{\frac{n-i}n}\right)\right).
	\eeq
	We now study the factor $\frac{|\rho_{i-1}^+|-|\rho_i^+|}{|\rho_i^+|}$. From the second display of \eqref{eqn:diff_rho},
	\beqq
	|\rho_{i-1}^+|-|\rho_i^+|
	=1+\frac{2(\g m-1)}{\sqrt{U_{i-1}}+\sqrt{U_i}} = 1+\frac{2(\la^{-\frac12}+1)^2n\left(1+O\left(\sigma_n n^{-2/3}\right)\right)}{\sqrt{U_{i-1}}+\sqrt{U_i}},
	\eeqq
	where $U_i$ is defined in \eqref{defn:Ui}. By \eqref{eqn:Ui_asymp}, we get 
	\begin{equation}\label{eqn:sqrtUi}
		\begin{split}
			\sqrt{U_i}&=2(\la^{-\frac12}+1)n\left(\frac{n-i}{n}\right)^{\frac12}\left(1+O\left(n^{-\frac23}\sigma_n\left(\frac{n-i}{n}\right)^{-1}\right)\right),
		\end{split}
	\end{equation}
	noting that $\left(\frac{n-i}{n}\right)^{-1}\geq \nu_n$. Together with the asymptotics
	\beq\label{eqn:Vi}
	\g m-(m-n+2i-1)
	=2n\left(\la^{-1/2}+O\left(\frac{n-i}n\right)\right),
	\eeq
	it follows that 
	\beq\label{eqn:rho_asymp2}
	|\rho_i^+|=\frac12\left(\g m-(m-n+2i-1)+\sqrt{U_i}\right)=\la^{-\frac12}n\left(1+O\left(\sqrt{\frac{n-i}n}\right)\right).
	\eeq
	Thus, 
	\begin{align}
		\frac{|\rho_{i-1}^+|-|\rho_i^+|}{|\rho_i^+|}
		=\frac{1+\la^{1/2}}{2}\frac1n\left(\frac{n-i}{n}\right)^{-\frac12}
		\left(1+O\left(\sqrt{\frac{n-i}{n}}+n^{-\frac23}\sigma_n\left(\frac{n-i}{n}\right)^{-1}\right)\right).\label{eqn:rho_ratio}
	\end{align}
	Combine \eqref{eqn:1-wi_asymp} and \eqref{eqn:rho_ratio}, we get
	\beq\label{eqn:ratio}
	\frac{\g_i-\w_i}{1-\w_{i+1}}
	=\frac1{4n}\left(\frac{n-i}{n}\right)^{-1}\left(1+O\left(\sqrt{\frac{n-i}{n}}+n^{-\frac23}\sigma_n\left(\frac{n-i}{n}\right)^{-1}\right)\right).
	\eeq
	Therefore, by \eqref{eqn:sum_A0i},
	\beq
	\sum_{i=3}^nA_{0i}>\sum_{i=n-n\nu_n^{-1}}^{n-n^{1/3}\sigma_n\nu_n}\frac1{4n}\left(\frac{n-i}{n}\right)^{-1}\left(1+O\left(\sqrt{\frac{n-i}{n}}+n^{-\frac23}\sigma_n\left(\frac{n-i}{n}\right)^{-1}\right)\right).
	\eeq
	Since
	\beqq
	\begin{split}
	\sum_{i=n-n\nu_n^{-1}}^{n-n^{1/3}\sigma_n\nu_n}\frac1n\left(\frac{n-i}{n}\right)^{-1}&=\frac23\log n +O(\log\sigma_n)=\frac23\log n +O(\log\log n) ,\\
	\sum_{i=n-n\nu_n^{-1}}^{n-n^{1/3}\sigma_n\nu_n}\frac1n\left(\frac{n-i}{n}\right)^{-\frac12}&=O(\nu_n^{-1/2}),\quad  \text{and} \quad n^{-\frac23}\sigma_n\sum_{i=n-n\nu_n^{-1}}^{n-n^{1/3}\sigma_n\nu_n}\frac1n\left(\frac{n-i}{n}\right)^{-2}=O(\nu_n^{-1}),
	\end{split}
	\eeqq
	we obtain the lower bound $\sum_{i = 3}^nA_{0i}>\frac16\log n +O(\log\log n)$.
	
	It remains to show $\sum_{i = 3}^nA_{0i}<\frac16\log n +O(\log\log n)$. Since $\w_i$ and $\frac{|\rho_{i-1}^{+}|-|\rho_i^+|}{|\rho_i^+|}$ are both increasing in $i$, $\g_i-\w_i$ is also increasing in $i$. Thus, it follows from \eqref{def:A0i} that $A_{0i}<\frac{\g_i-\w_i}{1-\w_i}$ for every $i=3,4,\dots,n$ and so, $\sum_{i = 3}^nA_{0i}<\sum_{i = 3}^n\frac{\g_i-\w_i}{1-\w_i}$. 
	By Lemma \ref{lem:gi-wi} and Corollary \ref{cor:wi}, the following three statements hold:
	\begin{align}
		\sum_{i=3}^{n-n\nu_n^{-1}}\frac{\g_i-\w_i}{1-\w_i}&=O\left(\sum_{i = 3}^{n-n\nu_n^{-1}}\frac1n\left(\frac{n-i}n\right)^{-1}\right)=O(\log \nu_n),\\
		\sum_{i=n-n^{1/3}\sigma_n\nu_n}^{n-n^{1/3}\sigma_n}\frac{\g_i-\w_i}{1-\w_i}&=O\left(\sum_{i=n-n^{1/3}\sigma_n\nu_n}^{n-n^{1/3}\sigma_n}\frac1n\left(\frac{n-i}n\right)^{-1}\right)=O(\log \nu_n),\\
		\sum_{i=n-n^{1/3}\sigma_n}^{n}\frac{\g_i-\w_i}{1-\w_i}&=O\left(\sum_{i=n-n^{1/3}\sigma_n}^{n} \frac1n \left(n^{-\frac13}\sigma_n^{\frac12}\right)^{-1}\right)=O(1).
	\end{align}
	Thus, we obtain 
	\beq
	\sum_{i = 3}^nA_{0i}<\sum_{i=n-n\nu_n^{-1}}^{n-n^{1/3}\sigma_n\nu_n}\frac{\g_i-\w_i}{1-\w_i}+O(\log \nu_n),
	\eeq
	where the sum over $i$ on the right hand side is $\frac16\log n +O(\log\log n)$ by \eqref{eqn:ratio}. This completes the proof of the lemma.
\end{proof}

We now study contribution from the sum $\sum_{i=3}^n B_{3i}$. The following lemma states that $\sum_{i=3}^n B_{3i}$ is close to $\sum_{i=3}^n B^*_{3i}$, where
\beqq
B^*_{3i} = (\w_{i} L^2_{i-1})+\w_i(\w_{i-1}L^2_{i-2})+\dots+\w_i\dots\w_4(\w_{3}L^2_{2}).
\eeqq
\begin{lemma}\label{lem:B3i_star_hat}
	With probability $1 - o(1)$, $\sum_{i = 3}^nB_{3i}-B^*_{3i} = O(1)$.
\end{lemma}
The new sum is much simpler, and we turn now to the task of computing it.
We begin by observing that $\sum_{i=3}^n B_{3i}*$ can be rewritten as
\beq\begin{split}\label{eq:B3i_decomposition}
&\sum_{i=3}^n B_{3i}^* = \sum_{i=4}^n (g_i-1)L_{i-1}^2\\
&=\sum_{i=4}^n (g_i-1)Y_{i-1}^2+\sum_{i=4}^n (g_i-1)\left[2Y_{i-1}(\alpha_{i-1}-\w_3\cdots\w_{i-1}\a_2)+(\alpha_{i-1}-\w_3\cdots\w_{i-1}\a_2)^2\right].
\end{split}\eeq
The dominant contribution comes from the first sum while the second sum is bounded of constant order.  We state this more precisely in the following two lemmas.

\begin{lemma}\label{lem:B3i_smallstuff}
For the second sum in \eqref{eq:B3i_decomposition}, with probability $1-o(1)$, we have the bound  
\beq
\sum_{i=4}^n (g_i-1)\left[2Y_{i-1}(\alpha_{i-1}-\w_3\cdots\w_{i-1}\a_2)+(\alpha_{i-1}-\w_3\cdots\w_{i-1}\a_2)^2\right]
=O(1).
\eeq
\end{lemma}
\begin{lemma}\label{lem:B3i_bigstuff} With probability $1-o(1)$,
\beq
\sum_{i=4}^n (g_i-1)Y_{i-1}^2 = \frac{\a}6 \log n+O(\log\log n).
\eeq
\end{lemma} 
From \eqref{eqn:logEn} and \eqref{eq:B3i_decomposition}, the Lemmas \ref{lem:Ri2_bound}-\ref{lem:B3i_bigstuff} together yield that
\beq
\log|E_n|=-\sum_{i = 3}^n L_i+\frac{1-\a}{6}\log n+ O(\log\log n).
\eeq
We postpone the proofs of Lemma \ref{lem:B3i_star_hat} and Lemma \ref{lem:B3i_smallstuff} to the \hyperref[sec:technical_lemmas]{Appendix} and turn now to proving Lemma \ref{lem:B3i_bigstuff}. 

In the proofs of Lemmas \ref{lem:B3i_smallstuff} and \ref{lem:B3i_bigstuff}, we will need the following lemma, which is a Hanson Wright type inequality (see, for example Proposition 1.1 from G{\"o}tze, Sambale, and Sinulis \cite{GSS_HansonWright}).  We employ this lemma in a similar manner to the way that Johnstone et al handle such quadratic forms in their paper \cite{JKOP1}. 

\begin{lemma}\label{lem:HansonWright}
Let $\mathbf{x}=(x_1,...,x_n)^T$ be a vector with independent subgamma entries satisfying $x_i\in SG(v,u)$ where $v,u\leq Cn^{-1}$ for some $C>0$. Then, for any symmetric matrix $A$,
\beq
|\mathbf{x}^TA\mathbf{x}-\EE\mathbf{x}^TA\mathbf{x}|=O(\nu_n n^{-1}\|A\|_{\HS})\text{ with probability at least }1-\nu_n^{-1},
\eeq
for any $\nu_n>0$ (For the purposes of this paper we take $\nu_n$ to be a slowly growing function such as $\log\log n$).
\end{lemma}
\begin{proof}
It follows from Definition \ref{def:SG} that, there is a constant $c>0$ such that $\|x_i\|_{\Psi_1}\leq cn^{-\frac12}$ for all $i=1,\dots, n$. Here, 
\[
\|X\|_{\Psi_1}:=\inf\left\{t>0: \bE\exp(|X|/t)\leq 2\right\}
\]
denotes the (exponential) Orlicz norm. Thus, by Proposition of 1.1 of \cite{GSS_HansonWright}, for some constants $C_1,C_2>0$, 
\beq\label{eqn:quadform}
\begin{split}
\bE|\mathbf{x}^TA\mathbf{x}-\EE\mathbf{x}^TA\mathbf{x}|
&\leq\int_0^\infty2\exp\left(-\frac1{C_1}\min\left\{\frac{n^2t^2}{\|A\|^2_{\HS}},\left(\frac{nt}{\|A\|_{\text{op}}}\right)^{1/2}\right\}\right)dt\\
&=\frac{2\|A\|_{\HS}}{n}\int_0^\infty \exp\left(-\frac1{C_1}\min\left\{u^2,\left(u\frac{\|A\|_{\HS}}{\|A\|_{\text{op}}}\right)^{1/2}\right\}\right)du\\
&\leq \frac{C_2\|A\|_{\HS}}{n}.
\end{split}
\eeq
Applying Markov's inequality to $\bP\left(|\mathbf{x}^TA\mathbf{x}-\EE\mathbf{x}^TA\mathbf{x}|>C_3\nu_n n^{-1}\|A\|_{\HS}\right)$ and using \eqref{eqn:quadform}, we obtain the lemma with appropriate constant $C_3>0$ depending on $C_1$ and $C_2$.
\end{proof}

\subsection{Proof of Lemma \ref{lem:B3i_bigstuff}}
We begin by showing that $\sum(g_i-1)Y_{i-1}^2$ is close to its expectation with probability approaching one, then proceed to compute the leading order term of the expectation.
\begin{definition}\label{def:HW_notations} We define the following notations to be used in this proof and also in the \hyperref[sec:technical_lemmas]{Appendix}:
\beq\begin{split}
		W &= \begin{pmatrix}
			1 & & & & \\
			\omega_4 & 1& & & \\
			\omega_4 \omega_5 & \omega_5 & 1 & & &\\
			\vdots & \vdots & \vdots & \ddots & &\\
			\omega_4 \dots \omega_{n-2} & \omega_5 \dots \omega_{n-2} & \dots & \omega_{n-2} & 1& \\
			\omega_4 \dots \omega_{n-1} & \omega_5 \dots \omega_{n-1} & \dots & \omega_{n-2} \omega_{n-1} & \omega_{n-1} &1 \\
		\end{pmatrix}, \\
		G  &= \text{diag} (g_4 -1, \dots, g_{n-1}-1, g_n-1),\\
		D  &=  \text{diag} (1+\tau_2,1+\tau_3, \dots, 1+\tau_{n-2}),\\
		\bfY &= (Y_3, Y_4 \dots, Y_{n-1})^T, \\
		 \bfX& = (X_3, X_4 \dots, X_{n-1})^T.\\
	\end{split}\eeq
	\end{definition}
Observe that $Y = W X$ by Definition \ref{defn:Yi} and we can write
\beq
\sum_{i=4}^n (g_i-1)Y_{i-1}^2 = \bfY^T G\bfY=\bfX^T W^T GW\bfX.
\eeq
Note that $\bfX$ is a vector of independent sub-gamma random variables satisfying the conditions of Lemma \ref{lem:HansonWright} and $W^TGW$ is a symmetric, deterministic matrix.  Thus, by the lemma, we conclude that, with probability $1-O(\sigma_n^{-1/2})$,
\beq
|\bfX^T W^T GW\bfX-\EE\bfX^T W^T GW\bfX|=O\left(\sigma_n^{1/2} n^{-1}\|W^T GW\|_{\HS}\right)
=O\left(\sigma_n^{1/2} n^{-1}\|W\| \|GW\|_{\HS}\right).
\eeq
To bound $\|W\|$, we break $W$ up as a sum of $n$ matrices, each containing one of the subdiagonals of the matrix $W$.  The first such matrix contains the elements $1,1,...,1$, the second contains $\w_4, \w_5,...,\w_{n-1}$, and so forth.  The norm of each of these matrices is equal to its largest element, so, using Lemma \ref{lem:wi}, we get
\beq
		\|W\| \leq 1 + \omega_{n-1} + \dots + \omega_{n-1}\dots \omega_4 \leq \frac{1}{1 - \omega_{n-1}} = O(n^{1/3}\sigma_n^{-1/2}).
\eeq
To bound $\|GW\|_{\HS}$, we use Lemma \ref{lem:gi_bounds} and Corollary \ref{cor:wi}, and conclude that
	\beq\begin{split}
		\|G W\|_{\HS} &= \left(\sum_{i = 4}^{n}(g_i-1)^2(1 + \omega_{i-1}^2 + \dots + \omega_{i-1}^2\dots \omega_4^2 ) \right)^{1/2} \\
		&\leq \left(\sum_{i = 4}^{n}\frac{(g_i-1)^2}{1 - \omega_{i-1}^2} \right)^{1/2} = O \left(\left(\sum_{i = 4}^{n} (1 - \omega_{i-1})^{-3} \right)^{1/2} \right) =O(n^{2/3} \sigma_n^{-1/4}).
	\end{split}\eeq
	Thus  with probability $1-O(\sigma_n^{-1/2})$,
	\begin{equation}
		\left|\bfX^T W^T GW\bfX-\EE\bfX^T W^T GW\bfX\right|= O(\sigma_n^{-1/4}) = o(1).
	\end{equation}
We now turn to the task of computing the leading term of 
\beq
\begin{split}
\bE \bfX^TW^T GW \bfX 
&= \sum_{i = 1}^{n-3} (W^T GW)_{ii} \bE X_{i+2}^2
= \sum_{i = 1}^{n-3} \sum_{j = i}^{n-3} (\omega_{i+3} \dots \omega_{j+2})^2\omega_{j+3}g_{j+4} \bE X_{i+2}^2.			
\end{split}
\eeq
It will be convenient to switch the order of summation and rewrite this as 
\beq\label{eq:A4sum_indexreverse}\begin{split}
\sum_{i = 1}^{n-3} \sum_{j = i}^{n-3} (\omega_{i+3} \dots \omega_{j+2})^2\omega_{j+3}g_{j+4} \bE X_{i+2}^2
&=\sum_{j = 1}^{n-3} \left(\sum_{i = 1}^{j} (\omega_{i+3} \dots \omega_{j+2})^2 \bE X_{i+2}^2\right)\left(\sum_{k=j}^{n-3}\w_{j+3}\cdots\w_{k+3}\right).
\end{split}\eeq
It turns out that the dominant contribution comes from the portion of the sum where the indices are restricted to $n-n(\log n)^{-2}<i\leq j<n-n^{1/3}\sigma_n\log n $.  We will begin by computing the sum on those indices and then show that the sum of the remaining terms is small.  Thus, our first task is to compute
\beq\label{eq:A4sum_brokenup}\begin{split}
\sum_{j=n-n(\log n)^{-2}}^{n-n^{1/3}\sigma_n\log n}
\left(\sum_{i=n-n(\log n)^{-2}}^j(\omega_{i+3}\cdots\omega_{j+2})^2
\EE X_{i+2}^2\right)
\left(\sum_{k=j}^{n-3}\omega_{j+3}\cdots\omega_{k+3}\right).
\end{split}\eeq
For the purposes of calculating this, we begin by computing the asymptotics of a product of the form 
\[
\prod_{i=i_1}^{i_2}\omega_i=\exp\left(\sum_{i=i_1}^{i_2}\log\w_i\right),
\]
where $n^{1/3}\sigma_n\ll n-i \ll n$ for indices $i$ in the range $i_1\leq i\leq i_2$. 
On this range of indices, it follows from the proof of Lemma \ref{lem:wi}\eqref{lem:wi(3)} that $\w_{i}$ has a series expansion in powers of $(\frac{n-i}{n})^{1/2}$ and $n^{-1/3}\sigma_n^{1/2}$.  Furthermore, using Taylor expansion of log, we can write
\beqq
\log\w_i
=-2(1+\lambda^{1/2})\left(\frac{n-i}{n}\right)^{1/2}\left(1+s\left(\frac{n-i}{n}\right)\right),
\eeqq
where $s(x)$ is a series in positive powers of $x^{1/2}$ and $n^{-1/3}\sigma_n^{1/2}$ (including mixed terms) and the series converges when $x=\frac{n-i}{n}$ for $i$ in the range of indices described above. 

Using this representation, we have
\beq\label{eq:wiproductcompute}\begin{split}
\prod_{i=i_1}^{i_2}\omega_i
&=\exp\left(\sum_{i=i_1}^{i_2}\left(-2(1+\lambda^{\frac12})\left(\frac{n-i}{n}\right)^{\frac12}\left(1+s(\tfrac{n-i}{n})\right)\right)\right)\\
&= \exp\left(-2(1+\lambda^{\frac12})n\int_{i_1/n}^{i_2/n}(1-x)^{\frac12}\left(1+s(1-x)\right)dx(1+O(n^{-1})) \right)\\
&= \exp\left(-\frac43(1+\lambda^{\frac12})n\left[\left(\frac{n-i_1}{n}\right)^{\frac32}(1+\tilde{s}(\tfrac{n-i_1}{n}))-\left(\frac{n-i_2}{n}\right)^{\frac32}(1+\tilde{s}(\tfrac{n-i_2}{n}))\right]  \right)\left(1+O\left((\tfrac{n-i_1}{n})^{\frac32}\right)\right),
\end{split}\eeq
where $(1-x)^{3/2}(1+\tilde{s}(1-x))$ is the series that emerges as the antiderivative of the integrand in the second line of the display above and $\tilde{s}(1-x)=O((1-x)^{1/2})$.
Next, we compute $\EE X_i^2$ on the indices $i<n-n^{1/3}\sigma_n\log n $. By
\eqref{eqn:varXi} and Lemma \ref{lem:rho_properties},
\beq\label{eqn:prod_gX}
\EE X_i^2
= \a \frac{\delta_i}{|\rho_{i}^+|}(1+\tau_{i-1})^2\left(\omega_i+1+O(n^{-2/3}\sigma_n^{-1/2})\right).
\eeq
We then consider each factors on the right hand side individually. Using Lemma \ref{lem:wi} to obtain asymptotics for $\w_i$ with $i<n-n^{1/3}\sigma_n\log n $, we obtain $\omega_i+1+O(n^{-\frac23}\sigma_n^{-\frac12})
=2+O\left(\sqrt{\tfrac{n-i}{n}}\right)$ and
\beq
\frac{\d_i}{|\rho_i^+|}=\frac{\omega_i}{\tau_{i-1}|\rho_i^+|}
=\frac{1-O(\sqrt{\frac{n-i}{n}})}{n(\la^{-1}-\frac{n-i+1}{n})}
=\frac{\la}{n}\left(1+O\left(\sqrt{\tfrac{n-i}{n}}\right)\right).
\eeq
Finally, we have
\beq\label{eqn:1+tau_sq}
(1+\tau_{i-1})^2=\left(1+\frac{m-n+i-1}{|\rho_{i-1}^+|}\cdot\frac{|\rho_{i-1}^-|}{|\rho_{i-1}^-|} \right)^2
=\left(1+\frac{|\rho_{i-1}^-|}{n\left(1-\frac{n-i+2}{n}\right)}\left(1+O(n^{-1})\right)\right)^2.
\eeq
Since $\sigma_n n^{-2/3}\ll \frac{n-i}{n}$, computations in the proof of Lemma \ref{lem:sumA0i} (in particular, \eqref{eqn:sqrtUi} and \eqref{eqn:Vi}) imply that
\beq
|\rho_{i-1}^-|=\frac12\left(\g m-(m-n+2i-3)-\sqrt{U_{i-1}}\right)=
\la^{-\frac12}n\left(1+O\left(\sqrt{\tfrac{n-i}{n}}\right)\right).
\eeq
Thus, by \eqref{eqn:1+tau_sq},
\beq
(1+\tau_{i-1})^2
=\left(\la^{-\frac12}+1\right)^2+O\left(\sqrt{\tfrac{n-i}{n}}\right).
\eeq
Putting all factors of \eqref{eqn:prod_gX} together, we conclude that, for $i<n-n^{1/3}\sigma_n\log n $,
\beq\label{eq:Xvarspecialized}
\EE X_i^2
=2\a n^{-1}(1+\la^{1/2})^2\left(1+O\left(\sqrt{\tfrac{n-i}{n}}\right)\right).
\eeq
We use the notation $C_0=\frac43(1+\lambda^{1/2})$ and $C_1=2 (\lambda^{1/2}+1)^2$ and combine \eqref{eq:wiproductcompute} and \eqref{eq:Xvarspecialized} to obtain  
\beq\label{eq:sum_to_approx_with_Gamma}\begin{split}
&\sum_{i=n-n(\log n)^{-2}}^j(\omega_{i+3}\cdots\omega_{j+2})^2
\EE X_{i+2}^2=\\
& \sum_{i=n-n(\log n)^{-2}}^j 
\exp\left(-2C_0n\left[\left(\tfrac{n-i-3}{n}\right)^{\frac32}(1+\tilde{s}(\tfrac{n-i-3}{n}))-\left(\tfrac{n-j-2}{n}\right)^{\frac32}(1+\tilde{s}(\tfrac{n-j-2}{n}))\right]  \right)\frac{\alpha C_1+O((\tfrac{n-i}{n})^{\frac12})}{n}.
\end{split}\eeq
This gives us
\beq\begin{split}
&\sum_{i=n-n(\log n)^{-2}}^j(\omega_{i+3}\cdots\omega_{j+2})^2
\EE X_{i+2}^2\\
&=\alpha C_1\left(1+O((\log n)^{-1})\right)\int^{(\log n)^{-2}}_{1-\frac{j}{n}}
\exp\left(-2C_0n\left[x^{\frac32}(1+\tilde{s}(x))-(1-\tfrac{j}{n})^{\frac32}(1+\tilde{s}(1-\tfrac{j}{n}))\right]\right)dx.
\end{split}\eeq
Next, we make the change of variables $u=x(1+\tilde{s}(x))^{2/3}(2C_0n)^{2/3}$.  Noting that $du=(2C_0n)^{2/3}(1+O(\tilde{s}(x)))dx$, the right hand side of the display above becomes
\beq\begin{split}
&\frac{\alpha C_1(1+O((\log n)^{-1}))}{(2C_0n)^{2/3}}\exp\left(2C_0n(1-\tfrac{j}{n})^{\frac32}(1+\tilde{s}(1-\tfrac{j}{n}))\right)
\int_{(2C_0n)^{\frac23}(1-\frac{j}{n})(1+\tilde{s}(1-\frac{j}{n}))^{\frac23}}
^{(2C_0n)^{\frac23}(\log n)^{-2}
(1+\tilde{s}((\log n)^{-2}))^{\frac23}
}
\exp(-u^{3/2})du.
\end{split}\eeq
The integrand $\exp(-u^{3/2})$ has antiderivative $-\frac23\Gamma\left(\frac23,u^{3/2}\right)$.  Furthermore, the asymptotics of the incomplete Gamma function (see Digital Library of Mathematical Functions 8.11.2) are
\beq
\Gamma(a,z)=z^{a-1}e^{-z}(1+O(z^{-1}))\quad\text{for fixed }a\text{ and }z\to\infty.
\eeq
Applying this to the preceding equation, we get 
\beq\label{eq:A4_dominant_inner_sum}\begin{split}
\sum_{i=n-n(\log n)^{-2}}^j(\omega_{i+3}\cdots\omega_{j+2})^2
\EE X_{i+2}^2
= \frac{\alpha C_1}{3C_0n} (1-\tfrac{j}{n})^{-1/2}(1+O((\log n)^{-1})).
\end{split}\eeq
It remains to calculate $\w_{j+3}g_{j+4}=\sum_{k=j}^{n-3}\w_{j+3}\cdots\w_{k+3}$ and then compute the outer sum in the expression \eqref{eq:A4sum_brokenup}.  Using Lemmas \ref{lem:gi_bounds} and \ref{lem:wi} for indices $j$ in our desired range, we have the lower bound
\beq\label{eq:S1_lowerbound}
\w_{j+3}g_{j+4}\geq\frac{\w_{j+3}(1+(\log n)^{-1})}{1-\w_{j+4}}=\frac{1+O\left((\log n)^{-1}+(\frac{n-(j+4)}{n})^{\frac12}\right)}{2(1+\la^{1/2})\left(\frac{n-(j+4)}{n}\right)^{1/2}}
=\frac{1+O\left((\log n)^{-1}+(\frac{n-j}{n})^{\frac12}\right)}{\frac32 C_0\left(\frac{n-j}{n}\right)^{1/2}}.
\eeq
The analogous upper bound obtained from Lemma \ref{lem:gi_bounds} is not tight enough.  Instead, we upper bound $\w_{j+3}g_{j+4}$ by rewriting it as two sums
\beq
\w_{j+3}g_{j+4}=\sum_{k=j}^{n-n^{1/3}\sigma_n}\w_{j+3}\cdots\w_{k+3}+\sum_{k=n-n^{1/3}\sigma_n+1}^{n-3}\w_{j+3}\cdots\w_{k+3}=:S_1+S_2.
\eeq
We will show that $S_1\leq\frac{1+O\left((\frac{n-j}{n})^{\frac12}\right)}{\frac32 C_0\left(\frac{n-j}{n}\right)^{1/2}}$ while $S_2=o(1)$.
For $S_2$, we have
\beq\label{eq:S2bound}\begin{split}
S_2
&=(\w_{j+3}\cdots\w_{n-n^{1/3}\sigma_n+4})\left(1+ \sum_{k=n-n^{1/3}\sigma_n+2}^{n-3}\w_{n-n^{1/3}\sigma_n+5}\cdots\w_{k+3}\right) \\
&=(\w_{j+3}\cdots\w_{n-n^{1/3}\sigma_n+4})g_{n-n^{1/3}\sigma_n+5}\\
&< (\w_{n-n^{1/3}\sigma_n\log n }\cdots\w_{n-\frac12 n^{1/3}\sigma_n\log n })g_{n-n^{1/3}\sigma_n +5}.
\end{split}\eeq
Equation \eqref{eq:wiproductcompute} implies that, for some $C>0$, the product above has the bound
\beq
\w_{n-n^{1/3}\sigma_n\log n }\cdots\w_{n-\frac12 n^{1/3}\sigma_n\log n }\leq \exp(-C(\sigma_n\log n )^{3/2}).
 \eeq
  Meanwhile, Lemmas \ref{lem:gi_bounds} and \ref{lem:wi} imply that $g_{n-n^{1/3}\sigma_n +5}=O(n^{1/3}\sigma_n^{-1/2})$.  Thus, we conclude that $S_2=o(1)$.
  
  For $S_1$, we will make use of the the asymptotic in \eqref{eq:wiproductcompute} to bound the product $\w_{j+3}\cdots\w_{k+3}$.  Although \eqref{eq:wiproductcompute} is only a valid asymptotic expression for indices satisfying $n-i\gg n^{1/3}\sigma_n$, it is nevertheless a valid upper bound for all indices in the range covered by $S_1$.  This is because \eqref{eq:wiproductcompute} was obtained using the approximation of $\w_i$ in Lemma \ref{lem:wi}\eqref{lem:wi(3)} which is an upper bound for the approximation in Lemma \ref{lem:wi}\eqref{lem:wi(2)} when $n-i=\Theta(n^{1/3}\sigma_n$).  Thus we obtain
  \beq
  S_1\leq\sum_{k=j}^{n-n^{1/3}\sigma_n}\exp\left(-C_0n\left[\left(\tfrac{n-j-3}{n}\right)^{\frac32}(1+\tilde{s}(\tfrac{n-j-3}{n}))-\left(\tfrac{n-k-3}{n}\right)^{\frac32}(1+\tilde{s}(\tfrac{n-k-3}{n}))\right] \right)(1+O(\tfrac{n-j-3}{n})^{\frac32}).
  \eeq
  The asymptotics of this sum can be computed using a similar approach to the computation of the sum in \eqref{eq:sum_to_approx_with_Gamma}.  Using this method, we arrive at
  \beq
  S_1\leq\frac{2}{3C_0}\left(\frac{n-j-3}{n}\right)^{-1/2} (1+O((\tfrac{n-j-3}{n})^{1/2})).
  \eeq
  Combining this with \eqref{eq:S1_lowerbound} and \eqref{eq:S2bound} along with the fact that $j\geq n-n(\log n)^{-2}$, we conclude that
  \beq\label{eq:A4_dominant_non-i}
  \w_{j+3}g_{j+4}=\frac{2}{3C_0}\left(\frac{n-j}{n}\right)^{-1/2}(1+O((\log n)^{-1})).
  \eeq
Finally, plugging the results from \eqref{eq:A4_dominant_inner_sum} and \eqref{eq:A4_dominant_non-i} into the summation \eqref{eq:A4sum_brokenup}, we get
\beq\begin{split}
&\sum_{j=n-n(\log n)^{-2}}^{n-n^{1/3}\sigma_n\log n }\;
\sum_{i=n-n(\log n)^{-2}}^j(\omega_{i+3}\cdots\omega_{j+2})^2
\EE X_{i+2}^2
\cdot\omega_{j+3}g_{j+4}\\
&=\sum_{j=n-n(\log n)^{-2}}^{n-n^{1/3}\sigma_n\log n }\frac{\alpha C_1}{3C_0 n}(1-\tfrac{j}{n})^{-1/2}\cdot\frac{2}{3C_0}(1-\tfrac{j}{n})^{-1/2}(1+O((\log n)^{-1}))\\
&=\frac{2\alpha C_1}{9C_0^2}\int_{n^{-2/3}\sigma_n\log n }^{(\log n)^{-2}}z^{-1}dz(1+O((\log n)^{-1}))
=\frac{2\alpha C_1}{9C_0^2}\left(\frac23\log n+O(\log\log n)\right)=\frac\a6 \log n+O(\log\log n).
\end{split}\eeq
Next, we must consider the terms in \eqref{eq:A4sum_indexreverse} whose indices do not satisfy $n-n(\log n)^{-2}<i\leq j<n-n^{1/3}\sigma_n\log n$ and we must show that the sum over those indices is $O(\log\log n)$.  More specifically, we will show that 
\begin{enumerate}[(a)]
\item 
$\displaystyle\sum_{j=n-n^{1/3}\sigma_n\log n}^{n-3}\;\sum_{i=1}^j 
(\omega_{i+3} \dots \omega_{j+2})^2\omega_{j+3}g_{j+4} \bE X_{i+2}^2=O(\log\log n),$
\item
$\displaystyle\sum_{j=1}^{n-n(\log n)^{-2}}\sum_{i=1}^j
(\omega_{i+3} \dots \omega_{j+2})^2\omega_{j+3}g_{j+4} \bE X_{i+2}^2=O(\log\log n),$
\item
$\displaystyle\sum^{n-n^{1/3}\sigma_n\log n}_{j=n-n(\log n)^{-2}}\;\sum_{i=1}^{n-n(\log n)^{-2}}(\omega_{i+3} \dots \omega_{j+2})^2\omega_{j+3}g_{j+4} \bE X_{i+2}^2=O(1).$
\end{enumerate}
To prove (a), we begin by using Lemma \ref{lem:gi_bounds} and the fact that $\EE X_i^2=O(n^{-1})$ to observe that, for some constant $C$,
\beq\label{eq:inner_sum_bound}
\sum_{i=1}^j (\omega_{i+3} \dots \omega_{j+2})^2\omega_{j+3}g_{j+4} \bE X_{i+2}^2\leq\sum_{i=1}^j \w_{j+2}^{2(j-i)}\cdot \frac{C}{n(1-\w_{j+3})}
\leq \frac{1}{1-\w_{j+2}^2}\cdot \frac{C}{n(1-\w_{j+3})}.
\eeq
Using Lemma \ref{lem:wi} and the fact that $\w_j$ is increasing in $j$, we conclude
\beq\begin{split}
&\sum_{j=n-n^{1/3}\sigma_n\log n}^{n-3}\;\sum_{i=1}^j 
(\omega_{i+3} \dots \omega_{j+2})^2\omega_{j+3}g_{j+4} \bE X_{i+2}^2\\
&\leq \sum_{j=n-n^{1/3}\sigma_n\log n}^{n} 
\frac{1}{1-\w_{j}^2}\cdot \frac{C}{n(1-\w_{j})}
\leq  \sum_{j=n-n^{1/3}\sigma_n\log n}^{n-n^{1/3}\sigma_n}\frac{C}{n(1-\omega_j)^2} 
+\sum_{j=n-n^{1/3}\sigma_n+1}^{n}\frac{C}{n(1-\omega_j)^2}\\
&=O\left(\int_{n^{1/3}\sigma_n}^{n^{1/3}\sigma_n\log n}x^{-1}dx\right)+O(1)
=O(\log\log n).
\end{split}\eeq
To prove (b) we observe that inequality \eqref{eq:inner_sum_bound} still holds.  Using this, we obtain the following result, where $C_1$ is the constant that comes from applying Corollary \ref{cor:wi}:
\beq\begin{split}
&\sum_{j=1}^{n-n(\log n)^{-2}}\sum_{i=1}^j
(\omega_{i+3} \dots \omega_{j+2})^2\omega_{j+3}g_{j+4} \bE X_{i+2}^2
\leq \sum_{j=1}^{n-n(\log n)^{-2}} \frac{1}{1-\w_{j+2}^2}\cdot\frac{C}{n(1-\omega_{j+4})}\\
&\leq \sum_{j=1}^{n-n(\log n)^{-2}}\frac{C}{n\cdot C_1^2\left(\frac{n-(j+4)}{n}\right)}
=O\left(\int_{(\log n)^{-2}}^1\frac1x dx \right)=O(\log\log n).
\end{split}\eeq
To prove (c) we observe that, on the indices we consider,
\beq\begin{split}
\left(\w_{i+3}\cdots\w_{j+2}\right)^2&=\left(\w_{i+3}\cdots\w_{\lfloor n-n(\log n)^{-2}\rfloor}\cdots\w_{j+2}\right)^2\\
&<\left(\w_{i+3}\cdots\w_{\lfloor n-n(\log n)^{-2}\rfloor}\right)^2
<\left(\w_{\lfloor n-n(\log n)^{-2}\rfloor}^2\right)^{\lfloor n-n(\log n)^{-2}\rfloor-(i+2)}.
\end{split}\eeq
Using this, along with Lemma \ref{lem:wi} and Corollary \ref{cor:wi}, we conclude
\beq\begin{split}
&\sum^{n-n^{1/3}\sigma_n\log n}_{j=n-n(\log n)^{-2}}\;\sum_{i=1}^{n-n(\log n)^{-2}}(\omega_{i+3} \dots \omega_{j+2})^2\omega_{j+3}g_{j+4} \bE X_{i+2}^2\\
&\leq \sum^{n-n^{1/3}\sigma_n\log n}_{j=n-n(\log n)^{-2}}\;\sum_{i=1}^{n-n(\log n)^{-2}}
\left(\w_{\lfloor n-n(\log n)^{-2}\rfloor}^2\right)^{\lfloor n-n(\log n)^{-2}\rfloor-(i+2)}\frac{1}{1-\w_{j+4}}\cdot O\left(\frac{1}{n}\right)\\
&\leq  \sum^{n-n^{1/3}\sigma_n\log n}_{j=n-n(\log n)^{-2}}
\frac{1}{1-\w_{\lfloor n-n(\log n)^{-2}\rfloor}^2}\cdot \frac{1}{(\frac{n-j}{n})^{1/2}}\cdot O\left(\frac{1}{n}\right).
\end{split}\eeq
To simplify this we use the fact that $\frac{1}{1-\w_{\lfloor n-n(\log n)^{-2}\rfloor}^2}=O(\log n)$ and we rewrite the summation as an integral, so the entire expression above becomes
\beq
O(\log n)\cdot \int_{n^{-2/3}\sigma_n\log n}^{(\log n)^{-2}}\frac{1}{x^{1/2}}dx
=O(1).
\eeq

\section{CLT for \texorpdfstring{$\sum_{i = 3}^n L_i$}{sum Li}}\label{sec:A.1}
From \eqref{eqn:Lj_full} and Definition \ref{defn:gi},
\beqq
\sum_{i = 3}^n L_i = \sum_{i = 3}^n g_{i+1}X_i + \sum_{i = 3}^n \a_i - g_3\a_2.
\eeqq
In this section, we show that $\sum_{i = 3}^n L_i$ satisfies the CLT
\beq
\frac{\sum_{i = 3}^n L_i}{\left(\sum_{i = 3}^n g_{i+1}^2\bE X^2_i\right)^{1/2}} \stackrel{d}{\to} \mathcal{N}(0,1)
\eeq
by showing the following claims.
\begin{enumerate}
	\item The mean-zero random variable $\sum_{i = 3}^n g_{i+1}X_i$ satisfies Lyapunov condition 
	\beq\label{eqn:Lyapunov}
	\frac{\sum_{i = 3}^n g^4_{i+1}\bE X^4_i}{\left(\sum_{i = 3}^n g_{i+1}^2 \bE X^2_i\right)^2} \to 0 \quad \text{ as } n \to \infty.
	\eeq
	\item $\frac{ \sum_{i = 3}^n \a_i - g_3\a_2}{\left(\sum_{i = 3}^n g_{i+1}^2\bE X^2_i\right)^{1/2}}$ converges to 0 in probability.
\end{enumerate} 
Here, knowing the order of the variance $\sum_{i = 3}^n g_{i+1}^2\bE X^2_i$ is sufficient for both claims, so we delay the computation of its leading term to the end of the section.

We now verify \eqref{eqn:Lyapunov}. It follows from \eqref{eqn:varXi} and   $|\rho_{i}^+| = \Theta(n)$ uniformly in $i$ that
\beq\label{eqn:bE}
\bE X_i^2 
 \geq \frac{C\a\delta_i}{n}.
\eeq 
Meanwhile, the uniform lower bound $|\rho_{i}^+| = \Omega(n)$ and Lemma \ref{lem:SG_boucheron} with $p=4$ imply
\beq
\bE X_i^4 \leq C\a^4(1+\tau_{i-1})^4\left(\frac{\delta_i^2}{|\rho_{i}^+|^2}+\frac1{|\rho_{i}^+|^4}\right)=O\left(\frac{\a^4\delta_i^2}{n^2}\right).
\eeq
We also need to estimate powers of $g_i$ for $3\leq i \leq n$. The following lemma shows that for most indices $i$, $g_i$ is of order $(1 -\w_i)^{-1}$. 
\begin{lemma}\label{lem:gi_bounds}
	Let $\{g_i\}_{i=3}^{n+1}$ be as above (see Definition \ref{defn:gi}). Then, 
	\begin{enumerate}[(i)]
		\item\label{lem:gi_bounds_lower} for any $k>0$ and sufficiently large $n$, $g_i > \frac{1 - \log^{-k}n}{1- \w_i}$ for all $3 \leq i \leq n -n^{1/3}$,
		\item\label{lem:gi_bounds_upper} for sufficiently large $n$, $g_i < \frac{1 + \sigma_n^{-3/2}}{1 - \w_i}$ for all $3 \leq i\leq n$.
	\end{enumerate}
\end{lemma}
\begin{proof}
	Fix $k>0$. Since $\{\w_i\}_{i=3}^n$ is increasing in $i$,
	\beq
	g_i > 1+\w_i+\w_i^2+\dots +\w_i^{n-i+1} = \frac{1 - \w_i^{n-i+2}}{1-\w_i}.
	\eeq 
	By Corollary \ref{cor:wi}, $\w_i \leq 1 - cn^{-1/3}\sigma_n^{1/2}$ for $3\leq i \leq n$. If in addition, $i \leq n -n^{1/3}$, then 
	\beq
	\w_i^{n-i+2} \leq \left(1 - cn^{-1/3}\sigma_n^{1/2}\right)^{n^{1/3}} < e^{-c\sigma_n^{1/2}}.
	\eeq
	The right hand side is less than $e^{-k \log\log n} = \log^{-k}n$ for sufficiently large $n$, so we obtain \eqref{lem:gi_bounds_lower}.
	
	For the upper bound, define $N_i = (1-\w_i)g_i - 1-\sigma_n^{-3/2}$. Then $g_i=\frac{N_i+1+\sigma_n^{-3/2}}{1-\w_i}$ and it suffices to show $N_i<0 $ for every $3 \leq i \leq n$. The case $i=n$ is clear, as
	\beq
	N_n = (1 - \w_n)(1+\w_n) - 1-\sigma_n^{-3/2} = -\w_n^2 - \sigma_n^{-3/2}<0.
	\eeq
	Suppose $N_i <0$. From the definition of $g_i$,
	\beq
	\begin{split}
		N_{i-1} 
		&= (1-\omega_{i-1})(1+\omega_{i-1}g_i) - 1 - \sigma_n^{-3/2}\\
		&= (1-\omega_{i-1}) \left(1 + \frac{\omega_{i-1}}{1 -\omega_{i}}(N_i + 1+\sigma_n^{-3/2})\right)- 1 - \sigma_n^{-3/2}.	
	\end{split}
	\eeq 
	By the induction hypothesis, 
	\beq
	\begin{split}
	N_{i-1} &<-\omega_{i-1} - \sigma_n^{-3/2} + \frac{\omega_{i-1}(1-\omega_{i-1})}{1 -\omega_{i}}( 1+\sigma_n^{-3/2})\\
	&=\frac{(\omega_{i-1} + \sigma_n^{-3/2})(\omega_i  -\omega_{i-1}) - (1-\omega_{i-1})^2 \sigma_n^{-3/2} }{1 - \omega_i}\\
	&<\frac{(\omega_{i-1} + \sigma_n^{-3/2})(\omega_i  -\omega_{i-1}) - (1-\omega_{n})^2 \sigma_n^{-3/2} }{1 - \omega_i}.
	\end{split}
	\eeq
	Note that $\omega_{i-1} + \sigma_n^{-3/2}\leq 2$. We then provide bounds for $\omega_i  -\omega_{i-1}$ and $1-\omega_{n}$, in order to show that the numerator is negative. We approach the first quantity by bounding the growth of $\w_i$, where $\w_i$ is considered as function of $i/m \in [3/m,\la]$. For brevity of the presentation, we define for $x \in [3/m,\la]$,
	\beq\label{eqn:fg}
	f(x) = \frac{x\left(\frac{m-n}{m} +x\right)}{\left(C_n - x \right)^2} \quad \text{and} \quad g(x)= \frac{1 -\sqrt{1 - f(x)}}{1 + \sqrt{1 - f(x - \frac{1}{m})}},
	\eeq
	where $C_n = \frac{\gamma m -(m-n+1)}{2m}$. Then $\frac{|\rho_i^\pm|}{m} 
	= \left(C_n - \frac{i-1}{m}\right) \left(1 \pm \sqrt{1 - f \left(\frac{i-1}{m}\right)}\right)$, which implies
	\begin{align}
		\w_i &= \frac{|\rho_i^-|}{|\rho_{i-1}^+|} 
		= \left(1 - \frac{1/m}{C_n  - \frac{i-2}{m}}\right) g \left(\frac{i-1}{m}\right).
	\end{align}
	Since both $f(x)$ and $f'(x) =  \frac{\frac{m-n}{m} + 2x}{(C_n - x)^2} +\frac{x\left(\frac{m-n}{m} +x\right) }{(C_n - x)^3}$ are increasing in $x$, so is 
	\begin{equation}\label{eqn:g'}
		g'(x) = \frac{\frac{1}{2}f'(x)/ \sqrt{1- f(x)}}{1 + \sqrt{1 - f(x - \frac{1}{m})}} + \frac{(1 -\sqrt{1 - f(x)} ) \frac{1}{2} f'(x-\frac{1}{m}) \big/ \sqrt{1- f(x - \frac{1}{m})} }{\left(1 + \sqrt{1 - f(x - \frac{1}{m})}\right)^2}.
	\end{equation}
	Therefore,
	\beq
	\w_i - \w_{i-1} \leq \left(1-\frac{1/m}{C_n-\frac{i-3}{n}}\right)\int_{\frac{i-2}m}^{\frac{i-1}m} g'(x)dx < \frac1mg'\left(\frac{n-1}m\right).	 
	\eeq
	Set $y_i = 1 -f \left(\frac{i-1}{m}\right)$. From \eqref{eqn:g'}, for sufficiently large $n$, 
	\begin{equation*}
		\begin{split}
			g'\left(\frac{n-1}m\right)	
			&=\frac{f'(\frac{n-1}{m})}{\sqrt{y_n}(1+\sqrt{y_{n-1}})^2} \cdot \frac{1+\sqrt{y_{n-1}} +\sqrt{\frac{y_n}{y_{n-1}}}(1-\sqrt{y_n})}2
			< \frac{f'(\frac{n-1}{m})}{\sqrt{y_n}(1 +\sqrt{y_{n-1}})^2}.
		\end{split}
	\end{equation*}
	
	We verify the above inequality by showing that for sufficiently large $n$,
	\beq\label{claim:yn}
	1 +\sqrt{1-\frac{y_{n-1}-y_n}{y_{n-1}}}+\frac{y_{n-1}-y_n}{\sqrt{y_{n-1}}}<2.
	\eeq
	We note the identity $y_i=\frac{U_i}{\left(\g m-(m-n+2i-1)\right)^2}$ for $i=3,\dots, n$, where $U_i$ is defined in \eqref{defn:Ui}. By \eqref{eqn:Ui_asymp} and the fact $\g=(1+\la^{\frac12})^2+\sigma_nn^{-\frac23}$, there are constants $1<c_1<c_2$ such that $c_1 \sigma_nn^{-\frac23} <y_n<y_{n-1}<c_1 \sigma_nn^{-\frac23}$. In addition, $U_{n-1}$ and $U_n$ are $\Theta(n^{4/3}\sigma_n)$ by \eqref{eqn:Ui_asymp}, and $U_{n-1}-U_n=4(\g m-1)=\Theta(n)$. Thus, 
	\begin{equation}
	y_{n-1}-y_n=\frac{\left(\g m-(m+n-1)\right)^2(U_{n-1}-U_n)-4\left(\g m-(m+n-2)\right)U_n}{\left(\g m-(m+n-1)\right)^2\left(\g m-(m+n-3)\right)^2}=\Theta(n^{-1}).
	\end{equation}
	Therefore, the left hand side of \eqref{claim:yn} has asymptotics $2-c_3n^{-\frac13}\sigma_n^{-1}+O(n^{-\frac23}\sigma_n^{-\frac12})$ as $n\to \infty$ for some $c_3>0$, and we obtain the claim. 
	
	We now consider $f'(\frac{n-1}{m})$. Note that $C_n - \frac{n-1}m = \sqrt{\la} + \frac{1}{2} n^{-1/3}\sigma_n + O(n^{-1})$, so using expression of $f'(x)$ as above, we have
	\beq
	\begin{split}
		f'(\frac{n-1}{m})
		= \frac{\frac{\la^{3/2}}{2}(1+\la^{-1/2})^{2}+\frac{1+\la}{2} n^{-2/3}\sigma_n + O(n^{-1})}{\left(\sqrt{\la} +\frac{1}{2}n^{-2/3} \sigma_n + O(n^{-1}) \right)^3}
		< \frac{1}{\sqrt{2}} (1+\la^{-1/2})^2.
	\end{split} 
	\eeq
	We obtain
	\beq\label{eqn:omega_diff}
	\omega_i - \omega_{i-1} 
	< \frac{1}{m} \cdot \frac{(1+\la^{-1/2})^2/\sqrt{2}}{ \sqrt{y_n}(1 +\sqrt{y_{n-1}})^2}.
	\eeq
	On the other hand,
	\beq\label{eqn:1-om}
	1 - \w_n > 1 - g \left(\frac{n-1}{m}\right) 
	= 1 - \frac{1 -\sqrt{y_n}}{1 + \sqrt{y_{n-1}}} 
	> \frac{2  \sqrt{y_n} 
	}{1 + \sqrt{y_{n-1}}}.
	\eeq
	Displays \eqref{eqn:omega_diff} and \eqref{eqn:1-om} together imply
	\beq
	\begin{split}
		N_{i-1} 
		&< \frac{\frac{1}{m} \cdot \frac{ (1 + \sigma_n^{-3/2}) \frac{1}{\sqrt{2}} (1+\la^{-1/2})^2 }{\sqrt{y_n}(1 + \sqrt{y_{n-1}})^2} - \frac{4 y_n}{(1 + \sqrt{y_{n-1}})^2} \sigma_n^{-3/2}}{1 - \omega_i}\\
		&= \frac{\frac{1}{\sqrt{2}} (1+\la^{-1/2})^2 - 4m y_n^{3/2}\sigma_n^{-3/2} + (1+\la^{-1/2})^2 \sigma_n^{-3/2} }{m(1 - \omega_i) \sqrt{y_n} (1 + \sqrt{y_{n-1}})^2}.
	\end{split}
	\eeq
	Since $y_n> n^{-2/3}\sigma_n$ and $0<\la \leq 1$ for all $n$, the numerator is less than $\frac{1}{\sqrt{2}} (1+\la^{-1/2})^2- 4\la^{-1} + O(\sigma_n^{-3/2})$, 
	which is negative for sufficiently large $n$. Therefore $N_{i-1} <0$ for sufficiently large $n$.
\end{proof}
Combining Lemma \ref{lem:gi_bounds} and Corollary \ref{cor:wi}, we obtain
\beq
\begin{split}
	\frac{\sum_{i = 3}^n g^4_{i+1}\bE X^4_i}{\left(\sum_{i = 3}^n g_{i+1}^2 \bE X^2_i\right)^2} 
	&= O \left(\frac{\sum_{i = 3}^n (1 - \w_{i+1})^{-4}\delta_i^2}{\left(\sum_{i = 3}^{n - n^{1/3}\sigma_n} (1 - \w_{i+1})^{-2}\delta_i \right)^2}\right) \\
	&=O\left(\frac{\sum_{i = 3}^{n-n^{1/3}\sigma_n}\left(\frac{n}{n-i}\right)^2 \left(\frac{i-1}{n}\right)^2+\sum_{i>n-n^{1/3}\sigma_n}(n^{1/3}\sigma_n^{-1/2})^4\left(\frac{i-1}{n}\right)^2}{\left(\sum_{i = 3}^{n-n^{1/3}\sigma_n}\frac{n}{n-i}\cdot \frac{i-1}{n}\right)^2}\right)\\
	&= o(n^{-1/3}).
\end{split}
\eeq
Therefore, Lyapunov condition \eqref{eqn:Lyapunov} holds.

The above computations suggest the variance $\sum_{i = 3}^n g_{i+1}^2\bE X^2_i$ is increasing in $n$, with lower bound $C \log n$ for some constant $C>0$. As $\sum_{i = 3}^n \a_i - g_3\a_2$ belongs to some sub-gamma family $\SG(v,u)$, Lemma \ref{lem:SG_boucheron} with $t =\sqrt{\log n}$ implies 
\beq
\PP\left(\left|\sum_{i = 3}^n \a_i - g_3\a_2\right| > \sqrt{2vt}+ut\right)\leq 2n^{-1/2}.
\eeq
Hence, the claim on convergence to zero in probability holds as long as the parameters $v$, $u$ satisfy $v = o(\sqrt{\log n})$ and $u = o(1)$. Indeed, by Lemma \ref{lem:SG},
\beq
v = \a \left(\frac{g_3^2 \tau_2}{|\rho_2^+|} +\sum_{i=3}^n \frac{\tau_i}{|\rho_i^+|} \right) = O(1), \quad u = \a \max \left\{\frac{g_3}{|\rho_2^+|}, \frac1{|\rho_i^+|}: 3 \leq i \leq n \right\} = O(n^{-1}).
\eeq 

We now provide the asymptotics for $\sum_{i = 3}^n g_{i+1}^2\bE X^2_i$. We will show that dominant contribution to the sum comes from indices $i\leq n -n^{1/3}\sigma_n \sqrt{\log n}$, while the sum over the remaining indices is at most order $\sqrt{\log n}$. 
\begin{lemma}\label{lem:variance_leadingterm}
	\beq
	\sum_{i = 3}^n g_{i+1}^2\bE X^2_i = \frac{\a}{3} \log n +o(\log n).
	\eeq
\end{lemma}
\begin{proof}
We begin by showing that the terms with indices $n-n^{1/3}\sigma_n\sqrt{\log n}\leq i\leq n-n^{1/3}\sigma_n$ and $n-n^{1/3}\sigma_n\leq i\leq n$, together, contribute only $O(\sqrt{\log n})$ to the sum. In these calculations, we use the fact that $\EE X_i^2=O(n^{-1})$ uniformly in $i$ and we bound $g_i$ using Lemma \ref{lem:gi_bounds}\eqref{lem:gi_bounds_upper} and Corollary \ref{cor:wi}. In particular, we obtain
\beq\begin{split}
\sum_{i=n-n^{1/3}\sigma_n\sqrt{\log n}}^{n-n^{1/3}\sigma_n}g_{i+1}^2\EE X_i^2
&=O\left(\sum_{i=n-n^{1/3}\sigma_n\sqrt{\log n}}^{n-n^{1/3}\sigma_n}\frac{n}{n-i}\cdot \frac1n\right)
=O(\sqrt{\log n}),\\
\sum_{i=n-n^{1/3}\sigma_n}^{n}g_{i+1}^2\EE X_i^2
&=O\left(\sum_{i=n-n^{1/3}\sigma_n}^n \frac{n^{2/3}}{\sigma_n}\cdot\frac1n\right)
=O(1).
\end{split}\eeq
We now compute the sum over the indices $i<n-n^{1/3}\sigma_n\sqrt{\log n}$. 
Using Lemma \ref{lem:gi_bounds}, we obtain
\begin{align}
g_{i+1}^2&= \frac{1}{(1-\omega_{i+1})^2 }(1+o(1))
= \frac{1}{4(1+\la^{1/2})^2}\left(\frac{n-i}{n}\right)^{-1}(1+o(1)).
\end{align}
Combining this with the computation of $\EE X_i^2$ in \eqref{eq:Xvarspecialized}, we get
\beq\begin{split}
g_{i+1}^2\EE X_i^2
&=\frac1n\left(\frac{n-i}{n}\right)^{-1}\left[\frac{\a}{2}
+o(1)\right].
\end{split}\eeq
Therefore, 
\beq\begin{split}
\sum_{i=3}^{n-n^{1/3}\sigma_n\sqrt{\log n}}g_{i+1}^2\EE X_i^2
&=\int_{n^{-2/3}\sigma_n\sqrt{\log n}}^1\frac1x\left[\frac\a2+o(1)\right]dx+O\left(\frac1n\cdot\frac{1}{n^{-2/3}\sigma_n\sqrt{\log n}}\right)\\
&=\frac\a2\cdot\frac23\log n(1+o(1)).
\end{split}\eeq
Since the indices $i>n-n^{1/3}\sigma_n\sqrt{\log n}$ only contribute $O(\sqrt{\log n})$ to the sum, the lemma is proved.
\end{proof}

\section{Proof of Lemma \ref{lem:unif_Ri}: Uniform bounds for \texorpdfstring{$R_i$}{Ri}} \label{sec:unif}

Rather than working directly with $\{R_i\}_{i=3}^n$, we consider the alternative process $\{\tR_i\}_{3 \leq i \leq n}$ given below.
 
Let $\brR_i = \phi_{n^{-1/3}/2}(R_i)$, where $\phi_u$ for $u>0$ is given by 
\beqq
\phi_u(x) = \begin{cases}
	x, &\quad |x|\leq u,\\
	\frac{x}{|x|}u, &\quad |x|> u.\\
\end{cases}
\eeqq
We also set
\beqq
\brR^{(1)}_i = \frac{\brR_{i-1}}{1 - \brR_{i-1}}, \quad \brR^{(2)}_i = \omega_i \frac{\brR_{i-1}^3}{1 - \brR_{i-1}}, \quad \brR^{(3)}_i = \omega_i\brR_{i-1}^2.
\eeqq
Consider the process
\begin{align}
	\tR_2&=R_2,\\
	\tR_i &= L_i + \omega_{i} \dots \omega_3 R_2 - A_{0i} + \tB_{0i} + \tB_{1i} +\tB_{2i} +\tB_{3i}, \quad 3 \leq i\leq n, \label{eqn:decomp_tR}
\end{align}
where \beqq
A_{0i}=\g_i - \w_i + \w_i(\g_{i-1} - \w_{i-1}) +\dots+\w_i\dots\w_4(\g_3 - \w_3),
\eeqq
and
\begin{align*}
	\tB_{0i} &= \left(\alpha_{i-1}+ ( \tau_{i-1} + \alpha_{i-1} ) \brR^{(1)}_i\right) \beta_i  + \omega_i \left(\alpha_{i-2} +( \tau_{i-2} + \alpha_{i-2} )\brR^{(1)}_{i-1}\right) \beta_{i-1} \\
	&\quad + \dots + \omega_i\dots \omega_4 \left(\alpha_2  + (\tau_2  + \alpha_2 ) \brR^{(1)}_3\right) \beta_3, \quad  \\
	\tB_{1i} &= \alpha_{i-1}\delta_i \brR^{(1)}_i + \omega_i\alpha_{i-2}\delta_{i-1}\brR^{(1)}_{i-1} + \dots + \omega_i\dots \omega_4\alpha_2 \delta_3 \brR^{(1)}_3,\\
	\tB_{2i} &= \brR^{(2)}_i + \omega_i\brR^{(2)}_{i-1} + \dots + \omega_i\dots \omega_4 \brR^{(2)}_3,\\
	\tB_{3i} &= \brR^{(3)}_i + \omega_i\brR^{(3)}_{i-1} + \dots + \omega_i\dots \omega_4 \brR^{(3)}_3.
\end{align*}
On the event $\max_{2 \leq i \leq n} |\tR_i| = o(n^{-1/3})$, observe that $\brR_i =\tR_i$ for all $i \geq 2$. In particular, $|\tR_2|\leq n^{-1/3}/2$, and $\brR_2 = \tR_2 = R_2$. This implies $\brR_3^{(\ell)} = R_3^{(\ell)}$ for $\ell = 1,2,3$. As a result, $\tR_3 = R_3$, which induces $\tR_4 = R_4$ and so on. Therefore, by showing that
\beq\label{goal:tR}
\max_{2 \leq i \leq n} |\tR_i| = o(n^{-1/3}), \quad \text{with probability }  1 - O(\log^{-5}n),
\eeq
we will obtain Lemma \ref{lem:unif_Ri}. 

We check \eqref{goal:tR} by showing, uniformly in $i$, each term in the decomposition \eqref{eqn:decomp_tR} is sufficiently small with probability $1-O(\log^{-5}n)$. First, we have $\g_i-\w_i= O\left(\frac1{n(1-\w_i)}\right)$ by Lemma \ref{lem:gi-wi}, and $(1-\w_i)^{-1} = O(n^{1/3}\sigma_n^{-1/2})$ uniformly in $i$ by Corollary \ref{cor:wi}. Thus,
\beq\label{eqn:A0i}
A_{0i}<\frac{\max_{3\leq j\leq i}(\g_j-\w_j)}{1-\w_i}=O\left(\frac1{n(1-\w_i)^2}\right) = o(n^{-1/3}).
\eeq
At the same time, a direct computation shows that
\beq\label{eqn:R2}
\begin{split}
R_2 
&= 1 + \a_2 + \tau_2 + \b_2+\delta_2 - \frac{\g m}{|\rho_2^+|} - \frac{(\a_1 +\tau_1)(\b_2 + \delta_2)}{\a_1 - \frac{\g m - (m-n+1)}{|\rho_1^+|}}\\
&= \w_2-\g_2 + \a_2 + \left(1 + \frac{\a_1+ \tau_1}{1 - \a_1}\right)\b_2  +\frac{\a_1+\tau_1}{1 - \a_1}\cdot \frac{1}{|\rho_2^+|},
\end{split}
\eeq
where in the last equality, we apply the identity
\beqq
\tau_i+\d_i(1+\tau_{i-1})+1-\tfrac{\gamma m}{|\rho_i^+|} = -(\g_i-\w_i).
\eeqq
By Lemma \ref{lem:gi-wi} and Corollary \ref{cor:wi}, $|\w_2-\g_2| = O\left(\frac1{n(1-\w_2)}\right)=O(n^{-1})$. Moreover, by Lemma \ref{lem:SG_boucheron},\\ $|\a_1| = O(n^{-1/2}\log^{1/2} n)$ with probability $1 - O(n^{-1})$ and $|\a_2|\vee|\b_2| = O(n^{-1/2}\log^{1/2} n)$ with probability $1 - O(n^{-1})$. Therefore,
\beq
|R_2| = O(n^{-1/2}\log^{1/2} n), \quad \text{with probability } 1-O(n^{-1}).
\eeq
We now show uniform bounds for all four sequences $\{\tB_{ji}\}$, $0\leq j \leq 3$ in Subection \ref{subsec:bound_tB}. The uniform bound for $L_i$ is provided in Subsection \ref{subsec:bound_Li}. Throughout these subsections, all the Big-O bounds are uniformly in $i$, for $3\leq i \leq n$.

\subsection{Uniform bound for \texorpdfstring{$\tB_{ji}, 0\leq j\leq3$}{tBji}}\label{subsec:bound_tB}
For fixed $i$, $\tB_{0i}$ is a sum of random variables
\beq
Z_j:= \w_{j+1}\dots\w_i\left(\a_{j-1}+(\tau_{j-1}+\a_{j-1})\brR_i^{(1)}\right)\b_j, \quad 3\leq j\leq i.  
\eeq
Let $\cF_i$ be the $\sigma-$algebra generated by $\a_1,\b_1,\dots,\a_i,\b_i$. Observe that $\brR_i^{(1)}$ is $\cF_{i-1}$-measurable and $\bE[Z_j\vert\cF_{j-1}]=0$ a.s. for all $j$. By Theorem 2.1 of \cite{Rio2009} (Marcinkiewicz--Zygmund type inequality), for any integer $p>2$, 
\beq\label{eqn:Rio}
\|\tB_{0i}\|_p^2\leq(p-1)(\|Z_i\|_p^2+\|Z_{i-1}\|_p^2+\dots+\|Z_{3}\|_p^2).  
\eeq
By Lemma \ref{lem:SG_boucheron}, there exists absolute constant $C>0$ such that for all integers $p>2$ and all $3 \leq j\leq n$,
\begin{align*}
	\|\a_{j-1}\|_p<C\a pn^{-1/2}\quad \text{and} \quad \|\b_{j}\|_p<C\a pn^{-1/2}.
\end{align*}
Also, $|\brR_i^{(1)}|\leq n^{-1/3}$. Hence, for $p = \lfloor2\log n\rfloor$,
\begin{align*}
	\|(\a_{j-1}+(\tau_{j-1}+\a_{j-1})\bR_i^{(1)})\b_j\|_p^2 
	&=\|\a_{j-1}+(\tau_{j-1}+\a_{j-1})\bR_i^{(1)}\|_p^2 \|\b_j\|_p^2
	\\&
	\leq\left(\|\a_{j-1}\|_p+n^{-1/3}(\tau_{j-1}+\|\a_{j-1}\|_p)\right)^2\|\b_j\|_p^2
	\leq C\a^2p^2n^{-5/3}.
\end{align*}
From \eqref{eqn:Rio},
\beq
\|\tB_{0i}\|_p^2< \frac{C(p-1)p^2\a^2n^{-5/3}}{1 -\w_i^2} \leq 8C\a^2n^{-4/3}\sigma_n^{-1/2}\log^{3} n.
\eeq
Apply Markov's inequality and take union bound, we obtain that with probability at least $1 - \frac{1}{n}$,
\beq
|\tB_{0i}|\leq e\|\tB_{0i}\|_{2\log n} = o(n^{-1/2}) \quad \text{for every } 3\leq i\leq n.
\eeq
Since $\bE[\a_{i-1}\delta_i\brR_i^{(1)}\vert\cF_{i-1}]=\a_{i-1}\delta_i\brR_i^{(1)}$ for all $i\leq n$, which is nonzero with positive probability, we cannot apply Theorem 2.1 of \cite{Rio2009} to bound $|\tB_{1i}|$. Instead, we use Minkowski's inequality. Let $p = 2 \log n$ as before. 
\beq
\begin{split}
	\|\tB_{1i}\|_p 
	&\leq n^{-1/3}(\delta_i\|\a_{i-1}\|_p + \sum_{j=3}^{i-1} \w_{j+1} \dots \w_{i} \delta_j\|\a_{j-1}\|_p )\\
	&< \frac{\delta_{i}n^{-1/3}}{1 - \w_i} \max_{2 \leq j \leq i-1} \|\a_{j}\|_p <C\a p n^{-1/2}\sigma_n^{-1/2} = O(n^{-1/2}\sigma_n^{-1/2}\log n).
\end{split}
\eeq
Thus, with probability at least $1 - \frac{1}{n}$,
\beq
|\tB_{1i}|\leq e\|\tB_{1i}\|_{2\log n} = O(n^{-1/2}\sigma_n^{-1/2}\log n) \quad \text{for every } 3\leq i\leq n.
\eeq
Lastly, $|\brR_i^{(2)}| = \omega_{i} \frac{|\brR_{i-1}^3|}{|1 - \brR_{i-1}|} \leq n^{-1}$, $|\brR_i^{(3)}| =  \omega_i \brR_{i-1}^2 \leq n^{-2/3}$, so uniformly in $i$,
\beq\label{eqn:tB2i}
|\tB_{2i}| < \frac{n^{-1}}{1 - \omega_i} = O(n^{-2/3} \sigma_n^{-1/2}), \quad \text{and} \quad |\tB_{3i}| < \frac{n^{-2/3}}{1 - \omega_i} =O(n^{-1/3} \sigma_n^{-1/2}).
\eeq
We have now bounded all the terms of $\tR_i$, except for $L_i$. We provide a uniform bound in $i$ for this quantity in the following subsection. This will conclude the proof of Lemma \ref{lem:unif_Ri}.
\subsection{Uniform bound for \texorpdfstring{$L_i$}{Li}}\label{subsec:bound_Li}
Recall that each $L_i$ is $Y_i$ plus a small term $\a_i-\w_3\dots\w_i\a_2$, where $Y_i$ is a weighted sum of independent random variables,
\beq
Y_i = \sum_{j = 3}^i \w_{j+1}\dots\w_i X_j+X_i, \quad 3 \leq i \leq n.
\eeq
We first show $Y_i$ is small, uniformly in $i$, in the lemma below.
\begin{lemma}\label{lem:unif_Yi}
	Assume $(\log\log n)^2\ll\sigma_n\ll(\log n)^2$. Then with probability $1 - O(\log^{-5}n)$,
	\beqq
	\max_{3 \leq i \leq n} |Y_i| = O\left(\frac{\sqrt{\log \log n}}{n^{1/3} \sigma_n^{1/2}}\right).
	\eeqq
\end{lemma}
In the course of the proof of Lemma \ref{lem:unif_Yi}, we need the following lower bound of the product $\w_{j+1}\dots\w_i$.
\begin{lemma}\label{lem:prod_wi}
	If $i\geq n-n^{1/3}\log^3n$ and $i<j\leq i+n^{1/3}\log^{-2}n$, then $\w_{i+1}\dots\w_j\geq\frac12$.
\end{lemma}
\begin{proof}[Proof of Lemma \ref{lem:prod_wi}] 
	Since $\w_i$ is increasing in $i$, $\log(\w_{i+1}\dots\w_j)\geq(j-i)\log\w_{i+1}$. We have $\w_{i}\geq1- cn^{-1/3}\log^{3/2}n$ for some constant $c>0$. There exists $C>0$ such that $\log(1-x)\geq -Cx$ for all $x\in(0,1)$, so
	\beqq
	\log\w_{i+1} \geq -C_1n^{-1/3}\log^{3/2}n
	\eeqq
	for some $C_1>0$. If $i<j\leq i+n^{1/3}\log^{-2}n$, then 
	\beqq
	\log(\w_{i+1}\dots\w_j)\geq(j-i)\log\w_{i+1}\geq -C\log^{-1/2}n\geq\log(1/2).
	\eeqq
\end{proof}

\begin{proof}[Proof of Lemma \ref{lem:unif_Yi}] 
	By Lemma \ref{lem:SG}, $Y_i\in\SG(v_{Y_i}, u_{Y_i})$ where
	\beq\label{eqn:uvYi}
	\begin{split}
	v_{Y_i}&=\sum_{j = 3}^i (\w_{j+1}\dots\w_i)^2 v_j +v_i 
	\leq \frac{v_i}{1 - \w_i^2} = \frac{\a(i-1)(1+\tau_{i-1})^2}{|\rho_i^+|^2(1-\w_i)},\\
	u_{Y_i}&=\max\left\{\frac{\a}{|\rho_i^+|}, \w_{j+1}\dots\w_i(1+\tau_{j-1})\frac{\a}{|\rho_j^+|}: 3\leq j \leq i-1\right\} \leq \frac{\a(1+\tau_{i-1})}{|\rho_i^+|}.
	\end{split}
	\eeq
	There exists a constant $C>0$ is such that $1+\tau_j \leq C$ for all $n$ and all $3\leq j\leq n$. 
	Thus, by Lemma \ref{lem:SG_boucheron}, for each $i$,
	\beq\label{eqn:Yi_prob}
	\bP\left(|Y_i|> \sqrt{\frac{C^2\a(i-1)t}{|\rho_i^+|^2(1-\w_i)}} + \frac{C\a t}{|\rho_i^+|}\right)\leq 2 e^{-t}.
	\eeq
	Change variable $t \mapsto t+\log 2n$ and take union bound, we have
	\beq
	\bP\left(\forall 3\leq i \leq n: |Y_i|> \sqrt{\frac{C^2\a(i-1)(t+\log 2n)}{|\rho_i^+|^2(1-\w_i)}} + \frac{C\a (t+\log 2n)}{|\rho_i^+|} \right)\leq e^{-t}.
	\eeq
	Fix $\eta>0$ and consider $i\leq n-n^{1/3}\log^{2+\eta}n$. By Corollary \ref{cor:wi}\eqref{cor:wi(1)}, $1 - \w_i > C_1\left(\frac{n-i}n\right)^{1/2}$. Therefore, for $t = \log n$, 
	\begin{align*}
		\sqrt{\frac{C^2\a(i-1)(t+\log 2n)}{|\rho_i^+|^2(1-\w_i)}} + \frac{C\a (t+\log 2n)}{|\rho_i^+|} = O\left(\sqrt{\frac{(i-1)\log n}{n^2(n^{1/3}\log^{2+\eta}n)^{1/2}}} \right) = O(n^{-1/3}\log^{-\eta/4}n).
	\end{align*}
	Take $\eta=1/2$. We have shown that with probability $1-O(n^{-1})$,
	\beq\label{eqn:maxYi_smalli}
	\max_{3\leq i\leq n-n^{1/3}\log^{2+\eta}n}|Y_i| = O(n^{-1/3}\log^{-1/2}n).
	\eeq
	
	Now consider $i>n-n^{1/3}\log^{2+\eta}n$. By Corollary \ref{cor:wi}, there exsists $c>0$ such that for all $3\leq i\leq n$,
	\beq
	1-\w_{i}> cn^{-1/3}\sigma_n^{1/2}.
	\eeq
	By \eqref{eqn:Yi_prob}, this implies that for some $c_1>0$,
	\beq\label{eqn:Yi_prob2}
	\bP\left(|Y_i|>c_1\frac{\sqrt{\log\log n}}{n^{1/3}\sigma_n^{1/4}}\right)\leq \frac2{\log^{10}n}.
	\eeq
	That is, we have $|Y_i| = o(n^{-1/3})$ for each $i>n-n^{1/3}\log^{2+\eta}n$, but the probability bound is too large to apply union bound over this range of indices. Instead, we apply \eqref{eqn:Yi_prob2} to a small number of indices $i>n-n^{1/3}\log^{2+\eta}n$, say $K$ of them. We then bound the maximum $Y_i$ over the $K+1$ subsets partitioned by these indices.
	
	Define for $2\leq i<j\leq n$, 
	\beq
	\tY_{j}^i = \frac{Y_j}{\w_{i+1}\dots\w_j}-Y_i.
	\eeq
	Note that $\tY_{i}^i=0$. As $\{Y_j\}$ satisfies $Y_j = \w_jY_{j-1}+X_j$, we have the recursion 
	\beq
	\tY_{j}^i =\tY_{j-1}^i+\frac{X_j}{\w_{i+1}\dots\w_j}.
	\eeq
	Thus, $\tY_{j}^i$ for fixed $i$ is a sum of independent random variables $\frac{X_k}{\w_{i+1}\dots\w_k}$ for $k=i+1,\dots,j$. We now show that $\tilde{Y}_j^i$ is also subgamma. Let $i$ and $j$ be as given in Lemma \ref{lem:prod_wi}. By Lemma \ref{lem:SG},  $X_k\in\SG(v_k,u_k)$ where $v_k$ and $u_k$ are increasing in $k$. Moreover, $u_j=\frac{\a(1+\tau_{j-1})}{|\rho_j^+|}\leq\frac{C\a}{n}$ and
	\beq
	\sum_{k=i+1}^{j}v_k\leq(j-i)v_j\leq \frac{2C^2\a n^{1/3}}{|\rho_j^+|\log^2n}=O\left(n^{-2/3}\log^{-2}n\right).
	\eeq
	Hence, for some $C>0$,
	\beq
	\tY_{j}^i=\frac{X_{i+1}}{\w_{i+1}}+\dots+\frac{X_{j}}{\w_{i+1}\dots\w_j}\in\SG
	\left(\frac{C\a}{n^{2/3}\log^2n},\frac{C\a}{n}\right).
	\eeq
	Applying Lemma \ref{lem:SG_boucheron} with $t = 10\log\log n$., we have for some $C>0$ and sufficiently large $n$, 
	\beqq
	\max_{1 \leq j' \leq j} \bP\left(|\tY_{j'}^i| >C\frac{\sqrt{\log\log n}}{n^{1/3}\log n}\right) \leq \frac2{\log^{10}n}.
	\eeqq
	By Etemadi's theorem \cite{Etemadi},
	\beq\label{eqn:Etemadi}
	\bP\left(\max_{i \leq j'\leq j} |\tY_{j'}^i| > 3C\frac{\sqrt{\log\log n}}{n^{1/3}\log n}\right) \leq 3 \max_{1 \leq j' \leq j} \bP\left(|\tY_{j'}^i| >C\frac{\sqrt{\log\log n}}{n^{1/3}\log n}\right)\leq\frac6{\log^{10}n}.
	\eeq
	Here, the power 10 can be made larger by choosing sufficiently large $C$.
	
	We now pick $K$ indices as proposed previously. Choose $n_0<n_1<\dots<n_K =n$ where $K\leq2\log^5n$ so that $n_0\leq n-n^{1/3}\log^3n$ and 
	\beqq
	\frac{n^{1/3}}{2\log^2n}\leq n_k - n_{k-1}\leq\frac{n^{1/3}}{\log^2 n}.
	\eeqq
	Take union bound of \eqref{eqn:Yi_prob2} over the set $\{n_k\}_{k=0}^K$, and take union bound of \eqref{eqn:Etemadi} over $K$ pairs $\{(n_{k-1},n_k)\}_{k=0}^K$ to have
	\beq
	|Y_{n_{k-1}}| \leq C\frac{\sqrt{\log\log n}}{n^{1/3}\sigma_n^{1/4}}  \quad \text{ and } \quad \max_{n_{k-1} \leq j \leq n_k} |\tY_{j}^{n_{k-1}}| \leq  4C \frac{\sqrt{\log \log n}}{n^{1/3} \log n}
	\eeq
	for all $K>0$ with probability $1 -O(\log^5n)$. On this event, for every $k=0, \dots,K$, if $j \in [n_{k-1}, n_k]$ then
	\beq
	|Y_j|<|Y_{n_{k-1}}|+|\tY_{j}^{n_{k-1}}| \leq 5C \frac{\sqrt{\log \log n}}{n^{1/3}  \sigma_n^{1/2}}.
	\eeq
	Together with \eqref{eqn:maxYi_smalli}, we conclude
	\beq\label{eq:Yi_max_bound}
	\max_{3\leq i\leq n}|Y_i| = O\left(\frac{\sqrt{\log \log n}}{n^{1/3}  \sigma_n^{1/2}}\right) \quad \text{with probability } 1 -O(\log^{-5}n).
	\eeq
	
	Lastly, note that $L_i=Y_i+s_i$, where 
	\beq\label{eqn:uvsi}
	s_i:=\a_i-\w_3\dots\w_i\a_2\in\SG\left(\frac{2\a \tau_i}{|\rho_i^+|}, \frac{\a}{|\rho_i^+|} \right)\subset\SG\left(\frac{C\a}{n},\frac{\a}{n}\right).
	\eeq
	The rightmost sub-gamma family is independent of $i$. Apply Lemma \ref{lem:SG_boucheron} with $t=n^{1/3-\epsilon}$ for small $\epsilon>0$ and take the union bound,
	\beq
	\bP\left(\max_{3\leq i\leq n}|s_i|>\frac{n^{1/6-\epsilon/2}}{n^{1/2}}\right) \leq \sum_{i = 3}^n\bP\left(|s_i|>\frac{n^{1/6-\epsilon/2}}{n^{1/2}}\right) \leq Cn\exp\left(-n^{1/3-\epsilon}\right),
	\eeq
	for some $C>0$. This completes our proof of Lemma \ref{lem:unif_Yi}.
\end{proof}

We now combine the bounds from all previous subsections. For $t = n^{-1/3} (\log \log n)^{-1/2}$,
\begin{align*}
	\bP(\max_{2\leq i\leq n}|\tR_i|>12t) 
	&\leq\bP(|R_2|\geq 6t)+\bP(\max_{3\leq i\leq n}|\tR_i|>6t)\\
	&\leq \frac1n+\bP(\max_{3\leq i\leq n}|L_i|\geq t)+\bP(|R_2|\geq t) 
	+ \bP(\max_{3\leq i\leq n}|A_{0i}|\geq t)+\bP(\max_{3\leq i\leq n}|\tB_{0i}| \geq t) \\
	&\quad 
	+ \bP(\max_{3\leq i\leq n}|\tB_{1i}|\geq t)+\bP(\max_{3\leq i\leq n}|\tB_{2i}| \geq t)+\bP(\max_{3\leq i\leq n}|\tB_{3i}|\geq t) \\
	&= O(\log^{-5}n),
\end{align*}
and we obtain Lemma \ref{lem:unif_Ri}.

\section{Extension all the way to the edge (Theorem \ref{thm:mainresult2})}

\label{sec:extension_to_edge}

We now consider the case where the sequence $\{\sigma_n\}_n$ satisfies 
\begin{equation}\label{eqn:new_sigma_n}
	\text{for some constant } \tau>0, \quad -\tau <\sigma_n \ll (\log n)^2 \quad \text{for all } n \in \bN,
\end{equation}
and restrict the matrix ensemble to LUE or LOE.  We begin by extending Theorem \ref{thm:mainresult} to this broader range of $\sigma_n$ in the case of LUE, utilizing spectral properties of LUE derived from its determinantal representation (see in particular \cite{GotzeTikhomirov}) in our proof. We then extend the result to LOE matrices using the relationship between eigenvalues of unitary and orthogonal ensembles (see \cite{ForresterRains}).

\begin{remark}
Using a similar technique (drawing on results in \cite{ForresterRains}), this result could be extended to the symplectic ensemble ($\b=4$). In fact, we expect it to hold for all $\b>0$, although proving this would require a substantially different set of techniques that does not rely on determinantal structures (perhaps similar to techniques used in \cite{lambertpaquette}). Here, we restrict our proof to LUE and LOE, which are the relevant cases for statistical and spin glass applications.
\end{remark}

\subsection{Set-up}



 Define for $x\in\bR$,
\beq
S_n(x)=\sum_{i=1}^n\log|d_++xn^{-2/3}-\mu_i|-C_\la n - \frac{1}{\la^{1/2}(1+\la^{1/2})}\sigma_n n^{1/3}+\frac{2}{3\la^{3/4}(1+\la^{1/2})^2}\sigma_n^{3/2} +\frac{\alpha-1}6 \log n.
\eeq
Theorem  \ref{thm:mainresult} implies that for $\bar{\sigma}_n=(\log\log n)^3$,
\[
\frac{S_n(\bar{\sigma}_n)}{\sqrt{\frac{\alpha}{3}\log n}} \stackrel{d}{\to} \mathcal{N}(0,1).
\]
We will show the exact CLT holds for $S_n(\sigma_n)$ by showing that with probability $1-o(1)$, 
\begin{equation}\label{eqn:WTS_Sn}
	S_n(\bar{\sigma}_n)-S_n(\sigma_n) = o(\sqrt{\log n}).
\end{equation}
Let 
\begin{align*}
\ve_n&=n^{-2/3}(\bs-\sn)\label{eqn:ve_i},\\
\ell_i&= \log((\gamma-\mu_i)+\ve_n)-\log|\gamma-\mu_i|-(\gamma-\mu_i)^{-1}\ve_n. 
\end{align*}
Note that $\ve_n$ as above is not the same as $\ve_n$ in \eqref{defn:wi_ei} that arises from the three-term recurrence. We then write 
\begin{equation}
	\begin{split}\label{eqn:diffSn}
		S_n(\bar{\sigma}_n)-S_n(\sigma_n) &=\sum_{i}\left\{\log(\gamma-\mu_i+\ve_n) - \log|\gamma - \mu_i |\right\} - \frac{n^{1/3}(\bs-\sn)}{\la^{1/2}(1+\la^{1/2})}+\frac{2(\bs^{3/2}-\sn^{3/2})}{3\la^{3/4}(1+\la^{1/2})^2}\\
		&= \sum_i \ell_i + \ve_n\left(\sum_{i=1}^n\frac{1}{\gamma-\mu_i}-\frac{1}{\lambda^{1/2}(1+\lambda^{1/2})}n\right)+O(\bs^2).
	\end{split}
\end{equation}
The first sum $\sum_i \ell_i $ can be approximated by a linear eigenvalue statistics, using the following two lemmas. The proof of Lemma \ref{lem:bound_pn} is included in Subsection \ref{sec:edgebound_pf}.  For Lemma \ref{lem:bound_pn} we let $\eta_1,\dots,\eta_n$ be unordered eigenvalues of $\frac1m M_{n,m}$.  Let $p_{n,LUE}(x)$ and $p_{n,LOE}(x)$ be the normalized one-point correlation functions of $\eta_1,\dots,\eta_n$ in the case where $M_{n,m}$ is from $LUE$ and $LOE$ respectively.
\begin{lemma}\label{lem:bound_pn}
	Let $z_\lambda=d_+^{-1}\lambda^{-1/6}$. Given $s_0\in\bR$, there exists $C=C(s_0)>0$ such that for sufficiently large $n$, for all $s\geq s_0$, both of the following statements hold.
	\begin{align*}
	p_{n,\text{LUE}}(d_++sn^{-2/3}) &\leq Cn^{-1/3}\exp\left(-2z_\lambda s\right).\\
	p_{n,\text{LOE}}(d_++sn^{-2/3}) &\leq Cn^{-1/3}\exp\left(-z_\lambda s\right).
	\end{align*}
\end{lemma}

\begin{lemma}\label{lem:4}
	Let $M_{n,m}$ be a scaled LOE/LUE.
	Assume $\sn>-\tau$ for all $n$. Let $\gamma = d_++\sn n^{-2/3}$ and $\bar{\gamma} = d_++\bs n^{-2/3}$. For $\epsilon>0$, there exists $k=k(\epsilon,\tau)>0$ such that for sufficiently large $n$,
	\begin{equation}\label{eqn:lem4a}
		\bP(\mu_1>\bar{\gamma}-n^{-2/3}) <\epsilon, \quad \bP(\mu_k>\gamma) <\epsilon.	
	\end{equation}
	Furthermore, there exist $c_i = c_i(\epsilon,\tau)$, $i=1,2$ such that for sufficiently large $n$,
	\begin{equation}\label{eqn:lem4b}
		\bP(\min_{i\leq n}|\gamma-\mu_i|<c_1n^{-2/3})<\epsilon, \quad \bP(\max_{i\leq k}|\gamma-\mu_i|>(c_2+|\sn|)n^{-2/3})<\epsilon.
	\end{equation}
\end{lemma}

\begin{proof}
	Lemma \ref{lem:4} of this paper is the LUE/LOE version of Lemma 4 of \cite{JKOP1}. There, letting $E=2+\sigma_n n^{-2/3}$ and $\bar{E}=2+\bs n^{-2/3}$, the probability bounds on the distance between location of singularities $E$, $\bar{E}$ to the eigenvalues of scaled GUE/GOE take the exact form as in \eqref{eqn:lem4a} and \eqref{eqn:lem4b}. The key ingredient to the proof is the convergence to the Tracy-Widom law $F_i$ (of type 2 and 1 for the unitary and orthogonal case, respectively) of the $j$th largest eigenvalues (after properly shifted and scaled) for all $j\leq k$ for some fixed $k$. Since the $k$th largest eigenvalues of L$\b$E matrices also satisfy Tracy-Widom convergence, the same proof argument applies. In particular, by replacing their notations with the analogous ones provided in the Table \ref{table:dictionary_lemma4}, we obtain a proof for our lemma.
	\begin{table}
		\caption{A dictionary to translate proof of Lemma 4 in \cite{JKOP1} to the setting of our Lemma \ref{lem:4} for LUE and LOE.}
		\label{table:dictionary_lemma4}
		\begin{tabular}{@{}cc@{}}
		\hline
			JKOP                          & C-WL                           \\ \hline
			$W=W_N$                       & $M_{n,m}$                      \\
			$\lambda_i$                   & $\mu_i$                        \\
			$\gamma$                      & $\tau$                         \\
			$E=2+\sigma_n n^{-2/3}$       & $\gamma=d_++\sigma_n n^{-2/3}$ \\
			$\bar{E}=2+\bs n^{-2/3}$      & $\bar{\gamma}=d_++\bs n^{-2/3}$  \\
			$\rho_n$                      & $p_n$                          \\
			Tail bound (63)               & Tail bound (Lemma \ref{lem:bound_pn})                 \\
			$x_{jN}=N^{2/3}(\lambda_j-2)$ & $x_{jn}=n^{2/3}(\mu_j-d_+)$   \\
			\hline
		\end{tabular}
	\end{table}

\end{proof}
By Lemma \ref{lem:4}, there exists a $k>0$ such that with probability at least $1-\epsilon$, for $i\leq k$,
\begin{align*}
	|\ell_i|
	&=\left|\log(n^{2/3}(\gamma-\mu_i)+\bs-\sn)-\log|n^{2/3}(\gamma-\mu_i)|-(\gamma-\mu_i)^{-1}\ve_n\right| \\
	&\leq \log(3\bs)+\log(c_2+|\sn|)+\frac{n^{-2/3}}{c_1}\bs\leq c_3\bs,
\end{align*}
for some constant $c_3=c_3(\epsilon,\tau)>0$. In the case $i>k$, by the fact $|\log(1+x)-x|\leq x^2/2$ for $x\geq 0$, we obtain
\begin{align*}
	|\ell_i| = \left|\log\left(1+(\gamma-\mu_i)^{-1}\ve_n\right)-(\gamma-\mu_i)^{-1}\ve_n\right| \leq \frac12\frac{\ve_n^2}{(\gamma-\mu_i)^2}. 
\end{align*}
Therefore, with probability at least $1-\epsilon$,
\begin{equation}\label{eqn:bound_elli}
	\begin{split}
	\left|\sum_i\ell_i\right|&\leq \ve_n^2\sum_{i>k}(\gamma-\mu_i)^{-2}+kc_3\bs\leq\ve_n^2\sum_{i=1}^n(\gamma-\mu_i)^{-2}+O(\bs).
	\end{split}
\end{equation}
It remains to approximate the two sums $\sum_{i=1}^n(\gamma-\mu_i)^{-1}-\frac{n}{\lambda^{1/2}(1+\lambda^{1/2})}$ and  $\sum_{i=1}^n(\gamma-\mu_i)^{-2}$ in order to verify \eqref{eqn:diffSn}. 

\begin{proposition}\label{prop:3}
	
Consider $\gamma=d_++\sigma_n n^{-2/3}$ where $\sigma_n$ satisfies  \eqref{eqn:new_sigma_n}, and $\alpha=1$ or $\alpha=2$. Then for any $\epsilon>0$, with probability at least $1-\epsilon$, the following two equations hold.
\begin{align*}
	\sum_{i=1}^n(\gamma-\mu_i)^{-1}-\frac{n}{\lambda^{1/2}(1+\lambda^{1/2})}&= O\left(\left(1+|\sigma_n|^{1/2}\right)n^{2/3}\right),\\
	\sum_{i=1}^n(\gamma-\mu_i)^{-2}&=O(n^{4/3}).
\end{align*}
\end{proposition}
The proof of Proposition \ref{prop:3} is first provided for the LUE case in Section \ref{sec:prop3_LUEproof}, and the proof of the LOE case is included in Section \ref{sec:extension_to_LOE}.
Applying Proposition \ref{prop:3} and the bounds \eqref{eqn:bound_elli} to \eqref{eqn:diffSn}, we obtain 
\[
S_n(\bar{\sigma}_n)-S_n(\sigma_n)=O(\bs^2)=o(\sqrt{\log n})
\]
as claimed, and this completes the proof of Theorem \ref{thm:mainresult2}.

\subsection{Proof of Proposition \ref{prop:3} for LUE}\label{sec:prop3_LUEproof}
As this section focuses solely on LUE matrices, we denote $p_{n,\text{LUE}}$ simply by $p_n$ throughout the section. For our proofs below, we will need the following result from G{\"o}tze and Tikhomirov:
\begin{lemma}[Theorems 1.5 and 1.6, \cite{GotzeTikhomirov}]\label{lem:GotzeTikhomirov}
Let $M_{n,m}$ denote an LUE matrix where $\frac{n}{m}\to\la\leq1$ as $n,m\to\infty$. Let $p_n$ denote the expected spectral density of the empirical spectral measure on $M_{n,m}$, and let $p_{MP}$ be that of the Mar{\v{c}}enko--Pastur measure (see \eqref{eq:def_MP} for definition of these measures). There exist constants $C,a>0$ depending on $\la$ such that, for $x\in[d_-+an^{-2/3},d_+-an^{-2/3}]$,
\beq
|p_n(x)-p_{MP}(x)|\leq\frac{C}{n(d_+-x)(x-d_-)}.
\eeq
Furthermore, for $\la=1$, this holds on the larger interval $x\in[d_-+an^{-2},d_+-an^{-2/3}]$.
\end{lemma}

As an initial step toward proving Proposition \ref{prop:3}, we define
\beq\label{eq:def_fc}
f_c(x)=\frac{1}{\gamma-x}\mathbf{1}\{|\gamma-x|>cn^{-2/3}\}
\eeq
and prove the following lemma about $f_c(x)$.
\begin{lemma}\label{lem:stieltjes_transf_asymp}
Let $\sigma_n$ be in the range $-\tau\leq\sigma_n\leq(\log\log n)^3$.  Then, for each $c>0$, we have
\beq
\EE \frac1n \sum_{j=1}^n f_c^l(\mu_j)=\begin{cases}
\dfrac{1}{\la^{1/2}(1+\la^{1/2})}+O\left((1+|\sigma_n|^{1/2})n^{-1/3}\right) ,& l=1\\
O(n^{1/3}) ,& l=2.
\end{cases}
\eeq
\end{lemma}
\begin{proof}
This lemma is analogous to Lemma 18 in \cite{JKOP1} and we follow a similar proof method.  We have
\beq
\EE \frac1n \sum_{j=1}^n f_c^l(\mu_j)=\int f_c^l(x)p_n(x)dx.
\eeq
This integral with respect to $p_n$, is well approximated by the integral with respect to $p_{MP}$ from the Mar{\v{c}}enko--Pastur measure, so our first task is to bound the error in making this change of measure.  More specifically, we will bound the difference by considering the integral over disjoint intervals:
\beq
\left|\int f_c^l\;p_n-\int f_c^l\;p_{MP}\right|\leq
\int_{I_n}\Big|f_c^l\cdot(p_n-p_{MP})(x)\big|dx+
\int_{J_n^-\cup J_n^+}\Big|f_c^l\cdot(p_n-p_{MP})(x)\big|dx,
\eeq
where the intervals $J_n^-,I_n,J_n^+$ are defined differently for the case of $\la<1$ and $\la=1$ such that the middle interval, $I_n$, corresponds to range on which we can apply the bounds in Lemma \ref{lem:GotzeTikhomirov}. In particular, for any $a>0$ and for $\la<1$, we define
\beq
J_n^-=(0,d_-+an^{-2/3}),\qquad
I_n=[d_-+an^{-2/3},d_+-an^{-2/3}],\qquad
J_n^+=(d_+-an^{-2/3},\infty).
\eeq
If $\la=1$, we set
\beq
J_n^-= (0,an^{-2}) ,\qquad
I_n= [an^{-2},d_+-an^{-2/3}] ,\qquad
J_n^+=(d_+-an^{-2/3},\infty).
\eeq
For the integral over $J_n^-\cup J_n^+$, we use the upper bound
\beq
\sup_{J_n^+}|f_c^l|\int_{J_n^+}(p_n+p_{MP})+
\sup_{J_n^-}|f_c^l|\int_{J_n^-}(p_n+p_{MP}).
\eeq
On $J_n^+$, we have $|f_c^l|=O(n^{2l/3})$. Direct computation shows that $\int_{J_n^+}p_{MP}=O(n^{-1})$ and, using the edge bounds from Lemma \ref{lem:bound_pn}, we see that,
\beq
\int_{J_n^+}p_n(x)dx=n^{-2/3}\int_{-a}^\infty p_n(2+sn^{-2/3})ds=O(n^{-1}).
\eeq
Thus, we conclude that 
\beq\label{eq:Jn+bound}
\sup_{J_n^+}|f_c^l|\int_{J_n^+}(p_n+p_{MP})=O(n^{\frac23 l-1}).
\eeq
On the interval $J_n^-$, the function $|f_c^l|$ is bounded above by a constant.  Two separate computations for $\la=1$ and $\la<1$ show that $\int_{J_n^-}p_{MP}=O(n^{-1})$.  For $p_n$, we observe that
\beq\begin{split}
\int_{J_n^-}p_n&\;\;\leq\;\; 1-\int_{I_n}p_n \;\;\leq\;\; 1-\int_{I_n}p_{MP}+\int_{I_n}|p_n-p_{MP}|.
\end{split}\eeq
We have $1-\int_{I_n}p_{MP}=O(n^{-1})$ based on our computations of $\int_{J_n^+}p_{MP}$ and $\int_{J_n^-}p_{MP}$.  For the difference of measures, we apply Lemma \ref{lem:GotzeTikhomirov} and obtain
\beq\label{eq:meas_diff_In}\begin{split}
\int_{I_n}|p_n-p_{MP}|&\leq\frac{C}{n}\int_{I_n}\frac{1}{(d_+-x)(x-d_-)}dx
\leq\frac{C}{n}\int_{d_-+an^{-2}}^{d_+-an^{-2/3}}\frac{1}{(d_+-x)(x-d_-)}dx\\
&<\frac{2C}{n}\int_{d_-+an^{-2}}^{\frac{d_++d_-}{2}}\frac{1}{(d_+-x)(x-d_-)}dx
=O(n^{-1}\log n).
\end{split}\eeq
We conclude that 
\beq\label{eq:Jn-bound}
\sup_{J_n^-}|f_c^l|\int_{J_n^-}(p_n+p_{MP})=O(n^{-1}\log n).
\eeq
For the integral over $I_n$, we consider separately the intervals $I_n^-:=I_n\cap[0,1]$ and $I_n^+:=I_n\cap(1,\infty)$.  On $I_n^-$, the function $|f_c^L|$ is bounded above by a constant.  Combining this with line \eqref{eq:meas_diff_In}, we get
\beq\label{eq:In-bound}
\int_{I_n^-}|f_c^l|\cdot |p_n-p_{MP}|=O(n^{-1}\log n).
\eeq
Next, we bound the integral on $I_n^+$.  Making the substitution $x=d_+-un^{-2/3}$, we obtain
\beq\label{eq:In+bound}\begin{split}
\int_{I_n^+}|f_c^l|\cdot |p_n-p_{MP}|
&\leq\int_1^{d_+-an^{-2/3}}\frac{\mathbf{1}_{|\gamma-x|>cn^{-2/3}}}{|\gamma-x|^l}\cdot\frac{C}{n(d_+-x)(x-d_-)}dx\\
&=O(n^{\frac23 l-1})\cdot\int_a^{(d_+-1)n^{2/3}}\frac{\mathbf{1}_{|u+\sigma_n|>c}}{|u+\sigma_n|^l}\cdot\frac{du}{u}\\
&\leq O(n^{\frac23 l-1})\int_{\min(a,c)}^\infty \frac{1}{u^{l+1}}du=O(n^{\frac23 l-1})
\end{split}\eeq
Putting together the results from \eqref{eq:Jn+bound}, \eqref{eq:Jn-bound}, \eqref{eq:In-bound} and \eqref{eq:In+bound}, we have shown that 
\beq
\left|\int f_c^l\;p_n-\int f_c^l\;p_{MP}\right|=O(n^{\frac23 l-1}).
\eeq
It remains to compute the integral of $f_c^l$ with respect to the Mar{\v{c}}enko--Pastur measure.

For $\sigma_n\geq c$, the integral $\int f_c p_{MP}$ is $-s_{MP}(\g)$, where $s_{MP}(z)$ denotes the Stieltjes transform of $p_{MP}$, given by
\beq
s_{MP}(z):=\int\frac{1}{x-z}p_{MP}(x)dx=\frac{-z-\la+1+\sqrt{(z-\la-1)^2-4\la}}{2\la z}.
\eeq
Likewise, $\int f_c^2 p_{MP}$ is the derivative of the Stieltjes transform, evaluated at $\gamma$.  Using this and $\gamma=(1+\sqrt{\la})^2+\sigma_n n^{-2/3}$, we conclude for $\sigma_n\geq c$,
\beq
\int f_c^l(x)p_{MP}(x)dx=\begin{cases}-s_{MP}(\gamma)=\frac{1}{\la^{1/2}(1+\la^{1/2})}+O((\sigma_n n^{-2/3})^{1/2}) & l=1, \\
s'_{MP}(\gamma)=O((\sigma_n n^{-2/3})^{-1/2}) & l=2.
\end{cases}
\eeq
In the case of $\sigma_n<c$, we consider the integral over two  sub-intervals $(d_-,d_c)$ and $(d_c,d_+)$, where we set $d_c=d_++(\sigma_n-c)n^{-2/3}$.  Observe that, on the first interval, $f_c^l(x)=\frac{1}{(\gamma-x)^l}$ and, on the second interval, $|f_c^l(x)|\leq c^{-l}n^{2l/3}$. Thus, the integral on $(d_c,d_+)$ has the bound
\beq\label{eq:boundon_dcd+}
\left|\int_{d_c}^{d_+}f_c^l(x)p_{MP}(x)dx\right|=O\left(n^{2l/3}\int_{d_c}^{d_+}\sqrt{d_+-x}dx\right)=O(n^{\frac23 l-1}).
\eeq
For the integral on $(d_-,d_c)$, we consider the cases of $l=1$ and $l=2$ separately.  For $l=1$, we have
\beq\begin{split}
\int_{d_-}^{d_c}f_c \cdot p_{MP}
&=\int_{d_-}^{d_c}\left(f_c(x)-\tfrac{1}{d_+-x}\right)p_{MP}(x)dx
+\int_{d_-}^{d_+}\tfrac{1}{d_+-x}p_{MP}(x)dx
-\int_{d_c}^{d_+}\tfrac{1}{d_+-x}p_{MP}(x)dx\\
&=\int_{d_-}^{d_c}\left(f_c(x)-\tfrac{1}{d_+-x}\right)p_{MP}(x)dx
+\frac{1}{\la^{1/2}(1+\la^{1/2})}+O(n^{-1/3}),
\end{split}\eeq
where the second equality holds by the the fact that the middle term is equal to $-s_{MP}(d_+)$.  To bound the remaining integral on the right side, we have 
\beq\begin{split}
\int_{d_-}^{d_c}\left|f_c(x)-\frac{1}{d_+-x}\right|p_{MP}(x)dx
&=|\sigma_n|n^{-2/3}\int_{d_-}^{d_c}\frac{1}{(\gamma-x)(d_+-x)}p_{MP}(x)dx\\
&=O\left(n^{-2/3}\int_{d_-}^{d_c}\frac{dx}{(\gamma-x)\sqrt{x(d_+-x)}}\right).
\end{split}\eeq
Note that when $\la=1$, $d_-=0$ and the integrand contains a singularity at $x=0$.  However, the integral still remains bounded near that singularity, so we can replace the last line with $O\left(n^{-2/3}\int_{d_-}^{d_c}\frac{dx}{(\gamma-x)\sqrt{d_+-x}}\right)$. By the change of variable $x=d_c-yn^{-2/3}$, this becomes
\beq
n^{-2/3}\int_{d_-}^{d_c}\frac{dx}{(\gamma-x)\sqrt{d_+-x}}
=n^{-1/3}\int_0^{(d_c-d_-)n^{2/3}}\frac{dy}{(y+c)\sqrt{y+(c-\sigma_n)}}=O(n^{-1/3}).
\eeq
Thus, we have shown that, for $\sigma_n<c$ and $l=1$,
\beq
\int f_c^l p_{MP}=\frac{1}{\la^{1/2}(1+\la^{1/2})}+O(n^{-1/3}).
\eeq
It remains to bound the integral $\int f_c^l p_{MP}$ in the case $\sigma_n<c$ and $l=2$. A bound on the portion over $(d_c,d_+)$ is already obtained in \eqref{eq:boundon_dcd+}. The the portion over $(d_-,d_c)$ is
\beq
\int_{d_-}^{d_c}f_c^2(x)p_{MP}(x)dx=O\left(\int_{d_-}^{d_c}\frac{\sqrt{d_+-x}}{(\gamma-x)^2\sqrt{x}}dx \right)=O\left(\int_{d_-}^{d_c}\frac{\sqrt{d_+-x}}{(\gamma-x)^2}dx \right),
\eeq
where the second equality follows by similar reasoning as above. Again, making the substitution $x=d_c-yn^{-2/3}$,
\beq
\int_{d_-}^{d_c}\frac{\sqrt{d_+-x}}{(\gamma-x)^2}dx = n^{1/3}\int_0^{(d_c-d_-)n^{2/3}}\frac{\sqrt{y+(c-\sigma_n)}}{(y+c)^2}dy=O(n^{1/3}).
\eeq
This completes the proof of Lemma \ref{lem:stieltjes_transf_asymp}.
\end{proof}

Besides estimation of the expectation, we also need the following bound on the variance. 
\begin{lemma}\label{lem:16}
If $\eta_1,\dots,\eta_n$ are the unordered eigenvalues of $\frac1m M_{n,m}$, where $M_{n,m}$ is sampled from LUE, then
\[
\mathrm{Var}\left[\frac1n\sum_{i=1}^nf(\eta_i)\right]\leq \frac1n\int f^2(x)p_{n,LUE}(x)dx.
\]
\end{lemma}
\begin{proof}
In Chapter V of \cite{szego}, Laguerre polynomials $L_n^{(a)}$ where $a=m-n>-1$ (for general $\beta$, $a=\frac{\beta}{2}(m-n+1)-1$) are given by two conditions:
\begin{enumerate}
	\item $\int_0^\infty L_j^{(a)}(x)L_k^{(a)}(x)dx = \Gamma(a+1)\binom{k+a}{k}\delta_{jk}$,
	\item coefficient of $x^k$ in $L_k^{(a)}(x)$ has sign $(-1)^k$.
\end{enumerate}
Let $\phi_k(x;a)=h_k^{-1/2}x^{a/2}e^{-x/2}L_k^{(a)}(x)$, where $h_k=\int_0^\infty L_k^{(a)}(x)^2x^a e^{-x}dx$. Then $(\phi_k)_k$ are orthonormal functions with respect to $((0,\infty),dx)$. 

Let $f_n(x_1,\dots,x_n)$ be the joint density of unordered eigenvalues $\eta_1,\dots, \eta_n$ of scaled LUE matrix $\frac 1m M_{n,m}$, $m\geq n$. Let $R_k(x_1,\dots,x_n)$ for $k\geq 1$ be the corresponding $k$-point correlation function, and $S_{n,\text{LUE}}(x)$ be the correlation kernel. Then, 
\begin{equation}
	R_k(x_1,\dots,x_k) = \frac{n!}{(n-k)!}\int\dots\int f_n(x_1,\dots,x_n)dx_{k+1}\dots dx_{n}.
\end{equation}
Moreover, for any integrable function $g$ that is symmetric in $k$ variables,
\begin{equation}\label{eqn:expectation_via_Rk}
	\bE g(\eta_1,\dots,\eta_k)=\frac{(n-k)!}{n!}\int\dots\int g(x_1,\dots,x_k)R_k(x_1,\dots,x_k)dx_1\dots dx_k.
\end{equation}
Note that the $k$-point correlation function for unordered eigenvalues of the unscaled LUE, denoted by $\tilde{R}_k$, is related to $R_k$ by
\[
R_k(x_1,\dots,x_k)=m^k\tilde{R}_k(mx_1,\dots,mx_k).
\]
The normalized one-point correlation of (scaled) eigenvalues then satisfies
\begin{equation}\label{eqn:corr}
	p_{n}(x)=\frac{1}{n}R_1(x)=\frac{1}{\lambda}\tilde{R}_1(mx).
\end{equation}
By the determinantal structure of the eigenvalues (see for example, Section 5.4 of \cite{deift1999}), $\tilde{R}_k$ satisfies
\[
\tilde{R}_k(y_1,\dots,y_k)=\det(S_{n,\text{LUE}}(y_i,y_j))_{i,j=1,\dots,k},
\]
where $S_{n,\text{LUE}}(x,y)=\sum_{j=0}^{n-1}\phi_j(x;a)\phi_j(y;a)$ and $a=m-n$. Thus $R_1(x)=mS_{n,\text{LUE}}(x,x)$ and
\begin{align*}
	R_2(x,y)
	&=m^2\left[R_1(mx)R_1(my)-S^2_{n,\text{LUE}}(mx,my)\right]\\
	&=n^2\left(p_{n}(x)p_{n}(y)-\lambda^{-2}S^2_{n,\text{LUE}}(mx,my)\right).
\end{align*}
Set $I=\bE\left[n^{-1}\sum_{i=1}^{n}f(\eta_i)\right]^2$. We have
\begin{equation}\label{eqn:I}
	\begin{split}
		I&=n^{-2}\bE\left[\sum_{i=1}^{n}f^2(\eta_i)\right]+n^{-2}\bE\left[\sum_{i\neq j}f(\eta_i)f(\eta_j)\right]\\
		&=n^{-1}\bE f^2(\eta_1)+n^{-2}n(n-1)\bE[f(\eta_1)f(\eta_2)]\\
		&=n^{-2}\int f^2(x)R_1(x)dx+n^{-2}\iint f(x)f(y)R_2(x,y)dxdy \quad \text{by } \eqref{eqn:expectation_via_Rk} \\
		&=n^{-1}\int f^2(x)p_{n}(x)dx+\left(\int f(x)p_{n}(x)dx\right)^2-\frac1\lambda\iint f(x)f(y)S^2_{n,\text{LUE}}(mx,my)dxdy.
	\end{split}
\end{equation}
Write $S_{n,\text{LUE}}(x,y)$ as a sum of products $\phi_j(x)\phi_j(y)$, the last integral on the right hand side of \eqref{eqn:I} is a sum of squares (of integrals) so it is positive. In addition, recall the definition of $I$ and that $\left(\int f(x)p_{n}(x)dx\right)^2=\left(n^{-1}\bE \sum_{i=1}^nf(\eta_i)\right)^2$. The last equality of \eqref{eqn:I} then implies
\[
\mathrm{Var}\left[\frac1n\sum_{i=1}^nf(\eta_i)\right] \leq n^{-1}\int f^2(x)p_{n}(x)dx.
\]
\end{proof}

We now combine Lemmas \ref{lem:stieltjes_transf_asymp} and \ref{lem:16} to obtain Proposition \ref{prop:3}.  For the $l=1$ case, 
\beq
\Var\left(\frac1n\sum f_c(\mu_i)\right)\leq\frac1n\EE\left(\frac1n\sum f_c^2(\mu_i)\right)=O(n^{-2/3}),
\eeq
where the inequality follows from Lemma \ref{lem:16} and the big-$O$ term follows from Lemma \ref{lem:stieltjes_transf_asymp}.  
For the $l=2$ case, we observe that $f_c^2$ is strictly positive, so $\frac1n\sum f_c^2(\mu_i)=O\left(\EE(\frac1n\sum f_c^2(\mu_i))\right)$ with high probability.  These observations along with the expectations in Lemma \ref{lem:stieltjes_transf_asymp} imply
\beq\label{eq:Prop3truncated}\begin{split}
	\sum_{i=1}^n f_c(\mu_i)-\frac{n}{\lambda^{1/2}(1+\lambda^{1/2})}&= O\left(\left(1+|\sigma_n|^{1/2}\right)n^{2/3}\right),\\
	\sum_{i=1}^n f_c^2(\mu_i)&=O(n^{4/3}).
\end{split}\eeq
Finally, from Lemma \ref{lem:4}, we know that, for any $\e>0$, there exists $c$ such $\sum f_c^l(\mu_i)=\sum(\gamma-\mu_i)^{-l}$ with probability $1-\e$.  Since \eqref{eq:Prop3truncated} holds for any $c$, we obtain Proposition \ref{prop:3}.

\subsection{Extension of Proposition \ref{prop:3} to LOE}\label{sec:extension_to_LOE}
We extend Proposition \ref{prop:3} from the LUE case to the LOE case using the same method that the authors of \cite{JKOP1} use to extend their result from the GUE case to the GOE case. Since the proof is nearly identical, we do not repeat it here, but rather summarize the key steps in the proof and provide the translation between their setting and ours.

In both our setting and that of \cite{JKOP1}, a key tool to extend results from $\a=1$ to the $\a=2$ is a result from Forrester and Rains about the relationships between eigenvalues of orthogonal, unitary, and symplectic ensembles \cite{ForresterRains}.  Among other findings, their Theorem 5.2 states that
\begin{align}
\text{even}(\text{GOE}_n\cup\text{GOE}_{n+1})&=\text{GUE}_n,\\
\text{even}(\text{LOE}_{n,m}\cup\text{LOE}_{n+1,\; m+1})&=\text{LUE}_{n,m}.
\end{align}
Here $\text{LOE}_{n,m}$ denotes the set of eigenvalues of the LOE matrix that we previously called $M_{n,m}$ (with the notations LUE, GOE, GUE defined similarly). The notation $\text{even}(\cdot)$ denotes the set containing only the even numbered elements among the ordered list of elements in the original set.

The other key tool in the extension from $\a=1$ to $\a=2$ is Cauchy's eigenvalue interlacing theorem.  This theorem states that, if a symmetric $(n+1)\times(n+1)$ matrix and its principal minor have eigenvalues $\la_1\geq\la_2\geq\cdots\la_{n+1}$ and $\mu_1\geq\mu_2\geq\cdots\geq\mu_n$ respectively, then the eigenvalues satisfy the relation
\beq
\la_1\geq\mu_1\geq\la_2\geq\cdots\geq\la_n\geq\mu_n\geq\la_{n+1}.
\eeq
The authors of \cite{JKOP1} use this to relate the eigenvalues of a GOE matrix $M_{n+1}$ to the eigenvalues of its principal minor, which is distributed as an $n\times n$ GOE matrix.  We can also do this for an LOE matrix, provided that we use the tridiagonal representation of LOE (this guarantees that the principal minor is also distributed as an LOE matrix).

Using these two tools, the authors of \cite{JKOP1} prove a theorem about  
$n\times n$ GUE and GOE matrices $M_n^\CC$ and $M_n^\RR$ (see Theorem 19 of \cite{JKOP1}).  We state below the analogous theorem in our setting, which follows from the same proof.
\begin{theorem}
Let $M_{n,m}^\CC$ and $M_{n,m}^\RR$ denote LUE and LOE matrices respectively.
If $f_n$ is a sequence of functions such that
\beq
f_n(M_{n,m}^\CC)=a_n+O(b_n)
\eeq
for some sequences $a_n$ and $b_n$, then
\beq
f_n(M_{n,m}^\RR)=a_n+O(b_n+\text{TV}(f_n))
\eeq
where $\text{TV}(f_n)$ denotes the total variation of $f_n$ and the big-O bounds hold with probability converging to 1.
\end{theorem}

In this theorem, the functions $f_n$ are taken to be single-variable functions where the notation $f_n(M_{n,m})$ is shorthand for $\sum_{i=1}^nf_n(\mu_i)$. The proof is for the unscaled version of these matrices, but it holds for the scaled version as well since scaling the argument does not change the total variation of the function.  Using this theorem, and noting that $\text{TV}(f_c^l)=O(n^{\frac23 l-1})$ for $f_c$ as defined in \eqref{eq:def_fc}, we can extend Lemma \ref{lem:stieltjes_transf_asymp} from the LUE case to the LOE case.  We can further use this theorem to obtain a weaker version of Lemma \ref{lem:16} for the LOE case, namely
\beq
\Var\left[\frac1n\sum_{i=1}^n f(\eta_i)\right]\leq O\left(\frac1n\int f^2(x)p_{n,\text{LOE}}(x)dx+\text{TV}^2(f)\right).
\eeq
These LOE versions of Lemmas \ref{lem:stieltjes_transf_asymp} and \ref{lem:16} are enough to extend Proposition \ref{prop:3} from the LUE case to the LOE case.

\subsection{Proof of Lemma \ref{lem:bound_pn}}\label{sec:edgebound_pf}
We use the same notations as in the proof of Lemma \ref{lem:16}. 
The following equations follow from displays (11) to (15) of \cite{Ma2012}. To begin, note that the one-point correlation function $\tilde{R}_1(x)$ of unordered eigenvalues of unscaled LUE matrix has integral representation
\[
\tilde{R}_1(x)
=\sum_{i=0}^{n-1}\phi_k(x;a)^2
=2\int_0^\infty\phi(x+z;a)\psi(x+z;a)dz,
\]
where
\begin{align*}
	\phi(x;a)&:=(-1)^n\sqrt{\frac{n(n+a)}{2}}\phi_n(x;a-1)x^{-1/2}\mathbbm{1}_{\{x\geq0\}},\\
	\psi(x;a)&:=(-1)^n\sqrt{\frac{n(n+a)}{2}}\phi_{n-1}(x;a+1)x^{-1/2}\mathbbm{1}_{\{x\geq0\}}.
\end{align*}
Throughout the remaining of the proof, we write $\phi(x)$ and $\psi(x)$ when the the parameter $a$ is clear from the context. Given integer $k$, let $k_{-}=k-\frac{1}{2}$. Set
\[
	u_n=(\sqrt{n_-}+\sqrt{m_-})^2, \quad r_n=(\sqrt{n_-}+\sqrt{m_-})\left(\frac{1}{\sqrt{n_-}}+\frac{1}{\sqrt{m_-}}\right)^{1/3},
\]
and define $z_n=z_n(s)$ by $d_+m+\lambda^{-2/3}sm^{1/3}=u_n+z_nr_n$. Then $z_n=z_\lambda s+O(n^{-1/3})$, where the big-O term is uniformly in $s$. We also define 
\begin{align*}
	\eta(z)&=u_n+zr_n,\\
	\phi^{(\eta)}(z)&=r_n\phi(\eta(z)), \quad \psi^{(\eta)}(z)=r_n\psi(\eta(z)).
\end{align*}
From \eqref{eqn:corr}, $p_{n,\text{LUE}}(d_++sn^{-2/3})$ is in fact
\begin{align*}
	\frac{1}{\lambda}\tilde{R}_1(u_n+z_nr_n) 
	&=\frac{2r_n^{-2}}{\lambda}\int_0^\infty \phi^{(\eta)}(z_n+zr_n^{-1})\psi^{(\eta)}(z_n+zr_n^{-1})dz
	=\frac{2r_n^{-1}}{\lambda}\int_{z_n}^\infty \phi^{(\eta)}(z)\psi^{(\eta)}(z)dz.
\end{align*}
	
By Proposition 2 of \cite{Ma2012}, 
\begin{equation}\label{eqn:prop2Ma}
	\forall z_0\in\bR,\exists N_0=N_0(z_0,\lambda), \quad n\geq N_0 \implies |\phi^{(\eta)}(z)|, |\psi^{(\eta)}(z)| \leq C(z_0)e^{-z} \quad \forall z\geq z_0.
\end{equation}
	
Apply \eqref{eqn:prop2Ma} with $z_0=z_\lambda s_0$, then for sufficiently large $n$, for all $s>s_0$,
\begin{align}\label{eqn:pn_LUE}
	p_{n,\text{LUE}}(d_++sn^{-2/3})
	\leq\frac{2r_n^{-1}}{\lambda} C(z_0)^2\exp(-2z_n )
	=O\left(n^{-1/3}\exp(-2z_\lambda s)\right), \quad n\to\infty.
\end{align}
We now verify the edge bound for $p_{n,\text{LOE}}(d_++sn^{-2/3})$. Equation (15) of \cite{Ma2012}, in our notations, states that for $x,y>0$,
\begin{align*}
	S_{n,\text{LOE}}(x,y)
	&=S_{n,\text{LUE}}(x,y)+\psi(x)\frac12\int_0^\infty\phi(u)\text{sgn}(y-u)du\\
	&= S_{n,\text{LUE}}(x,y)+\psi(x)\left[\frac12 I_\phi-\int_y^\infty\phi(u)du\right],
\end{align*}
where $I_\phi=\int_0^\infty\phi(u)du$. Recall $S_{n,\text{LUE}}(x,x)=\tilde{R}_1(x)$ and the relation \eqref{eqn:corr}. The above display implies
\begin{align*}
	p_{n,\text{LOE}}(x)=p_{n,\text{LUE}}(x)+\psi(mx)\left[\frac12 I_\phi-\int_{mx}^\infty\phi(u)du\right].
\end{align*}
Substitute $x=d_++\sigma_n n^{-2/3}$ and use notation $mx=u_n+z_nr_n$, we obtain 
\begin{align}\label{eqn:pn12}
	p_{n,\text{LOE}}(d_++sn^{-2/3})
	&=p_{n,\text{LUE}}(d_++sn^{-2/3})+\frac{r_n^{-1}\psi^{(\eta)}(s)}2\left[\frac12 I_\phi-\int_{z_n}^\infty\phi^{(\eta)}(z)dz\right].
\end{align}
By \eqref{eqn:prop2Ma}, $\int_{z_n}^\infty\phi^{(\eta)}(z)dz\leq Ce^{-z_\lambda s}$ for some $C=C(s_0,\lambda)>0$, for all $s\geq s_0$. In addition, the quantity $I_\phi$ is denoted by $\beta_N$ in \cite{Ma2012}, where it is shown to satisfies $I_\phi=\frac1{\sqrt{2}}+O(n^{-1})$. Thus, the second term on the right hand side of \eqref{eqn:pn12} is $O\left(n^{-1/3}e^{-z_\lambda s}\right)$ uniformly for $s\geq s_0$. We conclude 
\[
p_{n,\text{LOE}}(d_++sn^{-2/3}) \leq Cn^{-1/3}e^{-z_\lambda s}, \quad \forall s\geq s_0.
\]

\begin{appendix}

\section{Technical lemmas}\label{sec:technical_lemmas}
Consider the following process
\begin{align*}
	\hR_2&=R_2\\
	\hR_i&=L_i+\w_i\dots \w_3\hR_2-A_{0i}+\hB_{0i}+\hB_{1i}+\phi_{\frac2{n(1-\w_i)}}(B_{2i})+\hB_{3i}, \quad 3 \leq i \leq n
\end{align*}
where
\begin{align*}
	\hB_{0i} &= \left(\a_{i-1}+ ( \tau_{i-1} + \a_{i-1} ) \hR^{(1)}_i\right) \beta_i  + \omega_i \left(\a_{i-2} +( \tau_{i-2} + \alpha_{i-2} )\hR^{(1)}_{i-1}\right) \beta_{i-1} \\
	&\quad + \dots + \w_i\dots\w_4\left(\alpha_2  + (\tau_2  + \a_2 ) \hR^{(1)}_3\right) \b_3, \quad \hR^{(1)}_i = \frac{\hR_{i-1}}{1 - \phi_{1/2}(R_{i-1})}, \\
	\hB_{1i} &= \a_{i-1}\delta_i \hR^{(1)}_i + \w_i\a_{i-2}\delta_{i-1}\hR^{(1)}_{i-1} + \dots + \w_i\dots\w_4\alpha_2 \delta_3 \hR^{(1)}_3,\\
	\hB_{3i} &= \hR^{(3)}_i + \w_i\hR^{(3)}_{i-1} + \dots + \w_i\dots\w_4\hR^{(3)}_3, \quad \hR^{(3)}_i = \w_{i} \phi_{n^{-1/3}}(\hR_{i-1})\hR_{i-1}.
\end{align*}
The event that $2|1-\a_1|^{-1}>1$ and $|R_i| \leq n^{-1/3}$ and $|B_{2i}|\leq \frac2{n(1-\w_i)}$ for all $3\leq i \leq n$ occurs with probability $1 -O(\log^{-5}n)$. The bound for $|R_i|$ holds by Lemma \ref{lem:unif_Ri}, and bound for $|B_{2i}|$ follows from inequality \eqref{eqn:tB2i} for $\tB_{2i}$ in the proof of Lemma \ref{lem:unif_Ri}. Thus on this event, $\hR_2=R_2$ and  $\hR^{(\ell)}_3=R^{(\ell)}_3$ for $\ell=1,3$, and $\phi_{\frac2{n(1-\w_i)}}(B_{23})=B_{23}$. Thus $\hR_{3}=R_3$. Repeat the argument with increasing $i$, we obtain that 
\beq\label{eqn:hRvR}
\hR_i=R_i \quad \text{for every }2\leq i\leq n \quad \text{with probability }1-O(\log^{-5}n).
\eeq

\subsection{Proof of Lemma \ref{lem:Ri2_bound}}
Consider $\sum_{i = 3}^n\hR_i^2$.
From the inequality
\beq
\sum_{i=3}^n\|\hR_i^2\|_1\leq\sum_{i=3}^n\|\hR_i\|^2_2\leq\sum_{i=3}^n\|\hR_i\|_4^2,
\eeq
and Markov's inequality, it suffices to show the last sum is of order 1. Lemma \ref{lem:SG_boucheron} implies that if $X\in \SG(v,u)$, then $\|X\|_p\leq C_p(v^{\frac p2}+u^p)^{\frac1p}$. By $\eqref{eqn:uvYi}$ and \eqref{eqn:uvsi},
\beq\label{eqn:L4_Li}
\|L_i\|_4\leq\|Y_i\|_4+\|s_i\|_4\leq\frac{C\a^{\frac12}}{\sqrt{n(1-\w_i)}}.
\eeq
Also, $\|\a_i\|_4,\|\b_i\|_4 = O(n^{-\frac12})$ uniformly in $i$. Hence, by \eqref{eqn:R2},
\begin{align}\label{eqn:R2-hat}
	\|\hR_2\|_4\leq\|R_2\|_4
	&\leq |\w_2-\g_2|+\|\a_2\|_4+(1+C\|\a_1\|_4)\|\b_2\|_4+\frac{1+C\|\a_1\|_4}{|\rho_2^+|}=O\left(n^{-\frac12}\right).
\end{align}
Thus $\|\w_3\dots\w_i\hR_2\|_4\leq\w_3\|\hR_2\|_4=O(n^{-\frac32})$. Observe that $\left\|\phi_{\frac2{n(1-\w_i)}}(B_{2i})\right\|_4\leq\frac2{n(1-\w_i)}$, and $|A_{0i}|<\frac1{n(1-\w_i)^2}$ from \eqref{eqn:A0i}. Now, for each $i$,
\beqq
\begin{split}
	\|\left(\a_{i-1}+( \tau_{i-1}+\a_{i-1})\hR^{(1)}_i\right)\beta_i\|_4 
	&\leq \|\a_{i-1}\|_4\|\b_i\|_4+C\|\b_i\|_4\|\hR^{(1)}_i\|_4
	\leq Cn^{-1}+Cn^{-1/2}\|\hR^{(1)}_i\|_4.  
\end{split}
\eeqq
Hence,
\beq\label{eqn:L4_hB0i}
\|\hB_{0i}\|_4\leq\frac C{n(1-\w_i)}+\frac{Cn^{-\frac12}}{1-\w_i}\cdot\max_{3\leq j\leq i-1}\|\hR_j\|_4.
\eeq
Similarly, $\|\a_{i-1}\delta_i\hR_i^{(1)}\|_4\leq Cn^{-\frac12}\|\hR_{i-1}\|_4$ and $\|\hR_i^{(3)}\|_4\leq n^{-\frac13}\|\hR_{i-1}\|_4$ so 
\beq\label{eqn:L4_hB1i}
\|\hB_{1i}\|_4 
\leq \frac{Cn^{-1/2}}{1-\w_i}\max_{3\leq j\leq i-1}\|\hR_{j}\|_4 \quad \text{and} \quad \|\hB_{3i}\|_4 
\leq \frac{n^{-1/3}}{1-\w_i}\max_{3\leq j\leq i-1}\|\hR_{j}\|_4.
\eeq
Combining all the estimates, we have
\beq
\begin{split}
	\|\hR_i\|_4 
	&\leq \frac{C\alpha^{1/2} +o(1)}{\sqrt{n(1-\w_{i})}}+o(1)\max_{3\leq j\leq i-1}\|\hR_{j}\|_4, 3\leq i\leq n.
\end{split}
\eeq
Since $\|\hR_2\|_4 = O(n^{-\frac12})$, by induction we obtain for sufficiently large $n$,
\beq\label{eqn:L4_hRi}
\|\hR_i\|_4\leq\frac{C\alpha^{1/2}}{\sqrt{n(1-\w_{i})}}, \quad \text{for } i=3, \dots,n.
\eeq
Therefore, 
\beq
\begin{split}
	\sum_{i=3}^n\|\hR_i^2\|_1
	&=O\left(\sum_{i=3}^n\frac{C\alpha^{1/2}}{\sqrt{n(1-\w_{i})}}\right)
	=O\left(\frac1n\sum_{i=3}^{n-n^{\frac13}\sigma_n}\left(\frac{n-1}n\right)^{-\frac12}+\frac1n \sum_{i>n-n^{\frac13}\sigma_n}n^{-\frac13}\sigma_n^{\frac12}\right)=O(1),
\end{split}
\eeq
and we obtain $\sum_{i=3}^n\hR_i^2=O(1)$ with probability $1-o(1)$. By \eqref{eqn:hRvR}, the same statement applies to $\sum_{i=3}^nR_i^2$. 

\subsection{Proof of Lemma \ref{lem:AB_bound}}
By \eqref{eqn:hRvR}, it suffices to show that, with probability $1-o(1)$,
\beq
\sum_{i=3}^n \w_{3}\dots\w_iR_2+\hB_{0i}+\hB_{1i}+\phi_{\frac2{n(1-\w_i)}}(B_{2i}) = O(1).
\eeq
This holds as long as the $L_1$ norm of this sum is of order 1. By Definition \ref{defn:gi} and $\|R_2\|_1\leq \|R_2\|_4 = O(n^{-1/2})$ (which is a consequence of \eqref{eqn:R2-hat}),
\beqq
\sum_{i=3}^n\w_{3}\dots\w_i\|R_2\|_1 =\w_3g_4\|R_2\|_1=O(\w_3\|R_2\|_4)=O(n^{-\frac32}).
\eeqq
Here, $g_4=O(1)$ follows from Lemmas \ref{lem:gi_bounds} and \ref{cor:wi}, and a direct computation gives $\w_3=O(n^{-1})$.
By \eqref{eqn:L4_hB0i} and \eqref{eqn:L4_hRi} and Corollary \ref{cor:wi},
\beqq
\sum_{i=3}^n\|\hB_{0i}\|_1\leq\sum_{i=3}^n \frac{C}{n(1-\w_i)}+\sum_{i=3}^n\frac{C \a^{\frac12}}{n(1-\w_i)^{\frac32}} = O(1).
\eeqq
Similarly, by \eqref{eqn:L4_hB1i},
\beqq
\sum_{i=3}^n\|\hB_{1i}\|_1\leq\sum_{i=3}^n\frac{C\a^{\frac12}}{n(1-\w_i)^{\frac32}} = O(1).
\eeqq
Lastly, 
\beqq
\sum_{i=3}^n\left\|\phi_{\frac2{n(1-\w_i)}}(B_{2i})\right\|_1\leq\sum_{i=3}^n\frac2{n(1-\w_i)}=O(1).
\eeqq

\subsection{Proof of Lemma \ref{lem:E2_bound}}
	Observe that $E_1=\frac{a_1^2-\g m}{|\rho_1^+|} = \a_1-1$. Hence,
	\beqq
	E_2 = E_1(R_2-1) = (1-\a_1)(1-R_2).
	\eeqq
	By Lemma \ref{lem:SG_boucheron}, we have with probability $1 - O(n^{-1})$, $|\a_1| = O(n^{-1/2}\log^{1/2}n)$ and $|R_2| = O(n^{-1/2})$. Thus there exists $C_1<0<C_2$ such for sufficiently large $n$,
	\beqq
	\log|E_2|=\log|1-\a_1|+\log|1-R_2|\in(C_1,C_2)
	\eeqq
	with probability $1-O(n^{-1})$.

\subsection{Proof of Lemma \ref{lem:B3i_star_hat}}
	We apply \eqref{eqn:hRvR} to replace $B_{3i}$ by $\hB_{3i}$ for every $i=3,\dots, n$, then show that $\sum_{i = 3}^n\hB_{3i}-B^*_{3i}=O(1)$ with probability $1-o(1)$. Recall 
	\beqq
	\hB_{3i}=\hR_{i}^{(3)}+\w_iR_{i-1}^{(3)}+\w_i\dots\w_4R_{3}^{(3)},
	\eeqq
	where $\hR_{i}^{(3)}=\w_i\phi_{n^{-1/3}}(\hR_{i-1})\hR_{i-1}$. Consider
	\beqq
	\hC_{3i} = (\omega_{i}\hR^2_{i-1}) + \omega_i(\omega_{i-1}\hR^2_{i-2}) + \dots + \omega_i\dots \omega_4 (\omega_{3}\hR^2_{2}).
	\eeqq
	By Lemma \ref{lem:unif_Ri} and the fact $R_i=\hR_i$ for all $2\leq i\leq n$ with probability $1-o(1)$, it holds with probability $1-o(1)$ that $\hB_{3i}=\hC_{3i}$ for all $3\leq i\leq n$. Hence, it is sufficient to show $\sum_{i=3}^n\|\hC_{3i}-B^*_{3i}\|_1 = O(1)$.
	\beq
	\|\hC_{3i}-B^*_{3i}\|_1
	\leq\sum_{j=3}^{i-1}\w_i\dots\w_j\|\hR^2_{j-1}-L^2_{j-1}\|_1\leq\frac{\w_i}{1-\w_i}\max_{2\leq j\leq i-1}\|\hR^2_{j-1}-L^2_{j-1}\|_1.
	\eeq
	By H\"{o}lder's inequality,
	\beqq
	\|\hR^2_{i}-L^2_{i}\|_1\leq \|\hR_i-L_i\|_2\|\hR_i+L_i\|_2.
	\eeqq
	Apply triangle inequality, we have
	\beq
	\begin{split}
		\|\hR_i - L_i\|_2 
		&\leq  \|\omega_{i} \dots \omega_3 \hR_2\|_2 + \|A_{0i}\|_2 + \|\hB_{0i}\|_2 + \|\hB_{1i}\|_2+ \left\|\phi_{\frac{2}{n(1 - \omega_{i})}}(B_{2i}) \right\|_2 + \|\hB_{3i}\|_2.
	\end{split}
	\eeq
	Since $\|X\|_2\leq\|X\|_4$ for all random variables $X\in L_4(\bP)$, we can apply the bounds on $L_4$-norms obtained in the proof of Lemma \ref{lem:AB_bound}. Thus, for some $C>0$ and sufficiently large $n$, the following four inequalities hold.
	\beqq
	\|\w_i\dots\w_3\hR_2\|_2 \leq Cn^{-\frac32}, \quad \|A_{0i}\|_2\leq\dfrac C{n(1-\w_i)^2}, \quad \|\hB_{0i}\|_2\leq\dfrac C{n(1-\w_i)}+\dfrac C{n(1-\w_i)^{\frac32}}, \quad \|\hB_{1i}\|_2\leq\dfrac C{n(1-\w_i)^{\frac32}}.
	\eeqq
	At the same time, since $\|\hR_i^{(3)}\|_2\leq\|\hR_i^2\|_2=\|\hR_i\|^2_4=O\left(\frac1{n(1-\w_i)}\right)$,
	\beqq
	\|\hB_{3i}\|_2\leq\frac1{1-\w_i}\cdot\max_{3\leq i\leq i}\|\hR_i^{(3)}\|_2\leq\frac C{n(1-\w_i)^2}.
	\eeqq 
	Hence, $\|\hR_i - L_i\|_2 = O\left(\frac1{n(1-\w_i)^2}\right)$. Similarly, by \eqref{eqn:L4_Li} and \eqref{eqn:L4_hRi},
	\beq
	\|\hR_i+L_i\|_2\leq \|\hR_i\|_4+\|L_i\|_4=O\left(\frac1{\sqrt{n(1-\w_i)}}\right).
	\eeq
	Therefore,
	\beqq
	\|\hR^2_{i}-L^2_{i}\|_1\leq \|\hR_i-L_i\|_2\|\hR_i+L_i\|_2=O\left(\frac1{n^{\frac32}(1-\w_i)^{\frac52}}\right),
	\eeqq
	and $\|\hC_{3i}-B^*_{3i}\|_1=O\left(\frac1{n^{\frac32}(1-\w_i)^{\frac72}}\right)$. By Lemma \ref{lem:wi}, we conclude
	\beq
	\begin{split}
		\sum_{i=3}^n&\|\hC_{3i}-B^*_{3i}\|_1
		=O\left(\sum_{i=3}^{n-n^{\frac13}\sigma_n}\frac1{n^{\frac32}}\left(\frac{n-i}{n}\right)^{-\frac74}+\sum_{i=n-n^{\frac13}\sigma_n}^n n^{-\frac32}(n^{\frac13}\sigma_n^{-\frac12})^{\frac72} \right)
		=O\left(\sigma_n^{-\frac34}\right)=o(1).
	\end{split}
	\eeq

\subsection{Proof of Lemma \ref{lem:B3i_smallstuff}}
We expand the terms inside the sum to have
\beq
\sum_{i=4}^n (g_i-1)\left[2Y_{i-1}(\alpha_{i-1}-\w_3\cdots\w_{i-1}\a_2)+(\alpha_{i-1}-\w_3\cdots\w_{i-1}\a_2)^2\right]=:P_1-P_2+P_3,
\eeq
where
\beq\begin{split}
P_1:&=\textstyle{\sum_{i=4}^n}\;2(g_i-1)Y_{i-1}\a_{i-1},\\
P_2:&=\textstyle{\sum_{i=4}^n}\;2(g_i-1)Y_{i-1}\w_3\cdots\w_{i-1}\a_2,\\
P_3:&=\textstyle{\sum_{i=4}^n}\;(g_i-1)(\a_{i-1}-\w_3\cdots\w_{i-1}\a_2)^2.
\end{split}\eeq
Recalling that
\beq\begin{split}
Y_{i-1}=\sum_{j=3}^{i-2}X_j\w_{j+1}\cdots\w_{i-1}+X_{i-1},
\qquad X_j=(1+\tau_{j-1})(\delta_j\a_{j-1}+\b_j),
\end{split}\eeq
we further decompose $P_1$ into
\beq\begin{split}
P_1=&\sum_{i=4}^n 2(g_i-1)\alpha_{i-1} \bigg(\sum_{j=3}^{i-2}(1+\tau_{j-1})\alpha_{j-1}\delta_j\omega_{j+1}\cdots\omega_{i-1}+(1+\tau_{i-2})\alpha_{i-2}\delta_{i-1}\bigg)\\
&\quad+\sum_{i=4}^n 2(g_i-1)\alpha_{i-1} \bigg(\sum_{j=3}^{i-2}(1+\tau_{j-1})\beta_j\omega_{j+1}\cdots\omega_{i-1}+(1+\tau_{i-2})\beta_{i-1}\bigg)\\
=&:P_{11}+P_{12}.
\end{split}\eeq
We now rewrite each part of the summation as quadratic forms and apply Lemma \ref{lem:HansonWright} as follows. Define
\beq\begin{split}
&\bfa = (\alpha_2, \alpha_3 \dots, \alpha_{n-1})^T , \quad \bfb = (\beta_3, \beta_4 \dots, \beta_{n-1})^T, \\
&\bfaback = (\a_3,\a_4, \dots, \a_{n-1})^T,\quad \bfafront = (\alpha_2, \alpha_3, \dots, \alpha_{n-2})^T.
\end{split}\eeq

\textbf{Bound for $P_{11}$ part:} Observe that 
\beq
P_{11}=(\bfaback)^T Z \bfafront
\eeq
where $Z$ is the lower triangular matrix
\beq\begingroup
\medmuskip=0mu
Z=2
\begin{pmatrix}
(g_4-1)(1+\tau_2)\delta_3 & && &\\
(g_5-1)(1+\tau_2)\delta_3\omega_4 & (g_5-1)(1+\tau_3)\delta_4 &&&\\
(g_6-1)(1+\tau_2)\delta_3\omega_4\omega_5 & (g_6-1)(1+\tau_3)\delta_4\omega_5 & (g_6-1)(1+\tau_4)\delta_5&&\\
\vdots & & & \ddots& \\
(g_n-1)(1+\tau_2)\delta_3\omega_4\cdots\omega_{n-1} &\cdots
& & \cdots & (g_n-1)(1+\tau_{n-2})\delta_{n-1}
\end{pmatrix}.
\endgroup\eeq
Alternatively, we can express this as a quadratic form
\beq
\mathbf{a}^T\tilde{Z}\mathbf{a}
\eeq
where $\tilde{Z}$ is the matrix $Z$ with a row of zeros appended at the top and a column of zeros appended at the right. That is, 
\beq
\tilde{Z}_{ij}=2(g_{i+2}-1)(1+\tau_{j+1})\delta_{j+2}\omega_{j+3}\cdots\omega_{i+1}\quad\text{for }i\geq j+1.
\eeq
Next, we observe that
\beq
\mathbf{a}^T\tilde{Z}\mathbf{a}=\frac12 \mathbf{a}^T(\tilde{Z}+\tilde{Z}^T)\mathbf{a}.
\eeq
Since $\tilde{Z}+\tilde{Z}^T$ is a symmetric matrix and $\mathbf{a}$ is a vector of independent random variables satisfying $\a_i\in SG(c_1n^{-1},c_2n^{-1})$ we can apply Lemma \ref{lem:HansonWright}.  Furthermore, since $\tilde{Z}$ has zeros along the diagonal, $\EE \bfa^T Z\bfa=0$ so we conclude that with probability $1-o(1)$,
\beq
\mathbf{a}^T\tilde{Z}\mathbf{a}=O\left(\frac{\nu_n}{n}\| \tilde{Z}+\tilde{Z}^T\|_{\HS} \right)
=O\left(\frac{\nu_n}{n}\| \tilde{Z}\|_{\HS} \right),
\eeq
where $\nu_n$ is some slowly growing function of $n$ (for example $\sqrt{\log n}$).  We observe that 
\beq
\|\tilde{Z}\|_{\HS}=\sqrt{\sum_{i=2}^{n-2}\sum_{j=1}^{i-1}\tilde{Z}_{ij}^2}.
\eeq
We bound the quantity $\tilde{Z}_{ij}^2$ as follows:
\beq\begin{split}
\tilde{Z}_{ij}^2&=4(g_{i+2}-1)^2(1+\tau_{j+1})^2\delta_{j+2}^2\omega_{j+3}^2\cdots\omega_{i+1}^2
\leq C\frac{1}{(1-\omega_{i+2})^2}\left(\frac{j+1}{n}\right)^2 \omega_{j+3}^2\cdots\omega_{i+1}^2\\
&< \frac{C(i+2)}{n(1-\omega_{i+2})^2}\omega_{j+3}^2\cdots\omega_{i+1}^2.
\end{split}\eeq
For fixed $i$, 
\beq\begin{split}
\sum_{j=1}^{i-1}\tilde{Z}_{ij}^2
&\leq \sum_{j=1}^{i-1} \frac{C(i+2)}{n(1-\omega_{i+2})^2}
\omega_{j+3}^2\cdots\omega_{i+1}^2
< \frac{C(i+2)}{n(1-\omega_{i+2})^2}
\cdot\frac{1}{1-\omega_{i+1}^2}
< \frac{C(i+2)}{n(1-\omega_{i+2})^3}.
\end{split}\eeq
We now sum this quantity over the indices $i$, treating separately the indices $i\leq n-n^{1/3}\sigma_n$ and $i\geq n-n^{1/3}\sigma_n$.  Since we care only about the order of this quantity we omit the initial constant $C$, although $c$ will show up later denoting some other constant.  For the sum over indices less that $n-n^{1/3}\sigma_n$, we get
\beq\begin{split}
\sum_{i=2}^{n-n^{1/3}\sigma_n}  \frac{i+2}{n(1-\omega_{i+2})^3}
&\leq\sum_{i=2}^{n-n^{1/3}\sigma_n} \frac{i+2}{n}\left(\frac{cn}{n-(i+2)}\right)^{3/2}\\
&=O\left(n\int_{n^{-2/3}\sigma_n}^1(1-x)x^{-3/2}dx
\right)
=O\left(n\int_{n^{-2/3}\sigma_n}^1x^{-3/2}dx\right)
=O(n^{4/3}\sigma_n^{-1/2}).
\end{split}\eeq
Meanwhile,
\beq\begin{split}
\sum_{i=n-n^{1/3}\sigma_n}^{n-2}  \frac{i+2}{n(1-\omega_{i+2})^3}
&< \sum_{i=n-n^{1/3}\sigma_n}^{n-2}  \frac{1}{(1-\omega_{i+2})^3}
\leq \sum_{i=n-n^{1/3}\sigma_n}^{n-2} (cn^{1/3}\sigma_n^{-1/2})^3\\
&=O\left(n^{1/3}\sigma_n\cdot n\sigma_n^{-3/2}\right)
=O\left(n^{4/3}\sigma_n^{-1/2}\right).
\end{split}\eeq
Putting the two sums together, we get
\beq
\|\tilde{Z}\|_{\HS}=\sqrt{\sum_{i=2}^{n-2}\sum_{j=1}^{i-1}\tilde{Z}_{ij}^2}=O(\sqrt{n^{4/3}\sigma_n^{-1/2}})=O(n^{2/3}\sigma_n^{-1/4}),
\eeq
and thus, with probability $1-o(1)$,
\beq
\mathbf{a}^T\tilde{Z}\mathbf{a}=O\left(\frac1n\cdot n^{2/3}\sigma_n^{-1/4}\nu_n\right)=O(n^{-1/3}\sigma_n^{-1/4}\nu_n)
\eeq
where $\nu_n$ is, again, some slowly growing function.

\textbf{Bound for $P_{12}$ part:}
Using the vectors defined above and the matrices $W,G,D$ from Definition \ref{def:HW_notations}, we write
	\begin{equation}
		P_{12}=2 (\bfa^{(\text{f})})^T GW D \bfb = \begin{pmatrix}
			(\bfa^{(\text{f})})^T & \bfb^T
		\end{pmatrix} \begin{pmatrix}
			O & GW D \\
			(GW D)^T & O \\
		\end{pmatrix} \begin{pmatrix}
			\bfa^{(\text{f})} \\
			\bfb
		\end{pmatrix}.
	\end{equation} 
	Since the matrix has zeros on the diagonal, $\EE P_{12}=0$. By Lemma \ref{lem:HansonWright}, with probability $1-o(1)$,
	\begin{equation}\label{eqn:step7}
		2 (\bfa^{(\text{f})})^T GW D \bfb  = O \left(\frac{\nu_n}{n} \left\| \begin{pmatrix}
			O & GW D \\
			(GW D)^T & O \\
		\end{pmatrix}\right\|_{\HS}\right) = O \left(\frac{\nu_n}{n} \|GWD\|_{\HS}\right).
	\end{equation}
	We have
	\begin{equation}\label{eqn:norm_GWD}
		\begin{split}
			\|GWD\|_{\HS}^2 
			&= \sum_{i = 1}^{n-3} \sum_{j = 1}^i (GWD)_{ij}^2
			= \sum_{i = 1}^{n-3} \sum_{j = 1}^i (g_{i+3}-1)^2 (1+\tau_{j+1})^2 \omega_{j+3}^2\dots \omega_{i+2}^2\\
			&\leq C \sum_{i = 1}^{n-3} \sum_{j = 1}^i (g_{i+3}-1)^2 \omega_{j+3}^2\dots \omega_{i+2}^2.
		\end{split}
	\end{equation}
	For indices $1 \leq i \leq n - n^{1/3}\sigma_n-3$, 
	\[
	(g_{i+3}-1)^2 \leq \frac{n }{n -(i+3)}, \quad \text{ and } \quad \frac{1}{ 1 - \omega_i} \leq \sqrt{\frac{c n }{n - (i+3)}}.
	\]
	Thus, 
	\begin{equation}
		\begin{split}
			&\sum_{i = 1}^{n - n^{1/3}\sigma_n-3}\sum_{j = 1}^i (g_{i+3}-1)^2\omega_{j+3}^2\dots \omega_{i+2}^2
			\leq \sum_{i = 1}^{n - n^{1/3}\sigma_n-3} \frac{n }{n -(i+3)} \frac{1}{1 - \omega_{i+2}}\\
			&\quad\leq c'\sum_{i = 1}^{n - n^{1/3}\sigma_n-3}\left(\frac{n }{n -(i+3)}\right)^{3/2}
			= O \left(n \int_{n^{-2/3} \sigma_n}^{1} x^{-3/2}dx \right) = O(n^{4/3}\sigma_n^{-1/2}).
		\end{split}
	\end{equation}
	
	The contribution from the remaining terms is 
	\begin{equation}
		\begin{split}
			\sum_{i = n - n^{1/3}\sigma_n-2}^{n -3}\sum_{j = 1}^i (g_{i+3}-1)^2 \omega_{j+3}^2\dots \omega_{i+2}^2
			&\leq  \sum_{i = n - n^{1/3}\sigma_n-2}^{n -3} \sum_{j = 1}^i (Cn^{1/3} \sigma_n^{-1/2})^2 \omega_{j+3}^2\dots \omega_{i+2}^2 \\
			&\leq  \sum_{i = n - n^{1/3}\sigma_n-2}^{n -3} \frac{n^{2/3} \sigma_n^{-1}}{1 - \omega_{i+3}} 
			= O(n^{4/3} \sigma_n^{-1/2}).
		\end{split}
	\end{equation}
	Thus, we get $\|GWD\|_{\HS}^2 = O(n^{4/3} \sigma_n^{-1/2})$. By \eqref{eqn:step7}, we conclude that, with probability $1-o(1)$, 
	\[
	2 (\bfa^{(\text{f})})^T GW D \bfb  = O \left(\frac{\nu_n}{n} \|GWD\|_{\HS}\right) =O \left(\frac{\nu_n}{n} \cdot n^{2/3} \sigma_n^{-1/4}  \right) =o(1).
	\]

\textbf{Bound for $P_2$ part:} We recall two facts.  First, from Lemma \ref{lem:unif_Yi}, $\max_{3\leq i\leq n}|Y_i|=o(n^{-1/3})$ with probability $1-o(1)$.  Second, $\a_2\in SG(v,u)$ with $v,u=O(n^{-1})$ so, by Lemma \ref{lem:SG_boucheron}, $\a_2=O(n^{-1/2+\e})$ with probability $1-o(1)$ for any $\e>0$.  Combining these two facts, we deduce that, with probability $1-o(1)$,
\beq
P_2=o\left(n^{-1/3}n^{-1/2+\e}\sum_{i=4}^n (g_i-1)\w_3\cdots\w_{i-1}\right).
\eeq
The above sum can be crudely bounded as 
\beq
\sum_{i=4}^n (g_i-1)\w_3\cdots\w_{i-1}<\sum_{i=4}^n (g_i-1)\w_3\w_4
=O\left(n\cdot n^{1/3}\sigma_n^{-1/2}\cdot n^{-2}\right)
=O(n^{-2/3}\sigma^{-1/2}),
\eeq
where we use Lemmas \ref{lem:gi_bounds} and \ref{lem:wi} to bound $g_i$ and the definition of $\w_i$ to bound $\w_3, \w_4$. We obtain $P_2=o(1)$.

\textbf{Bound for $P_3$ part:}  We now bound $\sum_{i=4}^n (g_i-1)(\alpha_{i-1}-\omega_3\cdots\omega_{i-1}\alpha_2)^2$. This can be expressed as a quadratic form
\beq
\mathbf{a}^TQ\mathbf{a}
\eeq
where $\mathbf{a}$ is the same vector from before and $Q$ is a symmetric matrix with non-zero entries only in the first row, first column, and along the diagonal.  More specifically, these entries are
\beq
Q_{ij}=\begin{cases}
g_{i+2}-1 & i=j\geq2,\\
-(g_{j+2}-1)\omega_3\cdots\omega_{j+1} & i=1, j\geq2,\\
-(g_{i+2}-1)\omega_3\cdots\omega_{i+1} & i\geq2, j=1,\\
-\sum_{k=2}^{n-2} (g_{k+2}-1)(\omega_3\cdots\omega_{k+1})^2 & i=j=1.
\end{cases}
\eeq
Again, by Lemma \ref{lem:HansonWright}, with probability $1-o(1)$,
\beq
\mathbf{a}^T Q\mathbf{a}-\mathbb{E}\mathbf{a}^T Q\mathbf{a}=O\left(\frac{\nu_n}{n}\| Q\|_{\HS} \right),
\eeq
where
\beq\label{eq:Qnorm_bound}\begin{split}
\| Q\|_{\HS}^2&=Q_{11}^2+\sum_{i=2}^{n-2}Q_{ii}^2+2\sum_{i=2}^{n-2}Q_{i1}^2
<\left(\omega_3^2\sum_{i=2}^{n-2} (g_{i+2}-1)\right)^2
+3\sum_{i=2}^{n-2}(g_{i+2}-1)^2\\
&< C\left(\frac{1}{n^2} \sum_{i=2}^{n-2} \frac{1}{1-\omega_{i+2}}\right)^2
+\sum_{i=2}^{n-2}\frac{1}{(1-\omega_{i+2})^2}.
\end{split}\eeq
The first sum satisfies
\beq\label{eq:sum_gi_bound}\begin{split}
\left(\frac{1}{n^2} \sum_{i=2}^{n-2} \frac{1}{1-\omega_{i+2}}\right)^2
=O\left(\left(\frac{1}{n^2}\cdot n \cdot n^{1/3}\sigma_n^{-1/2}\right)^2\right)
=o(1),
\end{split}\eeq
while the second one is 
\beq\begin{split}
\sum_{i=2}^{n-2} \frac{1}{(1-\omega_{i+2})^2}
&= \sum_{i=2}^{n-n^{1/3}\sigma_n} \frac{1}{(1-\omega_{i+2})^2}+ \sum_{n-n^{1/3}\sigma_n}^{n-2} \frac{1}{(1-\omega_{i+2})^2}\\
&=O\left(n\int_{n^{-2/3}\sigma_n}^1x^{-1}dx+ n^{1/3}\sigma_n\cdot n^{2/3}\sigma_n^{-1}\right)
=O(n\log n).
\end{split}\eeq
We conclude
\beq
\mathbf{a}^T Q\mathbf{a}-\mathbb{E}\mathbf{a}^T Q\mathbf{a}=O\left(\frac{\nu_n}{n}\sqrt{n\log n}\right)=O\left(n^{-1/2}\nu_n\sqrt{\log n}\right).
\eeq
It remains to evaluate the expectation.
\beq\begin{split}
\mathbb{E}\mathbf{a}^T Q\mathbf{a}
&=\mathbb{E}\sum_{i=4}^n(g_i-1)(\a_{i-1}-\omega_3\cdots\omega_{i-1}\a_2)^2
=\sum_{i=4}^n(g_i-1)\mathbb{E}(\a_{i-1}^2+\omega_3^2\cdots\omega_{i-1}^2\a_2^2).
\end{split}\eeq
We note that $\mathbb{E}(\a_{i-1}^2+\omega_3^2\cdots\omega_{i-1}^2\a_2^2)=O(n^{-1})$ and, in the course of the proof above (see \eqref{eq:Qnorm_bound} and \eqref{eq:sum_gi_bound}), we showed that $\frac1n \sum_{i=4}^n(g_i-1)=O(1).$  Therefore, $P_3=O(1)$ with probability $1-o(1)$.

\end{appendix}

\pagebreak

\end{document}